\documentclass[10pt]{article}
\usepackage{amsbsy}
\usepackage{amsmath,amsfonts,amssymb,amsthm,mathrsfs,mathtools,stmaryrd}
\usepackage[a4paper,vmargin={3.5cm,3.5cm},hmargin={2.5cm,2.5cm}]{geometry}
\usepackage[font=sf, labelfont={sf,bf}, margin=1cm]{caption}
\usepackage{graphicx,graphics,xcolor}
\usepackage{epsfig}
\usepackage{latexsym}
\usepackage[applemac]{inputenc}
\usepackage{ae,aecompl}
\usepackage{soul}
\usepackage[english]{babel}
\usepackage[pdfpagemode=UseNone,bookmarksopen=false,colorlinks=true,urlcolor=blue,citecolor=blue,citebordercolor=blue,linkcolor=blue]{hyperref}
\usepackage{tikz,mathtools}
\usepackage{enumerate}

\usepackage{bbm,tikz-cd,ulem}
\makeatletter
\newcommand{\fixed@sra}{$\vrule height 2\fontdimen22\textfont2 width 0pt\shortuparrow$}
\newcommand{\newsearrow}{%
  {\!\mathrel{\text{\rotatebox[origin=c]{\numexpr 5*45}{\fixed@sra}}\!}}
}
\makeatother

\newcommand{\map}{\mathfrak{m}}
\newcommand{\emap}{\mathfrak{e}}

\newcommand{\rootface}{f_{\mathrm{r}}}

\def\bc{\widehat{\beta}}
\def\gc{\widehat{\gamma}}

\newcommand{\faces}{\mathsf{Faces}}

\newcommand{\Z}{\mathbb{Z}}
\renewcommand{\P}{\mathbb{P}}
\newcommand{\E}{\mathbb{E}}

\def\U{\mathbb{U}}
\def\N{\mathbb{N}}
\def\e{{\rm e}}
\def\d{{\rm d}}

\newcommand{\R}{\mathbb{R}}

\renewcommand\Pr[1]{\mathbb{P}\left(#1\right)}

\newtheorem{theorem}{Theorem}[section]

\newtheorem{proposition}[theorem]{Proposition}
\newtheorem{lemma}[theorem]{Lemma}
\newtheorem{corollary}[theorem]{Corollary}

\theoremstyle{definition}
\newtheorem{remark}[theorem]{Remark}

\def\build#1_#2^#3{\mathrel{
\mathop{\kern 0pt#1}\limits_{#2}^{#3}}}

\title{\vspace{-2cm}\textsc{Martingales in self-similar growth-fragmentations and their\\
 connections with random planar maps } \vspace{0.5cm}}

\date{}

\DeclareSymbolFont{extraup}{U}{zavm}{m}{n}
\DeclareMathSymbol{\varheart}{\mathalpha}{extraup}{86}
\DeclareMathSymbol{\vardiamond}{\mathalpha}{extraup}{87}

\makeatletter
\renewcommand*{\@fnsymbol}[1]{\ensuremath{\ifcase#1\or \spadesuit \or \varheart \or \vardiamond \or \clubsuit\or \or
   \mathsection\or \mathparagraph\or \|\or **\or \dagger\dagger
   \or \ddagger\ddagger \else\@ctrerr\fi}}
\makeatother

\author{Jean Bertoin \thanks{Universit\"at Z\"urich.\hfill  \texttt{jean.bertoin@math.uzh.ch}} \qquad  Timothy Budd \thanks{Niels Bohr Institute, University of Copenhagen \& IPhT, CEA, Universit\'e Paris-Saclay \hfill \texttt{timothy.budd@ipht.fr}} \qquad  Nicolas Curien \thanks{Universit\'e Paris-Sud.\hfill  \texttt{nicolas.curien@gmail.com}}  \qquad Igor Kortchemski\thanks{CNRS \& CMAP, \'Ecole polytechnique. \hfill \texttt{igor.kortchemski@normalesup.org}}}

\begin{document}
\normalem
\maketitle

\begin{abstract}
The purpose of the present work is twofold. 
First, we develop the theory of general self-similar growth-fragmentation processes by focusing on martingales which appear naturally in this setting and  by recasting classical results for branching random walks in this framework. In particular, we establish  many-to-one formulas  for growth-fragmentations and define the notion of intrinsic area of a growth-fragmentation. Second,  we identify a distinguished family of growth-fragmentations closely related to stable L\'evy processes, which are then shown to arise as the scaling limit of the perimeter process in Markovian explorations of certain random planar maps with large degrees (which are, roughly speaking, the dual maps of the stable maps of Le Gall \& Miermont \cite{LGM09}). As a consequence of this result, we are able to identify the law of the intrinsic area of these distinguished growth-fragmentations. This generalizes a geometric connection between large Boltzmann triangulations and a certain growth-fragmentation process, which was established in \cite{BCK}. 
\end{abstract}

\begin{figure}[!h]
 \begin{center}
 \includegraphics[height=7.5cm]{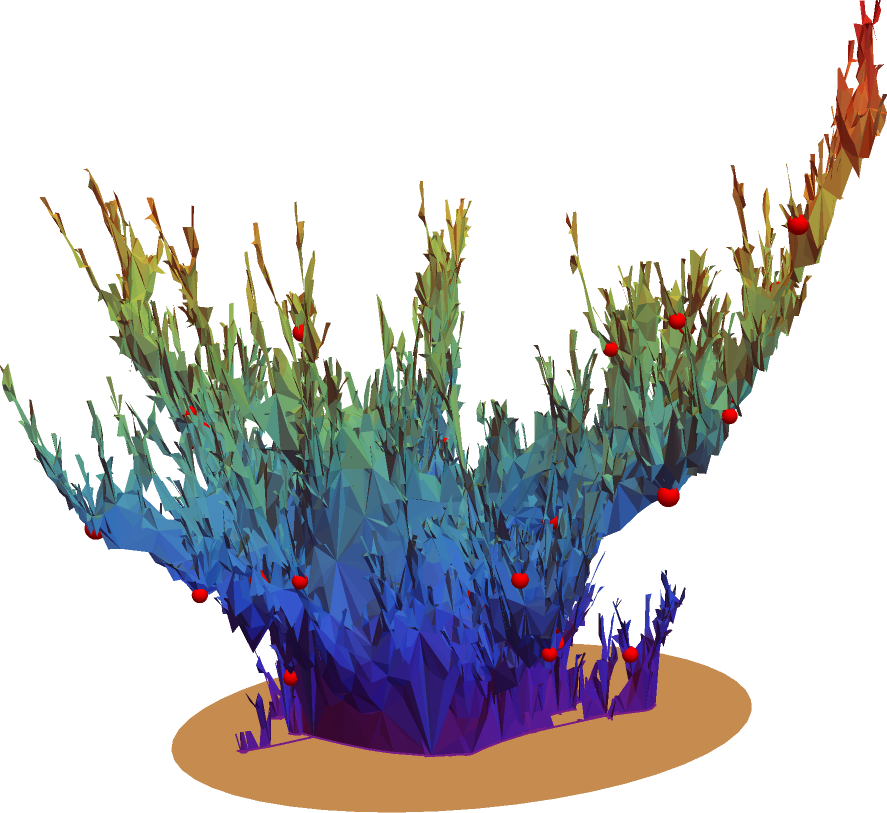}
  \includegraphics[height=7.5cm]{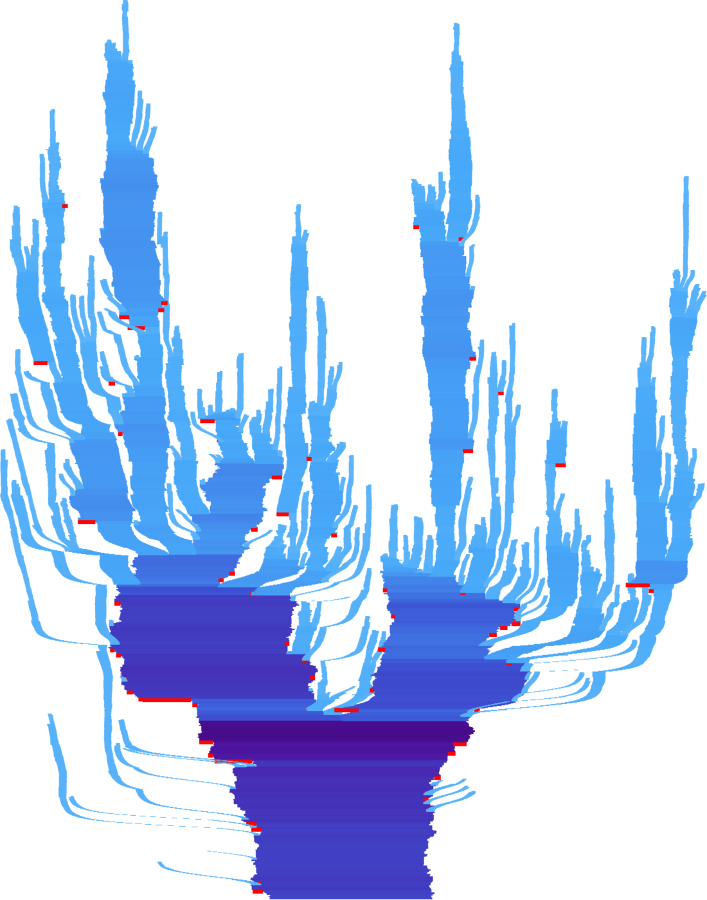}
 \caption{Left: A cactus representation of a random planar map with certain vertices of high degrees (which are the red dots), where the height of a vertex is its distance to the orange boundary. Right: A simulation of the growth-fragmentation process describing the scaling limit of its perimeters at heights (the red part corresponds to positive jumps of the process).}
 \end{center}
 \end{figure}

\section{Introduction}

\subsection{Motivation and overview}

The realm of scaling limit theorems for discrete models of  random planar geometry has attracted a considerable interest in the last two decades.  The initial purpose for this work is to deepen a  remarkable connection that has been pointed at recently in \cite{BCK} between certain random planar maps and a specific self-similar growth-fragmentation  (roughly speaking, a growth-fragmentation describes the evolution of masses of particles which can grow, melt and split as time passes; see below for details). We also refer to \cite{LG17,MS15} for related results

The main theorem in \cite{BCK} states an invariance principle for the rescaled process of the collection of  lengths of the cycles obtained by slicing at heights a random Boltzmann triangulation with a simple boundary whose size tends to infinity. The limit is given by a self-similar growth-fragmentation related to the completely asymmetric stable L\'evy process with index $3/2$.  
 In this work, we identify a one-parameter family of stable L\'evy processes with index $\theta\in(1,3/2)$,  
 which are related to a remarkable family of self-similar
 growth-fragmentations  $\left (\mathbf{{X}}^{(1-\theta)}_{\theta}\right )_{1 < \theta < 3/2}$. In turn the latter are shown to be the scaling limits  of lengths of the cycles obtained by slicing at heights random planar maps with large degrees.

 This can be seen as a first step to the construction of a family of random surfaces that generalize the Brownian map. Indeed,  Bettinelli and Miermont \cite{BetMie} have established that the free Brownian disk arises as the scaling limit of Boltzmann
quadrangulations with a boundary, which strongly suggests that the free Brownian disk  could also be constructed from the self-similar growth-fragmentation appearing in \cite{BCK}. 
We believe that the growth-fragmentation $\mathbf{{X}}^{(1-\theta)}_{\theta}$ should code similarly a so-called \emph{free stable disk}.  Further, just as the  Brownian map can be obtained as the limit of  Brownian disks when the size of the boundary tends to $0$, we believe that stable disks may be described by the limit of the growth-fragmentation $\mathbf{{X}}^{(1-\theta)}_{\theta}$ started from a single particle with size tending to $0$. 
 In particular, this would lead to examples of natural continuum random surface models with other scaling exponents than those of the Brownian map. 
Roughly speaking, when scaling distances by a factor $C$ in the Brownian map, the area measure is multiplied by $C^{4}$ while lengths of outer boundaries of metric balls are multiplied by $C^{2}$. In the case of a stable sphere with $1<\theta<3/2$, we expect that $C^{4}$ will be replaced with $C^{ (\theta+1/2)/(\theta-1)}$
 and $C^{2}$ with   $C^{ 1/(\theta-1)}$.

 The Brownian map, disk, or plane, are random metric spaces which come equipped with an intrinsic volume measure (here we shall rather call this an area measure, as it is related to a planar object). One can therefore naturally ask:  how can this area measure be recovered from this growth-fragmentation?  The naive answer that, just as for smooth surfaces, it should be simply derived from the product of the length of the cycles and the height, turns out to be wrong: the latter yields a measure that is not even locally finite.
One of the contributions of this work is to point out that  the area measure can be constructed for fairly general self-similar growth-fragmentations by relying on their branching structure, and more specifically on a so-called intrinsic martingale. This is an object of fundamental importance in the theory of branching random walks, see e.g. the lecture notes by Z. Shi \cite{ZShi} and references therein. We will actually see that the branching structure naturally yields a pair of martingales. In the case $\theta=3/2$, the first one is related to the area measure of the free Brownian disk, while the second is related via a change of probability to the Brownian plane, another fundamental random metric space. 
 In general, it seems to be an interesting and difficult question to identify the law of the terminal value of an intrinsic martingale. In this work, we establish  the  remarkable fact that total area of the distinguished growth-fragmentation $\mathbf{{X}}^{(1-\theta)}_{\theta}$ mentioned above is an inverse-size-biased stable distribution, by using the fact that $\mathbf{{X}}^{(1-\theta)}_{\theta}$ arises as a scaling limit of discrete objects (we have so far not been able to establish directly this fact without using random planar maps).

Let us now describe more precisely the content of this work.

\subsection{Description of the main results}

 Our approach to self-similar growth-fragmentations applies quite generally, without any {\it a priori} connection to  planar random geometry, and will thus first be developed in a general setting.  This then enables us to  identify a remarkable  one-parameter family of self-similar growth-fragmentations related to stable and hypergeometric L\'evy processes. We shall finally establish a geometric connection between this family and the stable maps of Le Gall and Miermont \cite{LGM09}, that generalizes the one in \cite{BCK} alluded to above.

\paragraph{Self-similar growth-fragmentations.} Markovian growth-fragmentation processes and cell systems have been introduced in \cite{BeGF}
to model  branching systems of cells where, roughly speaking, sizes of cells may vary as time passes and then suddenly divide into a mother cell and a daughter cell.  The  sum of the sizes of the mother and of its daughter immediately after a division event always  equals the size of the mother cell immediately before division.
Further, daughter cells evolve independently one of the others, and follow stochastically the same dynamics as the mother; in particular, they give birth in turn to granddaughters, and so on. Cell systems focus on the genealogical structure of cells, and a growth-fragmentation is then simply obtained as the  process of the family of the sizes of cells observed at a given time.
  We stress that  division events may occur instantaneously, in the sense that on every arbitrarily small time interval, a cell may generate infinitely many daughter cells who can then have arbitrarily small sizes. Division events thus correspond to negative jumps of the mother cell, and even though it was natural in the setting of  \cite{BeGF} to assume that the process describing the size of a typical cell  had no positive jumps, the applications to random
maps that we have in mind incite us to consider here more generally processes which may have jumps of both signs. Only the negative jumps correspond to division events, whereas the possible positive jumps play no role in the genealogy. The latter are only part of the evolution of processes and may be interpreted as a sudden macroscopic growth.

In this work, we  only consider self-similar growth-fragmentations, and to simplify we will write \emph{growth-fragmentation} instead of \emph{self-similar growth-fragmentation}. To start with, we shall recall the construction of a cell system from a positive self-similar Markov process (Sec.~\ref{sec:cellsystem}). 

\paragraph{Intrinsic martingales.}  An important feature is that
the point process of the logarithm of the sizes of cells at birth at a given generation forms a branching random walk. This yields a pair of genealogical martingales,  $({\mathcal M}^+(n), n \geq 0)$ and $({\mathcal M}^-(n), n \geq 0)$, which arise naturally as  intrinsic  martingales associated with that branching random walk (Sec.~\ref{sec:genealogical}). Using classical results of Biggins \cite{Biggins}, we observe that 
$ {\mathcal M}^+$ converges to $0$ a.s. whereas  ${\mathcal M}^-$ is uniformly integrable; the terminal value  of the latter is  interpreted as an intrinsic area.

We then introduce the growth-fragmentation as the process of the sizes of the cells at a given time (Sec.~\ref{sec:ssgf}). We show that its intensity measure can be expressed in terms of the distribution of an associated positive self-similar Markov process via a many-to-one formula (Theorem \ref{T1} in Sec.~\ref{sec:intensity}). Using properties of positive self-similar Markov processes, and in particular the fundamental connection with L\'evy processes due to Lamperti, we then arrive at a pair of (super-)martingales $M^+(t)$ and $M^-(t)$ indexed by continuous time (Sec.~\ref{sec:temporal}) and which are naturally related  to the two discrete parameter martingales, ${\mathcal M}^+(n)$ and ${\mathcal M}^-(n)$.

Our main purpose in Sect.~\ref{sec:spine1} is to describe explicitly the dynamics of growth-fragmentations under the probability measures which are obtained by tilting the initial one with these intrinsic martingales. Using the well-known spinal decomposition for branching random walks, we show (Theorems \ref{T2} and \ref{T9}) that the latter can be depicted by a modified cell system, in which, roughly speaking, the evolution of all the cells is governed by the same  positive self-similar Markov process, except for the Eve cell that follows a different self-similar Markov process (which, for a negative self-similarity parameter, survives forever in the case of ${\mathcal M}^+$ and is continuously absorbed in $0$ in the case of ${\mathcal M}^-$).
 We believe that the intrinsic martingales, many-to-one formulas, spinal decompositions ... which are developed here should have applications in the study of fine properties of self-similar growth-fragmentations (see e.g.~\cite{QShi,Dad17}).

\paragraph{A distinguished one-parameter family of growth-fragmentations.}

Once this is done, in Sec.~\ref{sec:gen} we give different ways to identify the law of a self-similar growth-fragmentation  through its cumulant function (Theorem \ref{thm:kappaphi}). Indeed, by \cite[Theorem 1.2]{QShi}, the law of a self-similar growth-fragmentation is characterized by a pair $(\kappa,\alpha)$, where $\kappa$ is the so-called cumulant function (see Eq.\,\eqref{eqkappa}) and $\alpha$ is the self-similarity parameter. We obtain in particular that the  law of a self-similar growth-fragmentation  is characterized by the distribution of the process describing the evolution of the  Eve cell under the modified probability measure obtained by tilting the initial one with either ${\mathcal M}^+$ or ${\mathcal M}^{-}$, which may be  a result of independent result.

Using this observation, we exhibit a distinguished family of growth-fragmentations which are closely related to $\theta$-stable L\'evy processes for $\theta \in ( \frac{1}{2}, \frac{3}{2}]$. They are described by a one-parameter family of  cumulant functions $\left(\kappa_{\theta}\right)_{ 1/2 < \theta \leq 3/2}$ given by
$$\kappa_{\theta}(q) = \frac{\cos(\pi(q-\theta)) }{\sin(\pi(q-2\theta))} \cdot \frac{\Gamma (q-\theta)}{\Gamma (q-2\theta)}, \qquad \theta<q<2\theta+1.$$
The growth-fragmentation with cumulant function $\kappa_{\theta}$ and self-similarity parameter $-\theta$ has the property that the evolution of the Eve cell obtained by tilting the dynamics by the martingale ${\mathcal M}^+$ (resp.~${\mathcal M}^-$) is  the $\theta$-stable L\'evy process with positivity parameter $\rho$ satisfying 
$$ \theta\cdot (1- \rho)  = \frac{1}{2},$$ 
killed when entering $(-\infty,0]$ and conditioned to survive forever (resp.~to be absorbed continuously at $0$). 
In the special case $\theta= {3}/{2}$, we recover the growth-fragmentation process without positive jumps  that appears in \cite{BCK}. This also gives an alternative way to see that $\theta=3/2$ plays a special role in the work of Miller \& Sheffield \cite{MS15} (see Remark \ref{rem:MS}).

\paragraph{Connections with random planar maps.}
The final part of the paper (Sec.~\ref{sec:maps}) establishes a connection between the  distinguished  growth-fragmentations with cumulant function $\kappa_{\theta}$ (for $\frac{1}{2} <\theta<\frac{3}{2}$) and a family of random planar maps that we now describe. Let  $ \mathbf{q}= (q_{k})_{ k \geq 1}$  be a non-zero sequence of non-negative numbers. We define a measure ${\tt w}$ on the set of all rooted bipartite planar maps by the formula
$${\tt w}(\map) := \prod_{f\in\faces(\map)} q_{\deg(f)/2},$$
where $\mathsf{Faces}( \map)$ denotes the set of all faces of a rooted bipartite planar map $\map$ and $\deg(f)$ is the degree of a face $f$  (see Sec.~\ref{sec:boltzmann} for precise definitions). Following \cite{LGM09}, we assume that $ \mathbf{q}$ is admissible, critical, non-generic and satisfies 
$$ q_{k} \underset{k \to \infty}{\sim} c\cdot \gamma^{k-1} \cdot k^{-\theta-1}$$ for certain $c, \gamma >0$ and $\theta \in ( \frac{1}{2}, \frac{3}{2})$. We then denote  by $B^{(\ell)}$ a random bipartite planar map  chosen proportionally to ${\tt w}$ and conditioned to have a root face of degree $2\ell$ (the root face is the face incident to the right of the root edge) and write $B^{(\ell),\dagger}$ for the \emph{dual} map of $B^{(\ell)}$ 
{(the vertices of $ B^{(\ell),\dagger}$ correspond to faces of $B^{(\ell)}$, and two vertices of $ B^{(\ell),\dagger}$ are adjacent if the corresponding faces are adjacent in  $B^{(\ell)}$)}. Our assumption implies, roughly speaking, that $B^{(\ell), \dagger}$ has vertices of large degree. The geometry of these maps has recently been analyzed in \cite{BCgrowth}.

We establish  that the growth-fragmentations with cumulant function $\kappa_{\theta}$ appear as the scaling limit  of perimeter processes in Markovian explorations (in the sense of \cite{Bud15}) in $ B^{(\ell), \dagger}$. In the regime $1<\theta<3/2$ (the so-called dilute phase) this connection takes a more geometrical form and we prove that the growth-fragmentation with cumulant function $\kappa_{\theta}$ and self-similarity parameter $1-\theta$, which we denote by $ \mathbf{X}_{\theta}^{(1-\theta)}$, describes the scaling limit of the perimeters of cycles obtained by slicing at all heights the map $B^{(\ell),\dagger}$ as $\ell \rightarrow \infty$ (Theorem \ref{thm:GF}).  This extends \cite{BCK}, where this was shown in the case of random triangulations with the growth-fragmentation $ \mathbf{X}_{ 3/2}^{(-1/2)}$.

As an application, we identify the law of the intrinsic area of the growth-fragmentation $ \mathbf{X}_{\theta}^{(1-\theta)}$ as a biased stable distribution (Corollary \ref{cor:area}). We also emphasize that the general study of  growth-fragmentations, and in particular the spinal decomposition, are used in an essential way to establish this result (see the proof of Proposition \ref{prop:scalingll}). The approach we follow is simpler and more general than the one in \cite{BCK}.

\bigskip

In the dilute case $\theta \in ( 1, \frac{3}{2}]$, as previously mentioned, we believe that these growth-fragmentations describe the breadth-first search of a one-parameter family of continuum random surfaces which includes the Brownian Map; these random surfaces should be the scaling limit of the dual maps of large random planar maps sampled according to their ${\tt w}$-weight and conditioned to be large (see \cite{BCgrowth} for more details about the geometry of these maps). These questions will be addressed in a future work.

\paragraph{Acknowledgments.} NC and IK acknowledge partial support from Agence Nationale de la Recherche, grant number ANR-14-CE25-0014 (ANR GRAAL), ANR-15-CE40-0013 (ANR Liouville) and from the City of Paris, grant ``Emergences Paris 2013, Combinatoire \`a Paris''. 
TB acknowledges support from the ERC-Advance grant 291092, ``Exploring the Quantum Universe'' (EQU). Finally, we would like to thank two anonymous referees for  useful comments.

\tableofcontents

\section{Cell systems and genealogical martingales}  
In this section, we start by recalling the construction of a cell system and then dwell on properties which will be useful in this work. As it was mentioned in the Introduction, we shall actually work with a slightly more general setting than in \cite{BeGF}, allowing cell processes to have positive jumps. It can be easily checked that the proofs of results from  \cite{BeGF} that we shall need here and which  were established under the assumption of absence of positive jumps, work just as well when positive jumps are allowed. 

The building block for the construction  consists in a positive self-similar Markov process $X=(X(t))_{t\geq 0}$, which either is absorbed  after a finite time at some cemetery point $\partial$ added to the positive half-line $(0,\infty)$, 
  or converges to $0$ as $t\to \infty$. We first recall the classical representation due to Lamperti \cite{Lam72}, which enables us to  view $X$ as the exponential of a  L\'evy process up to a certain time-substitution. We then construct cell systems based on a self-similar Markov process and derive some consequences which will be useful to our future  analysis. In particular, we point at a pair of intrinsic martingales which play a crucial role in our study.

\subsection{Self-similar Markov processes and their jumps}

\label{sec:lamperti}
Consider a quadruple $(\sigma^2,b,\Lambda,{\tt k})$, 
where $\sigma^2\geq 0$, $b\in\R$, ${\tt k}\geq 0$  and $\Lambda$ is a measure on $\R$ such that 
$\int(1\wedge y^2)\Lambda(\d y)<\infty$ and\footnote{The assumption $\int_{y>1}\e^{y}\Lambda(\d y)<\infty$ may be replaced by the weaker assumption that $\int_{y>1}\e^{ qy}\Lambda(\d y)<\infty$ for a certain $q>0$, but then a cutoff should be added to $q(1-\e^y)$, such as e.g.~$q(1-\e^y) \mathbbm{1}_{y \leq 1}$ in  \eqref{eqpsi}.} $\int_{y>1}\e^{y}\Lambda(\d y)<\infty$. Plainly, the second integrability requirement is always fulfilled when the support of $\Lambda$ is bounded from above, and in particular when $\Lambda$ is carried on $\R_-$.

The formula
\begin{equation}\label{eqpsi}
\Psi(q) \coloneqq -{\tt k}+ \frac{1}{2}\sigma^2q^2 + bq +\int_{\R}\left( \e^{qy}-1+q(1-\e^y)\right) \Lambda(\d y)\,,\qquad q\geq 0,
\end{equation}
is a slight variation of the  L\'evy-Khintchin formula; it
defines a convex function with values in $(-\infty, \infty]$ which we view as the Laplace exponent of 
a real-valued L\'evy process  $\xi=(\xi(t), t\geq0)$, 
where the latter is  killed at rate ${\tt k}$ when ${\tt k}>0$. Specifically, we have
\begin{equation}\label{EqLap}
E(\exp(q\xi(t)))=\exp(t\Psi(q))\qquad \hbox{for all $t,q\geq 0$}, 
\end{equation}
with the convention that $\exp(q\xi(t))=0$ when $\xi$ has been killed before time $t$ (we may think that $-\infty$ serves as cemetery point for $\xi$).  
We furthermore assume that either the killing rate ${\tt k}$ of $\xi$ is positive, or that  $\Psi'(0+)\in[-\infty,0)$; that is,  equivalently, that $\Psi$ remains strictly negative in some right-neighborhood of $0$. Recall that this is also the necessary and sufficient condition for $\xi$ either to have
 a finite lifetime, or  to drift to $-\infty$ in the sense that $\lim_{t\to \infty} \xi(t)=-\infty$ a.s.
 The case $\Lambda((-\infty,0))= 0$ when the process $X$ has no negative jumps until it dies, will be uninteresting for our purposes and thus implicitly excluded from now on.
 
Next, following Lamperti \cite{Lam72},  we fix some $\alpha\in\R$ and  define 
\begin{equation}\label{Eqcht}
\tau_t\,\coloneqq\,\inf\left\{r\geq0: \int_{0}^{r}\exp(-\alpha \xi(s))\d s\geq
t\right\}\,,\qquad t\geq 0.
\end{equation}
For every $x>0$, we write $P_x$ for the distribution of the time-changed process
$$X(t)\,:=\,x\exp\left\{\xi\left({\tau_{tx^{\alpha}}}\right)\right\}\,,\qquad
t\geq0,$$
with the convention that $X(t)=\partial$ for $t\geq \zeta\coloneqq x^{-\alpha}\int_0^{\infty}\exp(-\alpha \xi(s))\d s$. 
Then $X=(X(t), t\geq 0)$ is both (sub-)Markovian and self-similar,  in the sense that for every $x>0$,
\begin{equation}\label{scale}
\hbox{ the law of $\left(xX(x^{\alpha}t), t\geq0\right)$ under
${P}_1$ is ${P}_{x}$.}
\end{equation}
We shall refer here to $X$ as the self-similar Markov process with characteristics $(\sigma^2,b,\Lambda, {\tt k}, \alpha)$,
or simply $(\Psi, \alpha)$.  
We stress that in our setting, either $X$ is absorbed at the cemetery point $\partial$ after a finite time, or it converges to $0$ as time goes to infinity.

The {\it negative} jumps of $X$ will have an important role in this work, 
and it will be convenient to adopt throughout this text the notation 
$$\Delta_{\hbox{-}} Y(t)\coloneqq \left\{
\begin{matrix} Y(t)-Y(t-) &\hbox{ if } Y(t)<Y(t-),\\
0 &\hbox{ otherwise}&
\end{matrix}\right.$$
for every c\`adl\`ag real-valued process $Y$. 
 We next set for $q\geq0$
\begin{eqnarray}\label{eqkappa}
\kappa(q)&\coloneqq &\Psi(q)+\int_{(-\infty,0)}(1-\e^{y})^q\Lambda(\d y) \nonumber \\
&=& -{\tt k}+ \frac{1}{2}\sigma^2q^2 + bq +\int_{\R}\left( \e^{qy}-1+q(1-\e^y)+ {\mathbbm{1}}_{\{y<0\}}(1-\e^{y})^q\right) \Lambda(\d y).
\end{eqnarray}
Plainly, $\kappa: \R_+ \to (-\infty, \infty]$ is a convex function; observe also $\kappa(q)<\infty$ if and only if both $\Psi(q)<\infty$ and $\int_{(-\infty,0)}(1-\e^{y})^q\Lambda(\d y)<\infty$, and further that  the condition $\int_{(-\infty,0)}(1\wedge y^2)\Lambda(\d y)<\infty$ entails that $\int_{(-\infty,0)}(1-\e^{y})^q\Lambda(\d y)<\infty$ for every $q\geq 2$. 
 We call $\kappa$ the {\em cumulant function}; it plays a major role in the study of self-similar growth-fragmentations through the following calculation done in  \cite[Lemma 4]{BeGF}:
  \begin{eqnarray}\label{eq:berGFlem4} E_x\left(\sum_{0<s<\zeta}|\Delta_{\hbox{-}} X(s)|^{q}\right)= \left\{ \begin{array}{cc}  \displaystyle x^q\left(1-\frac{\kappa(q)}{\Psi(q)}\right) & \mbox{ if } \Psi(q)<0 \mbox{ and }\kappa (q) < \infty\\ \infty & \mbox{ otherwise. } \end{array}\right.\end{eqnarray}
Actually,  only the case when $X$ has no positive jumps is considered in \cite{BeGF}, however the arguments  there works just as well when $X$ has also positive jumps.

By convexity, the function $\kappa$ has at most two roots. We shall assume throughout this work
that there exists $\omega_+>0$ such that $\kappa(q)<\infty$ for all $q$ in some neighborhood of $\omega_+$ and 
\begin{equation} \label{omega+}
\kappa(\omega_{+})=0 \quad \hbox{and}\quad \kappa'(\omega_{+})>0.
\end{equation}
This forces
\begin{equation}\label{condkappa}
\inf_{q\geq 0} \kappa(q)<0,
\end{equation}
and since $\Psi\leq \kappa$, \eqref{condkappa}  is a stronger requirement than $\Psi(0)=-{\tt k}<0$ or $\Psi'(0+)<0$ that we previously made. Further, note that in the case when $\Psi(q)<\infty$ for all $q>0$ (which holds for instance whenever
the support of the L\'evy measure $\Lambda$ is bounded from above) and the L\'evy process $\xi$ is not the negative of a subordinator, then the Laplace exponent $\Psi$ is ultimately increasing, so $\lim_{q\to \infty} \kappa(q)=\infty$, and \eqref{condkappa} ensures \eqref{omega+}. See \cite{BS16} for a study of the case where $\kappa(q)>0$ for every $q \geq 0$.

We record for future use the following elementary facts (see Lemma 3.1 in \cite{BeWa} for closely related calculations) about the cumulant function shifted at its root $\omega_+$, namely
$$\Phi^+(q)\coloneqq \kappa(q+\omega_{+}), \qquad q\geq 0.$$

\begin{lemma}\label{Lphi+} \begin{enumerate}
\item[(i)] The function $\Phi^+$ can be expressed in a L\'evy-Khintchin form similar to \eqref{eqpsi}. More precisely,   the killing rate is $0$, the Gaussian coefficient $\frac{1}{2}\sigma^2$, and the L\'evy measure $\Pi$ given by 
$$\Pi(\d x)\coloneqq \e^{x\omega_+}\left(\Lambda(\d x) + \tilde \Lambda(\d x)\right)\,,\qquad x\in \R,$$
where $\tilde \Lambda(\d x)$ denotes the image of ${\mathbbm{1}}_{x<0}\Lambda(\d x)$ by the map $x\mapsto \ln(1-\e^x)$. 
\item[(ii)] Therefore $\Phi^+$ is the Laplace exponent of a L\'evy process, say $\eta^+=(\eta^+(t), t\geq 0)$. Further $\Phi^+(0)=0$, $(\Phi^+)'(0)=\kappa'(\omega_+)>0$,  and thus $\eta^+$ drifts to $+\infty$.
\end{enumerate}
\end{lemma}

We shall also sometimes consider the case when the equation $\kappa(q)=0$ possesses a second solution, which we shall then denote by $\omega_-$.
More precisely, we will say that Cram\'er's hypothesis holds when
\begin{equation} \label{omega-}
\hbox{ there exists $\omega_-< \omega_+$  such that } \kappa(\omega_{-})=0 \quad \hbox{and}\quad \kappa'(\omega_{-})>-\infty
\end{equation}
(note that $\kappa'(\omega_{-})<0$  by convexity of $\kappa$).

We are now conclude this section by introducing a first noticeable martingale, which is a close relative to the more important martingales to be introduced later on, and already points at the central role of the cumulant  function $\kappa$ and its roots. 

\begin{proposition} \label{P0}
Let $\omega$ be any root of the equation $\kappa(\omega)=0$. Then for every $x>0$, the process
$$X(t)^{\omega}+\sum_{0<s\leq t, s<\zeta}|\Delta_{\hbox{-}} X(s)|^{\omega}\,, \qquad t\geq 0,$$
(with the usual implicit convention that $X(t)^{\omega}=0$ whenever $t\geq \zeta$) is a uniformly integrable martingale under $P_x$, with terminal value $\sum_{0<s<\zeta}|\Delta_{\hbox{-}} X(s)|^{\omega}$.
\end{proposition}

\begin{proof}For $q=\omega$ a root of $\kappa$, the right-hand side  of \eqref{eq:berGFlem4} reduces to $x^{\omega}$, and an application of the Markov property at time $t$ yields our claim.
\end{proof}

\subsection{Cell systems and branching random walks} 
\label{sec:cellsystem}

We next introduce the notion of cell system and related canonical notation. We use the Ulam tree ${\mathbb U}=\bigcup_{n\geq 0}\N^n$  where $ \mathbb{N}= \{1,2, \ldots\}$ to encode the genealogy of a family of cells which evolve and split as time passes. We define a cell system as a family ${\mathcal X}\coloneqq ({\mathcal X}_u, u\in\U)$, where each ${\mathcal X}_u=\left({\mathcal X}_u(t)\right)_{t\geq 0}$ should be thought of as the size of the cell
labelled by $u$ as a function of its {\it age} $t$. The system  also implicitly encodes the birth-times $b_u$ of those cells. 

Specifically, each ${\mathcal X}_u$ is a c\`adl\`ag trajectory with values in $(0,\infty)\cup\{\partial\}$, 
which fulfills the following properties:
\begin{itemize}
\item  $\partial$ is an absorbing state, that is ${\mathcal X}_u(t)=\partial$ for all $t\geq \zeta_u\coloneqq \inf\{t\geq 0:{\mathcal X}_u(t)=\partial\}$,
\item either $\zeta_u<\infty$ or $\lim_{t\to\infty} {\mathcal X}_u(t)=0$. 
\end{itemize}

We should think of $\zeta_u$ as the lifetime of the cell $u$, and stress that $\zeta_u=0$ (that is
${\mathcal X}_u(t)\equiv \partial$) and $\zeta_u=\infty$ (that is $ {\mathcal X}_u(t)\in (0,\infty)$ for all $t\geq 0$) are both allowed. The {\it negative} jumps of cells will play a specific part in this work. 
The second condition above ensures that  for every given $\varepsilon >0$, the process ${\mathcal X}_u$ has at most finitely many negative  jumps of absolute sizes greater than $\varepsilon$. This enables us to enumerate the sequence of the positive jump  sizes and times of $-{\mathcal X}_{u}$ in the decreasing lexicographic order, say 
 $(x_1,\beta_1), (x_2,\beta_2), \ldots$, that is either  $x_i=x_{i+1}$ and then $\beta_i>\beta_{i+1}$, or $x_i>x_{i+1}$. In the case when ${\mathcal X}_{u}$ has only a finite number of negative jumps, say $n$, then we agree that $x_i=\partial$ and $\beta_i=\infty$ for all $i>n$. The third condition that we impose on cell systems, is that the sequence
 of negative  jump sizes and times of a cell $u$ encodes the birth-times and sizes at birth of its children $\{ui:i\in\N\}$. That is:
 \begin{itemize}
 \item for every $i\in\N$, the birth time $b_{ui}$ of the cell $ui$ is given by $b_{ui}= b_u+\beta_i$, and ${\mathcal X}_{ui}(0)=x_i$.
 \end{itemize} 
In words, we interpret the negative jumps of ${\mathcal X}_{u}$ as birth events of the cell system, each jump of size $-x<0$ corresponding to the birth of a new cell with initial size $x$, and daughter  cells are enumerated in the decreasing order of the sizes at birth. Note that if ${\mathcal X}_{u}$ has only finitely many negative jumps, says $n$, then
${\mathcal X}_{ui}\equiv \partial$ for all $i>n$. 

We now introduce for every $x>0$ a probability distribution for cell systems, denoted by ${\mathcal P}_x$,  which is 
of Crump-Mode-Jagers  type and can be described recursively as follows.  The Eve cell,  ${\mathcal X}_{\varnothing}$, has the law $P_x$ of the self-similar Markov process $X$ with characteristics $(\Psi,\alpha)$.  
Given ${\mathcal X}_{\varnothing}$, the processes of the sizes  of cells at the first generation, ${\mathcal X}_i=({\mathcal X}_i(s), s\geq 0)$ for $ i\in\N$, has the distribution of a sequence of independent processes with respective laws $P_{x_i}$, where $x_1\geq x_2\geq \ldots\ > 0$ denotes the ranked sequence of the positive jump  sizes of $-{\mathcal X}_{\varnothing}$. We continue in an obvious way for the second generation, and so on for the next generations; we refer to Jagers \cite{Jagers} for the rigorous  argument showing that this indeed defines uniquely the law ${\mathcal P}_x$. It will be convenient for definiteness  to agree that ${\mathcal P}_{\partial}$ denotes the law of the degenerate process on $\U$ such that ${\mathcal X}_u\equiv \partial$ for every $u\in\U$, $b_{\varnothing}=0$ and $b_u=\infty$ for $u\neq \varnothing$. We shall also write ${\mathcal E}_x$ for the mathematical expectation under the law ${\mathcal P}_x$. So, the self-similar Markov process $X$ governs the evolution of typical cells under ${\mathcal P}_x$, 
and $X$ will thus be often referred to as a cell process in this setting.

It is readily seen from the self-similarity of cells and the branching property that the point process on $\R$ induced  by the negative of the logarithms of the initial sizes of cells at a given generation, 
$${\mathcal Z}_n(\d z)=\sum_{|u|=n}\delta_{-\ln {\mathcal X}_u(0)}(\d z)\,,\qquad n\geq 0,$$
is a {\em branching random walk} (of course, the possible atoms corresponding to ${\mathcal X}_u(0)=\partial$ are discarded in the previous sum, and the same convention shall apply implicitly in the sequel). Roughly speaking, this means that for each generation $n$, the point measure ${\mathcal Z}_{n+1}$ is obtained from ${\mathcal Z}_{n}$ by replacing each of its atoms, say $z$, by a random cloud of atoms, $\{z+y^z_i: i\in\N\}$, where the family $\{y^z_i: i\in\N\}$ has a fixed distribution, and to different atoms $z$ of ${\mathcal Z}_{n}$ correspond  independent families $\{y^z_i: i\in\N\}$. We stress that the Lamperti transformation has no effect on the sizes of the jumps (the sizes of the jumps of $X$ and of $x\exp(\xi)$ are obviously the same) and thus no effect either on the 
 branching random walk ${\mathcal Z}_{n}$. In particular the law of ${\mathcal Z}_{n}$ does not depend on the self-similarity parameter $\alpha$.

The Laplace transform of the intensity of the point measure ${\mathcal Z}_1$,
which is defined by 
$$m(q) \coloneqq {\mathcal E}_1\left(\langle {\mathcal Z}_1, \exp(-q\cdot)\rangle\right)={\mathcal E}_1\left(\sum_{i=1}^{\infty} {\mathcal X}^q_i(0)\right)\,, \qquad q\in\R,$$
plays a crucial role in the study of branching random walks. In our setting, it is computed explicitly in terms of the function $\kappa$ in \eqref{eq:berGFlem4} and equals  \begin{eqnarray}  
  \label{eq:lem0i} m(q) = 1 - \kappa(q)/\Psi(q) \quad  \mbox{ when }\Psi(q) <0 \mbox{ and }\kappa(q) < \infty  \end{eqnarray} and infinite otherwise. In particular, when $\kappa(q) <0$,  the structure of the branching random walk  yields that
 \begin{eqnarray}{\mathcal E}_1\left(\sum_{u\in\U} {\mathcal X}^q_u(0)\right)=\frac{\Psi(q)}{\kappa(q) }\label{eq:lem0ii}.  \end{eqnarray}

 The  Laplace transform of the intensity $m(q)$ opens the way to  additive martingales, and in particular intrinsic martingales. These  have a fundamental role in the study of branching random walks, in particular in connection with the celebrated spinal decomposition (see for instance the Lecture Notes by Shi \cite{ZShi} and references therein), and we shall especially be  interested in describing its applications to self-similar growth-fragmentations.

\subsection{Two genealogical martingales and the intrinsic area measure} 
\label{sec:genealogical}

We start this section by observing from \eqref{eq:lem0i} that there is the equivalence
$$\kappa(\omega)=0 \quad \Longleftrightarrow \quad  m(\omega)=1.$$
Recall that $\omega_+$ is the largest root of the equation $\kappa(\omega)=0$ and,  whenever Cram\'er's condition \eqref{omega-} holds, $\omega_-$ the smallest root; in particular
 $m'(\omega_{+})>0$ and $m'(\omega_{-})\in(-\infty,0)$. One refers to $\omega_-$ as the Malthusian parameter (see e.g.~\cite[Sec.~4]{Biggins}).

This points at a pair of remarkable martingales, as we shall now explain. 
We write ${\mathcal G}_n$ for the sigma-field generated by the cells with generation at most $n$, i.e.~${\mathcal G}_n=\sigma({\mathcal X}_u: |u|\leq n)$ and  ${\mathcal G}_{\infty}=\bigvee_{n\geq 0} {\mathcal G}_n$. Note that for  a node $u$ at generation $|u|=n\geq 1$, the initial value ${\mathcal X}_{u}(0)$ of the cell labeled by $u$ is measurable with respect to the cell ${\mathcal X}_{u-}$, where $u-$ denotes the parent of $u$ at generation $n-1$.
We introduce for every $n\geq 0$ the ${\mathcal G}_{n}$-measurable variables
$${\mathcal M}^+(n)\coloneqq \sum_{|u|=n+1} {\mathcal X}_{u}^{\omega_+}(0) \quad \text{and} \quad {\mathcal M}^-(n)\coloneqq \sum_{|u|=n+1} {\mathcal X}^{\omega_-}_{u}(0),$$
where for the second definition, we implicitly assume that \eqref{omega-} holds.

Since $m(\omega_{\pm})=1$,  ${\mathcal M}^+$ and ${\mathcal M}^-$  are two $({\mathcal G}_n)$-martingales. Further, since $m'(\omega_+)>0$  and $m'(\omega_-)<0$, 
it follows from results of Biggins (see Theorem A in \cite{Biggins}) that ${\mathcal M}^+$ converges to $0$ a.s. whereas ${\mathcal M}^-$ is uniformly integrable (see also Lemma \ref{L5'} below). 

We now assume throughout the rest of  this section that the Cram\'er's hypothesis \eqref{omega-} holds and gather some important properties of the martingale ${\mathcal M}^-$, which  is known as the {\em Malthusian martingale} or also the {\em intrinsic martingale} in the folklore of branching random walks.
Call the cell process $X$ geometric if its trajectories, say starting from $1$, take values in $\{r^z: z\in\Z\}$  a.s.  for some fixed $r>0$. By Lamperti's transformation, 
$X$ is geometric if and only the L\'evy process $\xi$ is lattice, i.e. lives in $c\Z$ for some $c>0$.

\begin{lemma}\label{L5'}  Assume that \eqref{omega-}. Then the following assertions hold:
\begin{enumerate}
\item[(i)] ${\mathcal M}^-(1)\in L^{\omega_+/\omega_-}({\mathcal P}_x)$. 
\item[(ii)] Suppose that the cell process $X$ is not geometric. We have
$${\mathcal P}_x({\mathcal M}^-(\infty)>t)  \quad \mathop{\sim}_{t \rightarrow \infty} \quad  c \cdot t^{-\omega_+/\omega_-}$$
for a certain constant $c>0$. In particular,  ${\mathcal M}^-(\infty)\in L^r({\mathcal P}_x)$ if and only if $r<\omega_+/\omega_-$. 
\end{enumerate}
\end{lemma}

\begin{proof} (i) We start by observing that 
\begin{equation}\label{eqRiv}
\int_0^{\zeta} X^{\alpha+q}(s)\d s\in L^{\omega_+/q}(P_1)\qquad \hbox{for every }q\in(0,\omega_+]. 
\end{equation}
Indeed, Lamperti's transformation readily yields the identity
\begin{equation} \label{eq:chgtvr}
\int_0^{\zeta} X^{\alpha+q}(s)\d s=\int_0^{\infty} \e^{q\xi(t)}\d t.
\end{equation}
The finiteness of  its moment of order $\omega_+/q$ follows from Lemma 3 of Rivero \cite{Riv}, since 
$$E(\exp(\xi(1)\omega_+))=\exp(\Psi(\omega_+))<1.$$

Next, recall that, by construction, ${\mathcal M}^-(1)$ has the same law as $S(\zeta-)$, where
$$S(t)\coloneqq \sum_{0< s \leq t} |\Delta_{\hbox{-}} X(s)|^{\omega_-}\quad \hbox{ for }t\in[0,\zeta).$$
Using the Lamperti transformation, we easily see that the predictable compensator of $S$ is given by 
$$S^{(p)}(t)\coloneqq \int_0^{t} X^{\alpha +\omega_-}(s) \left( \int_{(-\infty,0)}(1-\e^{y})^{\omega_-}\Lambda(\d y)\right)  \d s,$$
and since $ \int_{(-\infty,0)}(1-\e^{y})^{q}\Lambda(\d y)=\kappa(q)-\Psi(q)$, we have simply that the process
$$S(t)-S^{(p)}(t) = \sum_{0< s \leq t} |\Delta_{\hbox{-}} X(s)|^{\omega_-}+\Psi(\omega_-)  \int_0^{t} X^{\alpha+\omega_-}(s)\d s$$
is a martingale. This martingale is obviously purely discontinuous and has quadratic variation
$$\sum_{0<s\leq t} |\Delta_{\hbox{-}} X(s)|^{2\omega_-}.$$

Using \eqref{eqRiv}, we know that  $S^{(p)}(\zeta-)\in L^{\omega_+/\omega_-}(P_1)$,  and  then the Burkholder-Davis-Gundy inequality reduces the proof to checking that 
\begin{equation} \label{eqBdG}
\sum_{0<s<\zeta} |\Delta_{\hbox{-}} X(s)|^{2\omega_-}\in L^{\omega_+/2\omega_-}(P_1).
\end{equation}
Suppose first $\omega_+/\omega_-\leq 2$, so that
$$\left(\sum_{0<s<\zeta} |\Delta_{\hbox{-}} X(s)|^{2\omega_-}\right)^{\omega_+/2\omega_-}\leq \sum_{0<s<\zeta} |\Delta_{\hbox{-}} X(s)|^{\omega_+}.$$
Since $\Psi(\omega_+)<\kappa(\omega_+)=0$, we know further from Lemma \ref{L1}  that 
$$E_1\left(\sum_{0<s<\zeta} |\Delta_{\hbox{-}} X(s)|^{\omega_+}\right) = -\Psi(\omega_+)E_1\left (\int_0^{\zeta} X^{\alpha+\omega_+}(s)\d s\right)=1,$$
which proves \eqref{eqBdG}. 

Suppose next that $2<\omega_+/\omega_-\leq 4$. The very same argument as above, with $2\omega_-$ replacing $\omega_-$, shows that 
$$\sum_{0< s \leq t} |\Delta_{\hbox{-}} X(s)|^{2\omega_-}+\Psi(2\omega_-)  \int_0^{t} X^{\alpha+2\omega_-}(s)\d s$$
is a purely discontinuous martingale with quadratic variation $\sum_{0<s\leq t} |\Delta_{\hbox{-}} X(s)|^{4\omega_-}$, and we conclude in the same way that \eqref{eqBdG} holds. We can repeat the argument when $2^k<\omega_+/\omega_-\leq 2^{k+1}$ for every $k\in\N$, and thus complete the proof of our claim by iteration.

(ii) The terminal value ${\mathcal M}^-(\infty)$ of the intrinsic martingale solves the critical homogeneous linear equation 
 \begin{eqnarray} \label{eq:pointfixevolume}{\mathcal M}^-(\infty)  \quad \mathop{=}^{(d)} \quad  \sum_{i=1}^{\infty} \Delta^{\omega_-}_i {\mathcal M}_i^-(\infty)\,, \end{eqnarray}
where in the right-hand side, $(\Delta_i)_{i\in\N}$ denotes the sequence of the absolute values of the negative jump sizes of the self-similar Markov process $X$ with characteristics $(\Psi,\alpha)$,
and $\left({\mathcal M}_i^-(\infty)\right)_{i\in\N}$ is a sequence of i.i.d. copies of ${\mathcal M}^-(\infty)$ which is further assumed to be independent of $(\Delta_i)_{i\in\N}$.

We infer in particular from results due to  Biggins (see Sec.~2 in \cite{Biggins}),  that this equation has a unique solution with given mean, and our claim is now  a consequence of Theorem 4.2 in \cite{JeOl}. Indeed,
$$\E\left(\sum_{i=1}^{\infty}\Delta_i^{\omega_-}\right)
=m(\omega_-)=1\quad 
, \quad \E\left(\sum_{i=1}^{\infty} \left(\Delta_i^{\omega_-}\right)^{\omega_+/\omega_-}\right) = \E\left(\sum_{i=1}^{\infty}\Delta_i^{\omega_+}\right)
=m(\omega_+)=1,$$
and also
$$\E\left(\left( \sum_{i=1}^{\infty}\Delta_i^{\omega_-}\right)^{\omega_+/\omega_-}\right) ={\mathcal E}_1({\mathcal M}_1(1)^{\omega_+/\omega_-})<\infty,$$
as it was proved in (i). 
\end{proof}

The terminal value ${\mathcal M}^-(\infty)$  will be referred here to as 
 the {\em intrinsic area} of the growth-fragmentation, appears in a variety of limit theorems; see, e.g. \cite{Nerman} and Theorem 6.1 in \cite{BiKy}. 
 
 \begin{remark}\label{rem:law} 
 Even though the law of ${\mathcal M}^-(\infty)$ is characterized by the distributional equation \eqref{eq:pointfixevolume}, it seems to be difficult to identify its distribution in general. However, we will be able to identify the law of a the intrinsic area of a distinguished family of growth-fragmentations by surprisingly using a connection with random planar maps.
 \end{remark}
 
 {In Sec.~\ref{sec:gen}, we will introduce a distinguished family of self-similar growth-fragmentations and identify the law of the limit of their Malthusian martingales as being size-biased versions of  stable random variable (see Corollary \ref{cor:area} below). To this end, we will crucially rely on a connection between these particular growth-fragmentations and random maps where in this context $ \mathcal{M}^-(\infty)$ is interpreted as the limit law for the \emph{area} (i.e.~the number of vertices) of the maps. Hence the name \emph{intrinsic area}.}

It is further well-known that the Malthusian martingale  yields an important random measure on the boundary $\partial \U$ of the Ulam tree; see e.g. Liu \cite{Liu} for background and references. Specifically, $\partial \U$ is a complete metric space  when endowed with the distance
$d(\ell, \ell')=\exp(-\inf\{n\geq 0: \ell(n)=\ell'(n)\})$, where the notation $\ell(n)$ designates the ancestor of the leaf $\ell$ at generation $n$. We can construct a (unique) random measure ${\mathcal A}$ on $\partial \U$, which we call the {\it intrinsic area measure}, such that the following holds. 
For every $u\in\U$,  we write $B(u)\coloneqq \{\ell \in\partial \U: \ell(n)=u\}$ for the ball in $\partial \U$ which stems from $u$, and introduce
$${\mathcal A}(B(u))\coloneqq \lim_{n\to \infty} \sum_{|v|=n, v(i)=u} {\mathcal X}^{\omega_-}_v(0)\,,$$
where $i=|u|$ stands for the generation of $u$, and for every vertex $v$ with $|v|\geq i$, $v(i)$ stands for the ancestor of $v$ at generation $i$ (the existence of this limit is ensured by the branching property). In particular, for $u=\varnothing$, the total mass of ${\mathcal A}$ is ${\mathcal A}(\partial \U)={\mathcal A}(B(\varnothing))={\mathcal M}^-(\infty)$. 
Alternatively, we can use in the same way  the Malthusian martingale to equip with a natural mass measure the continuum random tree associated by Rembart and Winkel  to the genealogy of  growth-fragmentations; see Corollary 4.2 in \cite{ReWi}.

\section{Self-similar growth-fragmentations and a many-to-one formula}
\label{sec:many-to-one}

\subsection{Self-similar growth-fragmentations}
\label{sec:ssgf}

The cell system ${\mathcal X}$ being constructed, we next define the cell population at time $t\geq 0$ as the family of  the sizes of the cells alive at time $t$, viz.
$${\bf X}(t)=\left\{ \!\!\left\{ {\mathcal X}_u(t-b_u): u\in\U\,,\, b_u\leq t< b_u+\zeta_u\right\}\!\!\right\},$$
where $b_u$ and $\zeta_u$ denote respectively the birth time and the lifetime of the cell $u$ and 
the notation $\left\{\!\!\left\{\cdots \right\}\!\!\right\}$ refers to multiset (i.e. elements are repeated according to their multiplicities).
According to Theorem 2 in \cite{BeGF} (again, the assumption of absence of positive jumps which was made there plays no role), \eqref{condkappa} ensures that the elements of ${\bf X}(t)$ can be ranked in the non-increasing order and then form a null-sequence (i.e. tending to $0$), say $X_1(t)\geq X_2(t)\geq \ldots\geq 0$; if ${\bf X}(t)$ has only finitely many elements, say $n$, then we agree that $X_{n+1}(t)=X_{n+2}(t)= \ldots =0$ for the sake of definiteness. Letting the time parameter $t$ vary, we call the process of cell populations $({\bf X}(t), t\geq 0)$ a {\it growth-fragmentation process} associated with the cell process $X$. 
We denote by $\P_x$ its law under ${\mathcal P}_x$.  

\begin{remark} By the above construction, the law of $ \mathbf{X}$ only depends on the law of the Eve cell  which in turn is characterized by $( \Psi, \alpha)$. However, different laws for the Eve cell may yield to the same growth-fragmentation process $ \mathbf{X}$. This has been analyzed in depth in \cite{QShi} where it is proved that the law of $ \mathbf{X}$ is in fact characterized by the pair $(\kappa, \alpha)$.  This reference only treats the case where positive jumps are not allowed, but the same arguments apply. 
\end{remark}

We now discuss the branching property of self-similar growth-fragmentations. In this direction, we first introduce 
$({\mathcal F}_t)_{t\geq 0}$, the natural filtration generated by $({\bf X}(t): t\geq 0)$, and recall from Theorem 2 in \cite{BeGF} that the growth-fragmentation process is Markovian with semigroup fulfilling the branching property. That is, conditionally on ${\bf X}(t)=(x_1, x_2, \ldots)$, the shifted process $({\bf X}(t+s): s\geq 0)$ is independent of ${\mathcal F}_t$ and its distribution is the same as that of the process obtained by
taking the union (in the sense of multisets) of a sequence of independent growth-fragmentation processes with respective distributions $\P_{x_1},\P_{x_2}, \ldots$. We stress that, in general, the genealogical structure of the cell system cannot be recovered from the growth-fragmentation process ${\bf X}$ alone, and this motivates working with the following enriched version. 

We introduce 
$$\bar{\bf X}(t)=\left\{ \!\!\left\{ ({\mathcal X}_u(t-b_u),|u|): u\in\U\,,\, b_u\leq t< b_u+\zeta_u\right\}\!\!\right\}, \qquad t\geq 0$$
that is we record not just the family of the sizes of the cells alive at time $t$, but also their generations. We denote  the natural filtration generated by $(\bar {\bf X}(t): t\geq 0)$ by  $(\bar{\mathcal F}_t)_{t\geq 0}$. 
In this enriched setting, it should be intuitively clear that the following slight variation of the branching property still holds. A rigorous proof can be given following an argument similar to that for Proposition 2 in \cite{BeGF}; details are omitted.

\begin{lemma}\label{L6} For every $t\geq 0$, conditionally on $\bar{\bf X}(t)=\left\{ \!\!\left\{ (x_i,n_i): i\in\N\right\}\!\!\right\}$, the shifted process $(\bar{\bf X}(t+s): s\geq 0)$ is independent of $\bar{\mathcal F}_t$ and its distribution is the same that of the process 
$$\bigsqcup _{i\in\N}\bar {\bf X}_i(s)\circ \theta_{n_i},\qquad s\geq 0,$$
where $\bigsqcup$ denotes the union in the sense of multisets,  
 $\theta_{n}$  the operator that consists in shifting all generations by $n$, viz. $\left\{ \!\!\left\{ (y_j, k_j): j\in \N \right\}\!\!\right\}\circ \theta_n=\left\{ \!\!\left\{ (y_j, k_j+n): j\in \N \right\}\!\!\right\}$, and the $\bar {\bf X}_i$ are independent processes, each having the same law as $\bar {\bf X}$ under ${\mathcal P}_{x_i}$.
\end{lemma} 

It is easy to see that ${\mathcal P}_x$-a.s., two different cells never split at the same time, that is for every $u,v\in\U$ with $u\neq v$, there is no $t\geq0$ with 
$b_u\leq t$, $b_v\leq t$ such that both $\Delta_{\hbox{-}} {\mathcal X}_u(t-b_u) <0$ and $\Delta_{\hbox{-}} {\mathcal X}_v(t-b_v) <0$. 
Further, there is an obvious correspondence between the negative jumps of cell-processes and those of $\bar{\bf X}$. Specifically, if for some $u\in\U$ and  $t>b_u$, 
${\mathcal X}_u((t-b_u)-)=a$ and ${\mathcal X}_u(t-b_u)=a'<a$, then at time $t$, one element $(a,|u|)$ is removed from $\bar{\bf X}(t-)$ and is replaced by a two elements $(a',|u|)$ and 
$(a-a', |u|+1)$. In the converse direction, the Eve cell process ${\mathcal X}_{\varnothing}$ is recovered from $\bar{\bf X}$ by following the single term with generation $0$, and this yields the birth times and sizes at birth of the cells at the first generation. Next, following the negative jumps of the elements in $\bar{\bf X}$ at the first generation, then the second generation, and so on, we can recover the family of the birth times and sizes at birth of cells at any given generation from the observation of the process $\bar{\bf X}$.

\subsection{The intensity measure of a self-similar growth-fragmentation}
\label{sec:intensity}

The main object of interest in the present section is
the intensity measure $\mu^x_t$ of ${\bf X}(t)$ under $\P_x$, which is defined by
$$\langle \mu^x_t, f\rangle \coloneqq \E_x\left(\sum_{i=1}^{\infty} f(X_i(t))\right),$$
with, as usual, $f: (0,\infty)\to [0,\infty)$ a generic measurable function and the convention $f(0)=0$. We shall obtain an explicit expression for $\mu_t^x$ in terms of the transition kernel of a certain self-similar Markov process. Formulas of this type are often referred to as {\it many-to-one} in the literature branching type processes (see e.g. the Lecture Notes by Shi \cite{ZShi}), to stress that the intensity of a random point measure is expressed in terms of the distribution of a single particle.  We stress however that,  even though self-similar growth-fragmentation processes are Markovian processes which fulfill the branching property (see Theorem 2 in \cite{BeGF}), their construction based on cell systems ${\mathcal X}$ and their genealogical structures of Crump-Mode-Jagers type make the analysis of $({\bf X}(t))_{ t\geq 0}$ as a process evolving with time $t$ rather un-direct, and the classical methods for establishing many-to-one formulas for branching processes do not apply straightforwardly in our setting.

In order to state the main result of this section, recall  that $\omega_+$ is the largest root of the equation $\kappa(q)=0$ and fulfills \eqref{omega+}. 
Recall also from Lemma \ref{Lphi+} that $\Phi^+(\cdot)=\kappa(\cdot +\omega_+)$ is the Laplace exponent of a L\'evy process $\eta^+$.
With the notation of Sec.~\ref{sec:lamperti},  we write $Y^+=(Y^+(t): t\geq0)$ for the self-similar Markov process with characteristics $(\Phi^+, \alpha)$, and denote by $P^+_x$   its law started from $Y^+(0)=x$,

\begin{theorem}\label{T1} For every $x>0$ and $t\geq 0$, there is the identity
$$\mu_t^x(\d y)=\left(\frac{x}{y}\right)^{\omega_{+}}P^+_x(Y^+(t)\in\d y)\,,\qquad y>0.$$
\end{theorem} 

\begin{proof} 
We start by observing that for every $q$ such that $\kappa(q)<0$, we have 
\begin{equation} \label{eq:mommu}
\int_0^{\infty}  \int_{(0,\infty)} y^{q+\alpha} \mu^1_t (\d y)  \d t = \E_1\left(\int_0^{\infty} \left( \sum_{i=1}^{\infty} X_i^{q+\alpha}(t)\right) \d t \right) = -1/\kappa(q).
\end{equation}
Indeed, due to the very construction of the growth-fragmentation process ${\bf X}$, there is the identity
$$\int_0^{\infty} \left( \sum_{i=1}^{\infty} X_i^{q+\alpha}(t)\right) \d t = \sum_{u\in\U} \int_0^{\zeta_u}{\mathcal X}^{q+\alpha}_u(t) \d t.$$
Because for each $u\in\U$, conditionally on ${\mathcal X}_u(0)=x$, ${\mathcal X}_u$  has the distribution of the self-similar Markov process $X$ with characteristics $(\Psi,\alpha)$ started from $x$, we deduce from Lemma \ref{L1}(i) below  that 
$$\E_1\left(\int_0^{\zeta_u}{\mathcal X}^{q+\alpha}_u(t) \d t \right)  = \E_1\left( -\frac{{\mathcal X}_u(0)^{q}}{\Psi(q)}\right).$$
An appeal to \eqref{eq:lem0ii} completes the proof of \eqref{eq:mommu}.

We then pick any $\theta>0$ with $\kappa(\theta)<0$ (note that this forces $\theta<\omega_{+}$) and consider $\tilde \Phi (q)\coloneqq \kappa(\theta+q)$. Then $\tilde \Phi$ is also the Laplace exponent of a L\'evy process, say $\tilde \eta$, which has  killing rate $-\kappa(\theta)>0$. We write $\tilde Y$ for the self-similar Markov process with characteristics $(\tilde \Phi, \alpha)$, and $\tilde \rho_t(x,\d y)$ for its transition (sub-)probabilities. As $\tilde \Phi(q)= \Phi^+(q+\theta-\omega_{+})$, the laws of the L\'evy processes $\eta^+$ and $\tilde \eta$ are equivalent on every finite horizon, and more precisely, there is the absolute continuity relation 
 $$E(F(\tilde \eta(s):{0\leq s \leq t}))=E(F( \eta^+(s): {0\leq s \leq t}) \exp((\theta-\omega_{+})\eta^+(t))),$$
 for every functional $F\geq 0$ which is zero when applied to a path with lifetime less than $t$. 
  It is straightforward to deduce that there is the identity 
 $$\tilde \rho_t(x, \d y)=\left(\frac{y}{x}\right)^{\theta-\omega_{+} }P^+_x(Y^+(t)\in\d y),$$
 and therefore our statement can also be rephrased as 
\begin{equation} \label{Eqtilde}
\mu_t^x(\d y)=\left(\frac{x}{y}\right)^{\theta}\tilde \rho_t(x, \d y)\,,\qquad y>0.
\end{equation}

To prove \eqref{Eqtilde}, we first note from the (temporal) branching property of growth-fragmentation processes (see Theorem 2 in \cite{BeGF}, or Lemma \ref{L6} here) that the family $(\mu^x_t)_{x,t\geq 0}$ fulfills the Chapman-Kolmogorov identity
$$\langle \mu_{t+s}^x,f\rangle = \int_{(0,\infty)} \langle \mu_t^y,f\rangle \mu_s(x, \d y).$$
However, this is not a Markovian transition kernel, as $\langle \mu_t^y,1\rangle$ may be larger than $1$, and even infinite. Nonetheless, thanks to Theorem 2 in \cite{BeGF},
since $\kappa(\theta)<0$, the power function $h:y\mapsto y^{\theta}$ is excessive for the 
growth-fragmentation process ${\bf X}$, that is  $\langle \mu_t^x, h\rangle \leq h(x)$ for every $x>0$. It follows that if we introduce the super-harmonic transform
$$\tilde \varrho_t(x, \d y)\coloneqq \left(\frac{y}{x}\right)^{\theta}\mu_t^x(\d y)\,,$$
then the semigroup property of the measures $\mu_t^x(\d y)$ is transmitted to $\tilde \varrho_t(x, \d y)$, which then form a transition (sub-)probability kernel of a Markov process on $(0,\infty)$. Further,  the measures $\mu_t^x(\d y)$ fulfill the scaling property 
$$\langle \mu_t^x,f\rangle = \langle \mu_{tx^{\alpha}}^1,f(x\, \cdot) \rangle$$ 
(see again Theorem 2 in \cite{BeGF}), and this also propagates to $\tilde \varrho_t(x, \d y)$. 

Now it suffices to observe that for every $q>0$ with $\kappa(q)<0$, we have
$$\int_0^{\infty} \int_{(0,\infty)} y^{q+\alpha-\theta} \tilde \varrho_t(1, \d y) \d t= \int_0^{\infty} \int_{(0,\infty)} y^{q+\alpha} \mu^1_t(\d y) \d t= -\frac{1}{\kappa(q)}\,,$$
where the second identity follows from \eqref{eq:mommu}. Recalling that $\kappa(q)=\tilde  \Phi(q-\theta)$, Lemma \ref{L1}(ii) below enables us to conclude that $\tilde \varrho= \tilde \rho$.
\end{proof}

In the proof of Theorem \ref{T1}, we used the following simple characterization of the transition kernel  of a self-similar Markov process  in terms of the Mellin transform of its potential measure.  
 
\begin{lemma}\label{L1} Let $X$ be a self-similar Markov process  with characteristics $(\Psi,\alpha)$.
 \begin{itemize}
\item[(i)] For every $q>0$ and $x>0$, we have
$$\int_0^{\infty} E_x(X(t)^{q+\alpha})  \d t =\left\{\begin{matrix} -{x^{q}}/{\Psi(q)}& \hbox{if $\Psi(q)< 0$,}\\
\infty & \hbox{ otherwise.}\end{matrix} \right.$$

\item[(ii)] Conversely,  consider a sub-Markovian transition kernel $(\varrho_t(x,\d y))_{t\geq 0}$ on $(0,\infty)$ which is self-similar, in the sense that for every $t\geq 0$ and $x>0$,  $$\langle \varrho_t(x, \cdot), f\rangle= \langle \varrho_{tx^{\alpha}}(1, \cdot), f(x\cdot)\rangle.$$ If further 
$$\int_{0}^{\infty} \d t\int_{(0,\infty)} y^{q+\alpha}\varrho_t(1,\d y)= -\frac{1}{\Psi(q)}\qquad \hbox{whenever $\Psi(q)< 0$ },$$
then $\varrho_t(x,\d y))= P_x(X(t)\in \d y)$.
\end{itemize}
\end{lemma}

\begin{proof} (i) For $x=1$, this follows readily from \eqref{EqLap}, \eqref{eq:chgtvr} and Tonelli's Theorem, and then for general $x>0$, from self-similarity.

(ii)  Our assumptions ensure that $(\varrho_t(x,\d y))_{t\geq 0}$ is the transition kernel of a self-similar Markov process on $(0,\infty)$, say $Y=(Y(t), t\geq 0)$. We  know from Lamperti \cite{Lam72} that there is a (possibly killed) L\'evy process $\eta=(\eta(t))_{t\geq 0}$ whose Lamperti transform has the same law as $Y$. We write $\Phi$ for the Laplace exponent of $\eta$,  which is given by $\Phi(q)=\ln\E(\exp(q\eta(1)))$. 
We deduce from (i)  that $\Phi(q)=\Psi(q)$ 
provided that $\Psi(q)<0$. Since  the set $\{q>0:\Psi(q)<0\}$ has a non-empty interior, we conclude by analytic continuation of the characteristic functions that  $\eta$ has the same law as $\xi$. 
\end{proof}

\subsection{Two temporal martingales} 
\label{sec:temporal}

We shall now present some applications of Theorem \ref{T1}, starting with remarkable temporal  martingales (recall also Proposition \ref{P0}) which, roughly speaking, are the temporal versions of the genealogical martingales of Sec.~\ref{sec:genealogical}.
Namely, we define first for every $t\geq 0$
$$M^+(t)\coloneqq \sum_{i=1}^{\infty} X_i^{\omega_+}(t)\quad \text{and} \quad M^-(t)\coloneqq \sum_{i=1}^{\infty} X_i^{\omega_-}(t),$$ 
where for the second definition, we implicitly assume that Cram\'er's hypothesis \eqref{omega-}  is fulfilled.

\begin{corollary}\label{C3}   \begin{enumerate} 
\item[(i)] For $\alpha\leq 0$, $M^+$ is a $\P_x$-martingale, whereas for $\alpha >0$, $M^+$ is  a $\P_x$-supermartingale which converges to $0$ in $L^1(\P_x)$.

\item[(ii)] Assume  \eqref{omega-}. For $\alpha\geq 0$, $M^-$ is a $\P_x$-martingale, whereas for
$\alpha <0$, $M^-$ is  a $\P_x$-supermartingale and there exists some constant $c\in(0,\infty)$ such that
 $$\E_x\left(M^-(t)\right) \sim  c x^{\omega_{+}}t^{(\omega_+-\omega_{-})/\alpha}\qquad \hbox{as }t\to \infty$$
 
For  $\alpha> 0$, there exists some constant $c'\in(0,\infty)$ such that
 $$\E_x\left(M^+(t)\right) \sim  c' x^{\omega_{-}}t^{-(\omega_+-\omega_{-})/\alpha}\qquad \hbox{as }t\to \infty.$$
\end{enumerate}
\end{corollary}

\begin{proof}  (i)  Set for $q>0$ and ${\bf x}=(x_1,\cdots)$ 
$$F_q({\bf x})=\sum_{n=1}^{\infty} x_n^q,$$
so that $M^{\pm}(t)=F_{\omega_{\pm}}({\bf X}(t))$. 
 We know from Theorem \ref{T1} that for every $x>0$, there are the identities
$$\E_x(F_{\omega_{+}}({\bf X}(t)))= \langle \mu_t^x, F_{\omega_{+}}\rangle =  x^{\omega_{+}}P_x^+(Y^+(t)\in(0,\infty)) ,$$
where $Y^+$ denotes  the self-similar Markov process  with characteristics $(\Phi^+, \alpha)$. Recall that $\Phi^+(0)=0$ and  $(\Phi^+)'(0)=\kappa'(\omega_{+})>0$, so that the  L\'evy process $\eta^+$ with Laplace exponent $\Phi^+$ drifts to $+\infty$. By Lamperti's construction, this entails that the lifetime of $Y^+$ is a.s. infinite if $\alpha \leq 0$, and a.s. finite if $\alpha>0$.
Thus $\E_x(F_{\omega_{+}}({\bf X}(t)))=x^{\omega_+}$ if $\alpha \leq 0$, whereas $\lim_{t\to \infty} \E_x(F_{\omega_{+}}({\bf X}(t)))=0$ if $\alpha >0$, and 
 our claims follows easily from the branching property of growth-fragmentations. 
 
 (ii) The first claim follows from the same argument as in (i).
Then assume $\alpha<0$ and take first $x=1$. Then combining Theorem \ref{T1}, Lamperti's construction and Esscher transformation,  we see that
$$\E_1\left(M^-(t)\right)=P\left(\int_0^{\infty} \exp(-\alpha \eta^-(s))\d s>t\right),$$
where $\eta^-$ denotes a L\'evy process with Laplace exponent $\Phi^-\coloneqq \kappa(\cdot+\omega_-)$.
The right-hand side can be estimated using results by Rivero \cite{Riv}. Indeed Cram\'er's condition holds, namely
$$E\left(\exp((\omega_{+}-\omega_-) \eta^-(1))\right)=1 \quad \hbox{and} \quad E\left (|\eta^-(1)|\exp((\omega_{+}-\omega_-)\eta^-(1))\right) =E(|\eta^+(1)|)<\infty\,,$$
and it follows from Lemma 4 in \cite{Riv} that 
$$P\left(\int_0^{\infty} \exp(-\alpha \eta^-(s))\d s>t\right) \, \sim\, c t^{  (\omega_{+}-\omega_-)/\alpha}\qquad \hbox{ as }t\to \infty.$$
This establishes our claim for $x=1$, and the general case then follows by scaling. 
The remaining assertion about $M^+$ follows  {\it mutatis mutandis} from the same argument.
\end{proof}

We next relate the (super-)martingales in continuous time $M^{\pm}$ to the discrete parameter martingale ${\mathcal M}^{\pm}$ introduced in the preceding section. 
 In this direction, we recall that even though  the martingale property of certain additive functionals
of a branching process with a discrete genealogy may be preserved by evaluation along increasing families of optional lines (see in particular \cite{Jagers},  \cite{Kyp0} and Section 6 in \cite{BiKy}), 
the family of cells alive at a given time does not  form an optional line (see Example c on page 190 in \cite{Jagers}), and  the aforementioned references thus do not apply directly here. However similar arguments can often be used, and then we shall merely provide sketch proofs.

Recall that $\bar{\bf X}(t)$ denotes the enriched growth fragmentation process in which the sizes of cells alive at time $t$ are recorded together with their generations, and that $(\bar{\mathcal F}_t)_{t\geq0}$ denotes the natural filtration of this enriched process. 

\begin{lemma}\label{L12} For every $t\geq 0$, there is the identity
$$M^+(t)=\lim_{n\to \infty}{\mathcal E}_x\left({\mathcal M}^+(n)\mid \bar{\mathcal F}_t\right)
\qquad {\mathcal P}_x\hbox{-a.s.},$$
and this convergence also holds in $L^1({\mathcal P}_x)$ when $\alpha\leq 0$. 
\end{lemma}

\begin{proof} We first observe from the branching property stated in Lemma \ref{L6} and the martingale property of ${\mathcal M}^+$ that for every $n\geq0$ and $t\geq 0$, we have
$${\mathcal E}_x\left({\mathcal M}^+(n)\mid \bar{\mathcal F}_t\right) 
= \sum_{|u|=n+1} {\bf1}_{b_u\leq t}{\mathcal X}^{\omega_+}_u(0) + \sum_{|v|\leq n}{\bf1}_{b_v\leq t}{\mathcal X}^{\omega_+}_v(t-b_v)\,, \qquad{\mathcal P}_x\hbox{-a.s.}$$
Because ${\mathcal M}^+(n)$ converges to $0$ a.s., the limit as $n\to \infty$ of the first term in the sum of the right-hand side equals $0$.
On the other hand, by monotone convergence, we have
$$\lim_{n\to\infty} \sum_{|v|\leq n}{\bf1}_{b_v\leq t}{\mathcal X}^{\omega_+}_v(t-b_v)
= \sum_{v\in\U}{\bf1}_{b_v\leq t}{\mathcal X}^{\omega_+}_v(t-b_v)=M^+(t).$$
The second assertion follows from the fact that $M^+$ is a martingale when $\alpha \leq 0$, so
$${\mathcal E}_x(M^+(t)) = x^{\omega_+} = {\mathcal E}_x({\mathcal M}^+(n)),$$
and we complete the proof with an application of  Scheff\'e's lemma.
\end{proof}

The intrinsic area measure ${\mathcal A}$ which was introduced at the end of  Sec.~\ref{sec:genealogical} enables us to describe close links between the Malthusian martingale ${\mathcal M}^-$, its terminal value, and the (super)-martingale $M^-$.
Recall that  for every leaf $\ell\in\partial \U$,
$b_{\ell}\coloneqq\lim_{n\to \infty} \uparrow b_{\ell(n)}$ with $\ell(n)$ the parent of $\ell$ at generation $n$.

\begin{theorem}\label{T3} Suppose  \eqref{omega-}. Then the following assertions hold:
\begin{enumerate}
\item[(i)] For every $t\geq 0$, there is the identity
$${\mathcal E}_x\left( {\mathcal M}^-(\infty)\mid \bar {\mathcal F}_t\right)= {\mathcal A}\left(\{\ell\in\partial \U: b_{\ell}\leq t\}\right) + M^-(t).$$
\item[(ii)] If $\alpha \geq 0$, ${\mathcal A}\left(\{\ell\in\partial \U: b_{\ell}<\infty\}\right)=0$ a.s., then the martingale $M^-$ is bounded in $L^p$ for every $p<\omega_+/\omega_-$, and there is the a.s. identity
$M^-(\infty)={\mathcal M}^-(\infty)$. 
\end{enumerate}
\end{theorem}

\begin{proof} (i) We start by observing that, just as in Lemma \ref{L12}, 
$${\mathcal E}_x\left({\mathcal M}^-(n)\mid \bar{\mathcal F}_t\right) 
= \sum_{|u|=n+1} {\bf1}_{b_u\leq t}{\mathcal X}^{\omega_-}_u(0) + \sum_{|v|\leq n}{\bf1}_{b_v\leq t}{\mathcal X}^{\omega_-}_v(t-b_v).$$
As $n\to\infty$, the left-hand side converges to  ${\mathcal E}_x\left( {\mathcal M}^-(\infty)\mid \bar {\mathcal F}_t\right)$. In the right-hand side, the first term of the sum converges to ${\mathcal A}\left(\{\ell\in\partial \U: b_{\ell}\leq t\}\right)$ by definition of the intrinsic area measure, and 
the second to $M^-(t)$. 

(ii) When $\alpha \geq 0$, $M^-(t)$ is a martingale and 
${\mathcal E}_x(M^-(t)) = x^{\omega_-} = {\mathcal E}_x({\mathcal M}^-(\infty))$. 
By (i), this entails that ${\mathcal A}\left(\{\ell\in\partial \U: b_{\ell}\leq t\}\right)=0$ a.s.  and 
 $M^-$ is uniformly integrable.
Finally, observe  that ${\mathcal M}^-(n)$ is $\bar {\mathcal F}_{\infty}$-measurable (see the discussion after Lemma \ref{L6}), thus so is ${\mathcal M}^-(\infty)$, and this yields $M^-(\infty)={\mathcal M}^-(\infty)$ a.s. Finally, recall from Lemma \ref{L5'} that then 
${\mathcal M}^-(\infty)\in L^p$ if and only if $p<\omega_+/\omega_-$.
\end{proof}

\begin{remark}\label{rem:lawM}Recall that the law of the martingale $ \mathcal{M}^{-}$  does not depend on the self-similarity parameter $\alpha$ but only on the  Laplace exponent $\Psi$ of the L\'evy process that drives the evolution of the Eve cell. Actually, the law of the  limiting value $ \mathcal{M}^-(\infty)$ only depends on the cumulant function $\kappa$. Indeed, by Theorem \ref{T3}, we have $ \mathcal{M}^-(\infty)=M^{-}(\infty)$ for $\alpha=0$, and so $ \mathcal{M}^-(\infty)$ only depends on the law of the associated homogeneous growth-fragmentation, which itself is  characterized by $\kappa$, see \cite{QShi}.
\end{remark}

\section{Spinal decompositions}
\label{sec:spine1}

Following Lyons, Pemantle \& Peres \cite{LPP}, the main purpose of this section will be to describe the growth-fragmentation process under the tilted probability measures associated with the martingales ${\mathcal M}^{\pm}$.  Recall that Lyons, Pemantle \& Peres  used additive martingales in branching random walks to tag a leaf at random in the genealogical tree of the process. The ancestral lineage of the tagged leaf forms the so-called spine, which  is viewed as a branch consisting of  tagged particles. Roughly speaking, the spinal decomposition claims that the untagged particles evolve as the ordinary ones in the branching random walk, whereas tagged particles reproduce according to a biased reproduction law, and the tagged child at the next generation is then picked by biased sample from the children of the tagged parent. 

Informally, using ${\mathcal M}^+$ to perform the probability tilting can be thought of as conditioning a distinguished cell to grow indefinitely, whereas the Malthusian martingale ${\mathcal M}^-$ rather corresponds to tagging a cell  randomly according to the intrinsic area measure ${\mathcal A}$. 
As the arguments are very similar, we shall provide complete proofs in the first case, and skip details in the second.

\subsection{Conditioning on indefinite growth}
\label{sec:plus}

In the setting of cell systems,  we first define a probability measure $\widehat{\mathcal P}^+_x$ describing the joint distribution of a cell system ${\mathcal X}=({\mathcal X}_u: u\in\U)$ (recall that we use canonical notation) and a leaf ${\mathcal L}\in\partial \U$,
where $\partial \U$ denotes the boundary of the Ulam tree, that is the space $\N^{\N}$ of infinite sequences of positive integers. 
Recall that ${\mathcal G}_n$ denotes the sigma-field generated by the cells with generation at most $n$. To start with, for every $n\geq 0$, the law of $({\mathcal X}_u: |u|\leq n)$ under 
$\widehat{\mathcal P}^+_x$  is absolutely continuous with respect to the restriction of ${\mathcal P}_x$ to ${\mathcal G}_n$, with density $x^{-\omega_{+}}{\mathcal M}^+(n)$, viz. 
$$\widehat{\mathcal P}^+_x\left(\Gamma_n\right)=x^{-\omega_+} {\mathcal E}_x\left({\mathcal M}^+(n) {\mathbbm{1}}_{\Gamma_n}\right)\qquad \hbox{ for every event }\Gamma_n\in{\mathcal G}_n.$$

We then discuss the tagged leaf ${\mathcal L}$. First, for every leaf $\ell\in\partial \U$, we write $\ell(n)$ for the parent of $\ell$ at generation $n$, that is the sequence of the $n$ first elements of $\ell$. Then the conditional law of the parent ${\mathcal L}(n+1)$ of the tagged leaf at generation $n+1$ is given by
\begin{equation}\label{eqlam}
\widehat{\mathcal P}^+_x\left({\mathcal L}(n+1)=v\mid {\mathcal G}_n\right) = \frac{{\mathcal X}_v^{\omega_+}(0)}{{\mathcal M}^+(n)}\,,\qquad\hbox{for every $v$ at generation $|v|=n+1$.}
\end{equation}
The coherence of this definition is ensured by the martingale property of ${\mathcal M}^+$ and the branching structure of cell systems. 
The sequence 
$(-\ln {\mathcal L}(n): n\geq 0)$ corresponds precisely to the spine in the framework considered by  Lyons, Pemantle \& Peres  \cite{LPP} and  provides the discrete (more precisely, generational) skeleton of the tagged cell, which we now introduce.

The birth-times $b_{{\ell}(n)}$ of the cells on the ancestral lineage of a leaf ${\ell}\in \partial \U$  form an increasing sequence, which converges to $b_{\ell}\coloneqq\lim_{n\to\infty} b_{{\ell}(n)}$. Focussing on the tagged leaf ${\mathcal L}$, we set $\hat{\mathcal X}(t)=\partial$ (recall that $\partial$ stands for a cemetery point) for $t\geq b_{\mathcal L}$ and 
$$ \hat{\mathcal X}(t) =  {\mathcal X}_{{\mathcal L}(n_t)}(t-b_{{\mathcal L}(n_t)})\qquad \hbox{ for } t<b_{{\mathcal L}}\,,$$
where $n_t$ denotes the  generation of the parent of the tagged leaf at time $t$, that is the unique integer $n\geq 0$ such that  $b_{{\mathcal L}(n)}\leq t <  b_{{\mathcal L}(n+1)}$. We should think of  $\hat{\mathcal X}(t)$ as the size of the {\it tagged cell} at time $t$;  understanding its evolution provides the key to  many properties of the law $\widehat{\mathcal P}^+_x$. 

 Observe from the very definition \eqref{eqlam},  that for every ${\mathcal G}_n$-measurable random variable ${\mathcal B}(n)\geq 0$, there is the identity
$$x^{\omega_{+}} \widehat{\mathcal E}^+_x(f({\mathcal X}_{{\mathcal L}_{n+1}}(0)) {\mathcal B}(n))={\mathcal E}_x\left(\sum_{|u|=n+1}
{\mathcal X}^{\omega_+}_u(0) f({\mathcal X}_u(0)){\mathcal B}(n)\right).$$
It should not come as a surprise that an analog identity with generations replaced by times holds.

\begin{proposition}\label{P2} For every $t\geq 0$, every measurable function $f: [0,\infty)\to [0,\infty)$, and every $\bar {\mathcal F}_t$-measurable random variable $B(t)\geq 0$, 
we have
$$ x^{\omega_{+}} \widehat{\mathcal E}^+_x(f(\hat{\mathcal X}(t)) B(t) )={\mathcal E}_x\left(\sum_{i=1}^{\infty} X^{\omega_{+}}_i(t)f(X_i(t))B(t)\right),$$
with the usual convention that $f(\partial)=0$. 
\end{proposition}

\begin{proof} Let us first assume that $B(t)$ is $\bar{\mathcal F}_t\wedge {\mathcal G}_k$-measurable for some fixed $k\in\N$. 
Since $f(\partial)=0$ and $\hat{\mathcal X}(t)=\partial$ for $t> b_{\mathcal L}$, we have
$$\widehat{\mathcal E}^+_x(f(\hat{\mathcal X}(t)) B(t) ) = \lim_{n\to \infty}\widehat{\mathcal E}^+_x(f(\hat{\mathcal X}(t)) B(t) \mathbbm{1}_{b_{{\mathcal L}(n+1)}>t}).$$
Then, provided that $n>k$, we have from the definition of the tagged leaf that 
$$\widehat{\mathcal E}^+_x(f(\hat{\mathcal X}(t)) B(t) \mathbbm{1}_{b_{{\mathcal L}(n+1)}>t})
=x^{-\omega_+} {\mathcal E}_x\left (\sum_{|u|=n+1} {\mathcal X}_u^{\omega_+}(0)  {\mathbbm{1}}_{b_u>t}  f({\mathcal X}_{u(t)}(t-b_{u(t)}))B(t) \right) ,$$
where, for every cell $u$ with $b_u>t$, $u(t)$ denotes the most recent ancestor of $u$ which is alive at time $t$. Just as in the proof of Lemma \ref{L12}, we decompose the family of cells $u$ at generation $n+1$ which are born after time $t$ into sub-families having  the same most recent ancestor alive at time $t$. Applying the branching property in Lemma \ref{L6} and the martingale property of ${\mathcal M}^+$, we get
$$ {\mathcal E}_x\left (\sum_{|u|=n+1} {\mathcal X}_u^{\omega_+}(0)  {\mathbbm{1}}_{b_u>t}  f({\mathcal X}_{u(t)}(t-b_{u(t)}))B(t) \right) =
{\mathcal E}_x\left(  \sum_{|v|\leq n}{\bf1}_{b_v\leq t}{\mathcal X}^{\omega_+}_v(t-b_v) f({\mathcal X}_{v}(t-b_{v}))B(t) \right).$$

Letting $n\to \infty$ and applying  monotone convergence establish  the formula of the statement when $B(t)$ is $\bar{\mathcal F}_t\wedge {\mathcal G}_k$-measurable. The case when we only assume that $B(t)$ is $\bar{\mathcal F}_t$-measurable then follows by a monotone class argument. 
\end{proof}

We next turn our attention to  the so-called {\it spinal decomposition} popularized by Lyons, Pemantle \& Peres  for branching random walks. In words, we follow the tagged cell as time passes, and for each of its negative jumps, we record the entire growth-fragmentation process which that jump generates
(that is, as usual, we interpret a negative jump as a birth event, and then record the growth-fragmentation process corresponding to the new-born cell). Roughly speaking, we shall see in the next theorem, that under $\widehat{\mathcal P}^+_1$, the law of the tagged cell  is given by that of a certain self-similar Markov process, and that conditionally on the path of the tagged cell, the growth-fragmentation processes generated by the negative jumps of $\hat{\mathcal X}$
are independent with the law $\P_x$, where $x$ is the (absolute) size of the jump. 

 In order to give a precise statement, we need to  label the negative jumps of $\hat{\mathcal X}$.
 If those  jumps times were isolated, then we could simply enumerate them in increasing order. Alternatively, if the tagged cell converged to $0$, then we could enumerate its negative jumps in the increasing order of their absolute sizes. However this is not the case in general, and we shall therefore introduce a deterministic algorithm, which is tailored for our purpose.

We first label each negative jump of the tagged cell by a pair $(n,j)$, where $n\geq 0$ denotes the generation of the tagged cell immediately before the jump occurs, and $j\geq 1$ the rank of that jump amongst the negative jumps which occurred while the generation of the tagged cell equals $n$, including the terminal jump when the generation of the tagged cell increases by one unit. We then write $\hat{\bf X}_{n,j}=(\hat{\bf X}_{n,j}(t): t\geq 0)$ for the growth-fragmentation generated by the $(n,j)$-jump. Specifically, on the one hand, if  the generation of the tagged cell does not increase during the $(n,j)$-jump, and if $u$ denotes the label of  the cell born at this birth event, then
$$\hat{\bf X}_{n,j}(t)=\left\{ \!\!\left\{ {\mathcal X}_{uv}(t+b_u-b_{uv}): v\in\U\,,\, b_{uv}\leq t+b_u< b_{uv}+\zeta_{uv}\right\}\!\!\right\}.$$
On the other hand, if the $(n,j)$-jump occurs at  an instant when the generation of the tagged cell increases, $u$ and $ui$ are the labels of the tagged cell immediately before, respectively after, the jump, then 
$$\hat{\bf X}_{n,j}(t)=\left\{ \!\!\left\{ {\mathcal X}_{uv}(t+b_{ui}-b_{uv}): v\in\U\,,v\neq i\, ,\, b_{uv}\leq t +b_{ui}< b_{uv}+\zeta_{uv}\right\}\!\!\right\}.$$
For definitiveness, we agree as usual that $\hat{\bf X}_{n,j}(t)\equiv \partial$ when the $(n,j)$-jump does not exist.
It should be plain that the entire growth-fragmentation process ${\bf X}$ can be recovered from
the process $(\hat{\mathcal X}(t), n_t)_{0\leq t< b_{\mathcal L}}$ of the tagged cell and its generation, and the family of processes $\left(\hat{\bf X}_{n,j}: n\geq 0, j\geq 1\right)$.

Next, recall Lemma \ref{Lphi+} and the notation there; 
in particular  $\eta^+$ denotes a L\'evy process  with L\'evy measure $\Pi$. By the L\'evy-It\^o decomposition, we can think of the negative jumps of $\eta^+$ as resulting from the superposition of two independent Poisson point processes on $\R_+\times (-\infty,0)$, the first with intensity $\e^{x\omega_+ }\d t \, \Lambda(\d x)$, and the second with intensity $\e^{x\omega_+ }\d t \, \tilde \Lambda(\d x)$. Because 
$\kappa(\omega_+)=0$, \eqref{eqkappa} entails that 
$$\int_{(-\infty,0)}\e^{x\omega_+ } \tilde \Lambda(\d x) = \int_{(-\infty,0)}(1-\e^x)^{\omega_+}\Lambda(\d x)=-\Psi(\omega_+)<\infty,$$
so if we mark the negative jumps of $\eta^+$ that correspond to the second point process, then the number of marked jumps up to time $t\geq 0$,  $N^+(t)$, forms a Poisson process with intensity $-\Psi(\omega_+)$. Recall also that we defined  $Y^+$  as the self-similar Markov process with characteristics $(\Phi^+,\alpha)$ started from $1$. Specifically we write the Lamperti time-substitution (recall \eqref{Eqcht})
$$\tau^+_t\, \coloneqq\,\inf\left\{r\geq0: \int_{0}^{r}\exp(-\alpha \eta^+(s))\d s\geq
t\right\}\quad \hbox{for} \quad0\leq t <I^+\coloneqq \int_{0}^{\infty}\exp(-\alpha \eta^+(s))\d s,$$
and 
$$Y^+(t)\coloneqq \exp(\eta^+(\tau^+_t))$$
with the usual convention that $Y^+(t)=\partial$ for $t\geq I^+$. We are now able to
claim:

\begin{theorem}\label{T2}  The distribution of $(\hat{\mathcal X}(t), n_t)_{0\leq t< b_{\mathcal L}}$ under $\widehat{\mathcal P}^+_1$ is the same as that of 
$$(Y^+(t), N^+(\tau^+_t))_{0\leq t < I^+}.$$
Further, conditionally on $(\hat{\mathcal X}(t), n_t)_{0\leq t< b_{\mathcal L}}$,
the processes $\hat{\bf X}_{n,j}$ for $ n\geq 0$ and $ j\geq 1$ are independent, and  each $\hat{\bf X}_{n,j}$ has the (conditional) law $\P_x$, with $x$ the absolute size of the jump of $\hat{\mathcal X}$ with label $(n,j)$. 

\end{theorem}

\begin{proof} We need only to establish the statements for $\alpha=0$, as the general case  then follows from Lamperti's transformation. We focus first on the first claim and observe,  by the branching property of cell systems and the Markov property of Poisson point processes, that it suffices to verify that the distribution of $(\hat{\mathcal X}(t))_{0\leq t\leq b_{{\mathcal L}(1)}}$ under $\widehat{\mathcal P}^+_1$ is the same as that of $(\exp(\eta^+(t)))_{0\leq t \leq T_1}$, where $T_1$ denotes the instant when the first atom of the Poisson point process with intensity 
$\e^{x\omega_+ }\d t \,  \tilde \Lambda(\d x)$ arises. We stress that the endpoints $b_{{\mathcal L}(1)}$ and $T_1$ are included in the life-interval of those processes.

In this direction, we use Lemma \ref{Lphi+}(i) and  decompose the L\'evy process $\eta^+$ as the sum of two independent L\'evy processes, 
$$\eta^+(t)=\xi^+(t)+\nu(t),$$ where $(\nu(t))_{t\geq 0}$ is a compound Poisson process with L\'evy measure $\e^{x\omega_+ } \tilde \Lambda(\d x)$, so for every $q,t\geq 0$
$$
E\left(\exp(q\nu(t))\right) 
=\exp\left(t\int_{(-\infty,0)} \left( (1-\e^{x})^{q+\omega_+ }-(1-\e^{x})^{\omega_+ }\right) \Lambda(\d x)\right).
$$
The identity 
$$E\left(\exp(q\xi^+(t))\right)= E\left(\exp(q\eta^+(t))\right)/E\left(\exp(q\nu(t))\right)$$
combined with  \eqref{eqkappa} then entails that the Laplace exponent of the L\'evy process  $\xi^+$ is $\Psi(q+\omega_+)-\Psi(\omega_+)$. 

Since $T_1$ is the instant of the first negative jump of $\nu$, standard properties of Poisson random measures then show  that $\Delta_{\hbox{-}} \nu(T_1)=\Delta_{\hbox{-}} \eta^+(T_1)$ has the law $-\Psi(\omega_+)^{-1}\e^{x\omega_+ } \,  \tilde \Lambda(\d x)$, and is independent of the process $(\eta^+(t))_{0\leq t < T_1}= (\xi^+(t))_{0\leq t < T_1}$ (note that the right-extremity $T_1$ of the time interval is now excluded). The latter  has the distribution of the  L\'evy process $\xi^+$ killed at an independent exponential time with parameter $-\Psi(\omega_+)$, and we deduce from above that its Laplace exponent  is $\Psi^+(q)\coloneqq\Psi(q+\omega_+)$. 
This entirely describes the law of $(\eta^+(t))_{0\leq t \leq T_1}$, and we shall now check that the law of $(\ln \hat{\mathcal X}(t))_{0\leq t\leq b_{{\mathcal L}(1)}}$ under $\widehat{\mathcal P}^+_1$ can be depicted in the same way. 

Consider an arbitrary functional $F\geq 0$ on the space of finite c\`adl\`ag paths, and $g:(-\infty,0)\to \R_+$ measurable. We aim at computing the quantity
$$\widehat {\mathcal E}^+_1\left(F(\ln {\mathcal X}_{\varnothing}(s): 0\leq s <b_{{{\mathcal L}(1)}})g\left(\ln {\mathcal X}_{{{\mathcal L}(1)}}(0)-\ln {\mathcal X}_{\varnothing}({{\mathcal L}(1)}-)\right)\right),$$
which, by the definition of the tagged leaf,  we can express in the form (recall that we place ourselves in the homogeneous case $\alpha=0$)
\begin{eqnarray*} & & {\mathcal E}_1\left(\sum_{t>0}F(\ln {\mathcal X}_{\varnothing}(s): 0\leq s <t)g\left(\ln |\Delta_{\hbox{-}}{\mathcal X}_{\varnothing}(t)|-\ln {\mathcal X}_{\varnothing}(t-)\right) |\Delta_{\hbox{-}}{\mathcal X}_{\varnothing}(t)|^{\omega_+}\right)\\
&=& E\left(\sum_{t>0}F(\xi(s): 0\leq s <t)g\left(\ln(1-\exp(\Delta_{\hbox{-}} \xi(t)))\right) |1- \exp(\Delta_{\hbox{-}}\xi(t))|^{\omega_+} \exp(\omega_+\xi(t-))\right).
\end{eqnarray*}
We can now compute this expression using the L\'evy-It\^o decomposition of $\xi$ and the compensation formula. We get 
\begin{eqnarray*} & & E\left(\int_0^{\infty} F(\xi(s): 0\leq s <t)\exp(\xi(t)\omega_+ )\,\d t\right) \int_{(-\infty,0)}g\left(\ln(1-\e^x)\right) |1- \e^x|^{\omega_+} \Lambda(\d x)\\
&=& E\left(\int_0^{\infty} F(\xi(s): 0\leq s <t)\exp(\xi(t)\omega_+ )\,\d t\right) \int_{(-\infty,0)}g\left(y\right) \e^{y\omega_+}  \tilde \Lambda(\d y).
\end{eqnarray*}
This shows that under $\widehat {\mathcal P}^+_1$, the variable $\ln {\mathcal X}_{{{\mathcal L}(1)}}(0)-\ln {\mathcal X}_{\varnothing}({{\mathcal L}(1)}-)$ and the process 
$(\ln \hat{\mathcal X}(t))_{0\leq t<b_{{\mathcal L}(1)}}$ are independent. The former has the law $-\Psi(\omega_+)^{-1}\e^{x\omega_+ } \,  \tilde \Lambda(\d x)$, and the latter that of 
$\xi$ killed according to the multiplicative functional  $\exp(\xi(t)\omega_+)$. It is immediately seen (and well-known) that this yields a  L\'evy process 
with Laplace exponent $\Psi^+(q)=\Psi(q+\omega_+)$, and this completes the check of the first assertion. 

For the second assertion about the conditional distribution of the families of growth-fragmentations which stem from the spine, we shall only check that under $\widehat{\mathcal P}^+_1$, conditionally on $(\hat{\mathcal X}(t))_{0\leq t\leq b_{{\mathcal L}(1)}}$, the processes $\hat{\bf X}_{0,j}$ for $ j\geq 1$ are independent, and more precisely, each $\hat{\bf X}_{0,j}$ has the (conditional) law $\P_x$, with $x$ the absolute size of the $j$-th largest negative jump of $\hat{\mathcal X}$ on $[0, b_{{\mathcal L}(1)}]$. Indeed, the more general assertion in the statement then follows easily from the branching property of cell systems. 

Consider an arbitrary functional $F\geq 0$ on the space of finite c\`adl\`ag paths, and for every $j\geq 1$, a functional $G_j\geq 0$ of the space of multi-set valued paths, and for every $x>0$, set
$$g_j(x)\coloneqq \E_x(G_j({\bf X})).$$
We aim at checking that there is the identity 
$$\widehat {\mathcal E}^+_1\left(F( {\mathcal X}_{\varnothing}(s): 0\leq s \leq b_{{{\mathcal L}(1)}})
\prod_{j=1}^{\infty}G_j(\hat{\bf X}_{0,j})\right)=\widehat {\mathcal E}^+_1\left(F( {\mathcal X}_{\varnothing}(s): 0\leq s \leq b_{{{\mathcal L}(1)}})
\prod_{j=1}^{\infty}g_j(D_j(t))\right),$$
where $(D_j(t))_{j\geq 1}$ denotes the sequence formed by the absolute values of the negative jumps 
of ${\mathcal X}_{\varnothing}$ that occurred strictly before time $t$ and the value of ${\mathcal X}_{\varnothing}$ at time $t$, ranked in the non-increasing order.

By the definition of the tagged cell,  the left hand side equals
$$ {\mathcal E}_1\left(\sum_{t>0}F( {\mathcal X}_{\varnothing}(s): 0\leq s \leq t) |\Delta_{\hbox{-}}{\mathcal X}_{\varnothing}(t)|^{\omega_+}\prod_{j=1}^{\infty}G_j(\hat{\bf X}_{0,j})
\right),$$
and then, by the branching property of cell systems under ${\mathcal P}_1$ and the definition of $g_j(x)$, the latter quantity can be expressed as
$${\mathcal E}_1\left(\sum_{t>0}F( {\mathcal X}_{\varnothing}(s): 0\leq s \leq t) |\Delta_{\hbox{-}}{\mathcal X}_{\varnothing}(t)|^{\omega_+}\prod_{j=1}^{\infty}g_j(D_j(t))
\right).$$
Again, by the definition of the tagged cell, we now see that this quantity coincides with
$$\widehat {\mathcal E}^+_1\left(F( {\mathcal X}_{\varnothing}(s): 0\leq s \leq b_{{{\mathcal L}(1)}})
\prod_{j=1}^{\infty}g_j(D_j(t))\right),$$
and this is precisely what we wanted to show. 
\end{proof}

Theorem \ref{T2} states in particular that under $\widehat{\mathcal P}^+_1$, the tagged cell  has the distribution of the self-similar Markov process $Y^+$ with characteristics 
$(\Phi^+,\alpha)$, and this entails  that the lifetime $b_{{\mathcal L}}$ of the tagged cell is infinite $\widehat{\mathcal P}^+_1$-a.s. when $\alpha\leq 0$, whereas 
  $b_{{\mathcal L}}<\infty$  for 
$\alpha > 0$. In both cases $\lim_{t\to b_{\mathcal L}}
\hat {\mathcal X}(t)=\infty$, $\widehat{\mathcal P}^+_1$-a.s. 
Note also that combining Proposition \ref{P2} and Theorem \ref{T2} enables us to recover the many-to-one formula in Theorem \ref{T1}.

\subsection{Starting the growth-fragmentation with indefinite growth from $0$}
In this section, we shall always assume that $\alpha \leq 0$, so the tagged cell does not explode, and we know further from Corollary \ref{C3}(i)  that $M^+$ is a $\P_x$-martingale. We write $\P_x^+$ for the law of the growth-fragmentation ${\bf X}$ under $\widehat{\mathcal P}^+_x$, and note from Lemma \ref{L12} that there is the relation of absolute continuity
$$x^{\omega_+}\E^+_x(A(t))=\E_x(M^+(t) A (t))$$
for every ${\mathcal F}_t$-measurable variable $A(t)\geq 0$. 
This enables us in particular to view ${\bf X}(t)$ under $\P_x^+$ as a random non-increasing null sequence.
We shall now use the spinal decomposition to investigate the asymptotic behavior of $\P^+_x$ as $x\to 0+$.

 Keeping in mind the connection between growth-fragmentations and 
planar random geometry, a motivation for this stems from the study in \cite{BMR} of non-compact scaling limits of uniform random planar quadrangulations with a boundary, in the regime when the perimeter of the boundary is small compared to the size of the quadragulation.

Roughly speaking, the spinal decomposition consists in assigning the role of Eve to the tagged cell $\hat{\mathcal X}$ rather than to ${\mathcal X}_{\varnothing}$ in the description of cell systems.  
Theorem \ref{T2} thus incites us to introduce another distribution for cell-systems, denoted by ${\mathcal Q}^+_x$, which is defined as follows. Under ${\mathcal Q}^+_x$, the Eve cell ${\mathcal X}_{\varnothing}$ follows the law $Q^+_x$ of the self-similar Markov process $Y^+$ with characteristics $(\Phi^+,\alpha)$. 
Recall that $\lim_{t\to \infty} Y^+(t)=\infty$ a.s., so the absolute values of the negative  jumps of $Y^+$ cannot be ranked in the non increasing order. However, this only a minor issue; indeed, we may use for instance the easy fact that $\lim_{t\to \infty} \e^{-t} Y^+(t)=0$ a.s., and then
rank the jumps sizes and times of $t\mapsto -\e^{-t}{\mathcal X}_{\varnothing}(t)$ in the non-increasing order of their sizes, say 
$(\e^{-\beta_1}y_1, \beta_1),(\e^{-\beta_2}y_2, \beta_2), \ldots$. We 
get the initial sizes and birth-times of the daughter cells at the first generation. That is, conditionally  on ${\mathcal X}_{\varnothing}$, the processes 
${\mathcal X}_1, {\mathcal X}_2, \ldots$ are independent 
processes with laws $P_{y_1}, P_{y_2}, \ldots$.  We iterate for the next generations just as in Sec.~\ref{sec:cellsystem}. We stress that only the mother cell evolves according to $Y^+$, whereas the daughters, granddaughters, ... all evolve according to $X$. Theorem \ref{T2} entails in particular that the law $\P_x^+$ of the growth-fragmentation ${\bf X}$ 
is the same under ${\mathcal Q}^+_x$ as under $\hat{\mathcal P}^+_x$ (genealogies are of course different; in other words the cell processes have different laws under ${\mathcal Q}^+_x$ as under $\hat{\mathcal P}^+_x$, even though the growth-fragmentations they induce have the same distribution. This feature has been analyzed in depth by Q. Shi \cite{QShi}).

An important motivation for introducing the law ${\mathcal Q}^+_x$ for cell systems, is that the law $Q^+_x$ of the Eve cell ${\mathcal X}_{\varnothing}$ possesses  a non-degenerate weak limit $Q^+_0$ as $x\to 0+$. More precisely, since $\Phi^+(0)=0$, $(\Phi^+)'(0)>0$ and $\alpha<0$, the family of laws of Markov processes  $(Q^+_x)_{x\geq 0}$ fulfills the Feller property,  see e.g. \cite{BeYo0}. It follows readily that  as $x\to 0+$, ${\mathcal Q}^+_x$ converges weakly, in the sense of finite-dimensional distributions for families indexed by $\U$,  towards ${\mathcal Q}^+_0$, the law of the cell system under which the Eve cell ${\mathcal X}_{\varnothing}$ has the law $Q^+_0$ and all the other cells are self-similar Markov processes with characteristics $(\Psi, \alpha)$. 

We shall now show that under ${\mathcal Q}^+_0$, the multisets ${\bf X}(t)$ can be still viewed as non-increasing null sequences, a.s. Recall that for every $q>0$ and every multiset ${\bf x}
=\left\{\!\!\left\{x_i: i\in\N \right\}\!\!\right\}$, we write
$F_q({\bf x})\coloneqq \sum_{i\in\N} x_i^q$ for the sum of the elements of ${\bf x}$ raised to the power $q$ and repeated according to their multiplicity. 

\begin{lemma}\label{L8} Assume $\alpha <0$. For every $t\geq 0$ and   $q>0$ with $\kappa(q)\leq 0$, we have $F_{q}({\bf X}(t))<\infty$, ${\mathcal Q}^+_0$-a.s. 
\end{lemma}

\begin{proof} Consider  the (multi-)set formed by the cells alive at time $t$ once the Eve cell has been removed, ${\bf X}^*(t)\coloneqq {\bf X}(t)\backslash \{{\mathcal X}_{\varnothing}(t)\}$. Since  $F_{q}({\bf X}(t))$ is a $\P_x$-supermartingale, the construction of ${\mathcal Q}^+_0$ entails$${\mathcal Q}^+_0\left( F_{q}({\bf X}^*(t)) \mid {\mathcal X}_{\varnothing}\right)
\leq \sum_{0< s < t}|\Delta_{\hbox{-}} {\mathcal X}_{\varnothing}(s)|^q.$$
To complete the proof, we just need to check that $\sum_{0< s < t}|\Delta_{\hbox{-}} Y^+(s)|^q<\infty$, $Q^+_0$-a.s.

Just as in the proof of Lemma \label{L5'}(i), one can show using Lemma \ref{Lphi+}(i) that if we set 
$$c^+(q)\coloneqq \int_{(-\infty,0)} (1-\e^x)^q \e^{x\omega_+ }\left( \Lambda(\d x)+ \tilde \Lambda(\d x)\right) <\infty,$$
then the predictable compensator  under $Q^+_0$ of $t \mapsto \sum_{0< s \leq t}|\Delta_{\hbox{-}} Y^+(s)|^q$
is $ t\mapsto c^+(q) \int_0^t (Y^+(t))^{q+\alpha}\d s$.

 That the latter is indeed a well-defined (i.e. with finite values) process for every $q>0$ is easy and can be seen e.g. from the results of Chaumont and Pardo \cite{ChPa}. This entails that $F_{q}({\bf X}^*(t))<\infty$ ${\mathcal Q}^+_0$-a.s.
\end{proof}

Lemma \ref{L8} enables us to rank the elements of ${\bf X}(t)$ in the non-increasing order, and we then obtain a sequence in $\ell^q(\N)$, 
${\mathcal Q}^+_0$-a.s.  We write $\P^+_0$ for the law of ${\bf X}$ under ${\mathcal Q}^+_0$, and claim the following
limit theorem in the spectrally negative case (we believe that this restriction should be essentially superfluous, however it makes the argument somewhat simpler). 

\begin{corollary} \label{C8} Assume $\alpha <0$ and that cells have no positive jumps a.s., that is $\Lambda((0,\infty))=0$. For any $q>0$ with $\kappa(q)\leq 0$,
$\P^+_x$ converges  weakly as $x\to 0$ towards  $\P^+_0$, in the sense of finite dimensional distributions for processes with values in $\ell^q(\N)$.
\end{corollary}

\begin{proof} We work under ${\mathcal Q}^+_0$. Introduce  the first passage times $T_x\coloneqq \inf\{t\geq 0: {\mathcal X}_{\varnothing}(t)=x\}$ for every $x>0$ and observe  that $T_x<\infty$ and $\lim_{x\to 0} T_x=0$ a.s. (note that we use here the assumption of absence of positive jumps). 
Imagine that we kill all the daughter cells ${\mathcal X}_i$ for  $i\in\N$ which are born before $T_x$ together with their descent,  and let ${\bf X}^{(x)}(t)$ be the family of surviving cells which are alive at time $t$. By the Markov property of the Eve cell, we see that the shifted process $({\bf X}^{(x)}(t+T_x))_{t\geq 0}$ has the law $\P^+_x$. 

Second, for every $q>0$ with $\kappa(q)< 0$, the same argument as in the proof of Lemma \ref{L8} yields  
$${\mathcal Q}^+_0\left( F_{q}({\bf X}(t)\backslash {\bf X}^{(x)}(t)) \mid {\mathcal X}_{\varnothing}\right)
\leq \sum_{0< s < T_x\wedge t}|\Delta_{\hbox{-}} {\mathcal X}_{\varnothing}(s)|^q,$$
from which we infer 
$$\lim_{x\to 0+} F_{q}({\bf X}(t)\backslash {\bf X}^{(x)}(t))=0\qquad \hbox{a.s.}$$
Third, we recall from Lemma \ref{L8} that the family ${\bf X}(t)$ is $q$-summable a.s. The branching property enables us to apply the same argument as in the proof of Corollary 4 in \cite{BeGF}; this yields
$$\lim_{x\to 0+}\sum_{u\in \U}|{\mathbbm{1}}_{t\geq b_u}{\mathcal X}_u(t-b_u)-{\mathbbm{1}}_{t+T_x\geq b_u}{\mathcal X}_u(t+T_x-b_u)|^q=0\qquad \hbox{in probability}$$
and the statement now follows easily. \end{proof}

  Here is  an application  of the preceding result to the extinction time. Specifically, recall that when $\alpha <0$, the growth-fragmentation ${\bf X}$ under $ \mathbb{P}_{x}$ for $x>0$ is absorbed at $(0,0\ldots)$ after an a.s.  finite time ${\epsilon}\coloneqq \inf\{t\geq 0: {\bf X}(t)\equiv 0\}$; see Corollary 3 in \cite{BeGF}. We now obtain the following polynomial lower-bound estimates for the tail distribution of this absorption time, which contrasts sharply with the exponential decay proved by Haas for self-similar (pure) fragmentation processes
  (see Proposition 14 in \cite{Haas}). We use the notation $f\asymp g$ for a pair of functions $f,g: \R_+\to (0,\infty)$ such that the ratio $f(t)/g(t)$ remains bounded away from $0$ and $\infty$ as $t\to \infty$.
  For the sake of simplicity, we concentrate on the spectrally negative case, although this restriction is probably superfluous. 

\begin{corollary}\label{C4} 
 Assume $\alpha <0$ and that cells have no positive jumps a.s., that is $\Lambda((0,\infty))=0$. 
Then  we have $\P_1({\epsilon}>t)\asymp t^{\omega_{+}/\alpha}$.

\end{corollary}
\begin{proof} We shall first prove the lower bound
$$ \liminf_{t\to \infty} t^{-\omega_{+}/\alpha} \P_1({\epsilon}>t)\geq \E^+_{0} (1/M^+(1)) >0.$$

We use self-similarity and write $\P_1({\epsilon}>t)=\P_{t^{1/\alpha}}({\epsilon}>1)$. Then
recalling that extinction does not occur $\widehat {\mathcal P}^+_x$-a.s., we have from Proposition \ref{P2}  that 
$$t^{-\omega_+/\alpha}\P_1({\epsilon}>t)= \widehat {\mathcal E}^+_{t^{1/\alpha}} (1/M^+(1))= \E^+_{t^{1/\alpha}} (1/M^+(1)).$$

Corollary \ref{C8} shows that the law of $M^+(t)$ under $\P^+_{t^{1/\alpha}}$ converges to that of  $M^+(t)$ under $\P^+_0$ as   $t\to \infty$. Then, for any $a>0$, we have
$$\lim_{t\to\infty} \E^+_{t^{1/\alpha}} (a\wedge(1/M^+(1))) = \E^+_{0} (a\wedge(1/M^+(1)))$$
and the right-hand side is a strictly positive quantity since we know from Lemma \ref{L8} that $M^+(1)<\infty$ $\P^+_0$-a.s.

In order to establish the upper-bound
$$ \limsup_{t\to \infty} t^{-\omega_{+}/\alpha} \P_1({\epsilon}>t)<\infty,$$
we need first to consider the homogeneous growth-fragmentation which corresponds to taking $\alpha=0$. Specifically, we write
$X^{(0)}(t)=\exp(\xi(t))$ for the self-similar Markov process with characteristics $(\Psi, 0)$, and then ${\mathcal X}^{(0)}$ and ${\bf X}^{(0)}$ for the corresponding cell-system and homogeneous growth-fragmentation. It has been shown in the proof of Corollary 3 in \cite{BeGF} that the growth-fragmentations ${\bf X}$ and ${\bf X}^{(0)}$ can be constructed simultaneously, so that for every $q>0$ with $\kappa(q)<0$, there is the following upper bound
for the extinction time  of the growth-fragmentation ${\bf X}$:
$$\epsilon \leq c(q) S^{-\alpha/q},$$
where $c(q)>0$ is some constant,
$$S\coloneqq \sup\left\{\e^{-t\kappa(q)} (X_*^{(0)}(t))^q: t\geq 0 \right\},$$
and $X_*^{(0)}(t)$ denotes the size of the largest cell  in the family ${\bf X}^{(0)}(t)$. 

Recall from Proposition 3 of \cite{BeGF} that ${\bf X}^{(0)}(t)=
(X^{(0)}_1(t), \ldots)$ can be viewed as a compensated-fragmentation process  (see Definition 3  in \cite{BeCF}), and then from Corollary 3 in \cite{BeCF} (it is assumed in \cite{BeCF} that $q\geq 2$, but actually only $\kappa(q)<\infty$ is needed) that the process
$$\exp(-t \kappa(q))\sum_{i=1}^{\infty} (X^{(0)}_i(t))^{q}\,,\qquad t\geq 0$$
is a martingale.  Doob's maximal inequality yields  
$\P\left( S>x\right)=O( x^{-1})$,  thus 
$\P_1(\epsilon>t)\leq c t^{q/\alpha}$ for some constant $c<\infty$, and then by self-similarity, 
$$\P_x(\epsilon>t)\leq c t^{q/\alpha}  x^q.$$

To finish the proof, we simply observe from an application of the branching property of growth-fragmentation processes and the upper-bound above, that
$$\P_1(\epsilon>2t)= \E_1(\P_1(\epsilon>2t \mid {\mathcal F}_t)) \leq c  t^{q/\alpha}\E_1\left(\sum_{i=1}^{\infty} X^q_i(t)\right).$$
The same argument as in the proof of Corollary \ref{C3}(ii), now  using Theorem 5 in \cite{ArRi} instead of Lemma 4 of \cite{Riv}, shows that $\E_1\left(\sum_{i=1}^{\infty} X^q_i(t)\right)
\leq c(q) t^{(\omega_+-q)/\alpha}$, which completes the proof.

\end{proof}

\noindent{\bf Remark.}  We conjecture that, as a matter of facts, 
$$ \lim_{t\to \infty} t^{-\omega_{+}/\alpha} \P_1({\epsilon}>t)= \E^+_{0} (1/M^+(1)),$$ 
but we do not have a rigorous proof of this. We note further that the arguments above show rather indirectly that
$\E^+_{0} (1/M^+(t)) \in(0,\infty)$ for all $t>0$, a fact which does not seem to follow straightforwardly from the construction of ${\mathcal Q}^+_0$. More precisely, it is then easy to deduce that the process 
$(1/M^+(t): t>0)$ is a $\P^+_0$ supermartingale.

\subsection{Tagging a cell randomly according to the intrinsic area}
\label{sec:area}

Throughout this section, we assume that Cram\'er's hypothesis \eqref{omega-} holds. 
Similarly to Sec.~\ref{sec:plus}, but with simplifications due to the uniform integrability of ${\mathcal M}^-$ (see Lemma \ref{L5'}), we introduce the probability measure $\widehat{\mathcal P}^-_x$ describing the joint law of a tagged leaf ${\mathcal L}$ on $\partial \U$ 
and a cell system ${\mathcal X}=({\mathcal X}_u: u\in\U)$. The law of $({\mathcal X}_u: u\in\U)$ under 
$\widehat{\mathcal P}^-_x$  is absolutely continuous with respect to  ${\mathcal P}_x$  with density $x^{-\omega_{-}}{\mathcal M}^-(\infty)$, and conditionally on $({\mathcal X}_u: u\in\U)$, the random leaf ${\mathcal L}$
has the law ${\mathcal A}(\cdot) /{\mathcal M}^-(\infty)$.
Theorem \ref{T3} entails in particular that for $\alpha \geq 0$, the distribution $\P^-_x$ of the growth-fragmentation process ${\bf X}$ under $\widehat{\mathcal P}^-_x$ is absolutely continuous with respect to $\P_x$, with density
$x^{-\omega_-}M^-(\infty)$. 
 We use the same notation as that introduced in Sec.~\ref{sec:plus} and state without proof analogous results.

\begin{proposition} \label{P3} 
 For every $t\geq 0$, every measurable function $f: [0,\infty)\to [0,\infty)$, and every $\bar {\mathcal F}_t$-measurable random variable $B(t)\geq 0$, we have
$${\mathcal E}_x\left(\sum_{i=1}^{\infty} X^{\omega_{-}}_i(t)f(X_i(t))B(t) \right) = x^{\omega_{-}} \widehat{\mathcal E}^-_x(f(\hat{\mathcal X}(t))B(t)),$$
with the convention that $f(\partial)=0$.  
\end{proposition}

In turn, the  spinal decomposition has now the following form. The function $\Phi^-(q)\coloneqq \kappa(q+\omega_-)$ is the Laplace exponent of a L\'evy process $\eta^-$, whose L\'evy measure can be expressed in the form $\e^{\omega_-x}(\Lambda (\d x)+ \tilde \Lambda(\d x))$. This enables us to mark the negative jumps of $\eta^-$ just as in Sec.~\ref{sec:plus}, and the number $N^-(t)$ of marked jumps up to time $t$ is a Poisson process with intensity $-\Psi(\omega_-)$. We then introduce the law $Q^-_x$ of the self-similar Markov process $Y^-$ with characteristics $(\Phi^-,\alpha)$ started from $x$. In the obvious notation (in particular $I^-\coloneqq \int_{0}^{\infty}\exp(-\alpha \eta^-(s))\d s$ and $\tau^-_t$ refers to the Lamperti's time substitution in this setting), we arrive at:

\begin{theorem}\label{T9}  The distribution of $(\hat{\mathcal X}(t), n_t)_{0\leq t< b_{\mathcal L}}$ under $\widehat{\mathcal P}^-_1$ is the same as that of 
$$(Y^-(t), N^-(\tau^-_t))_{0\leq t < I^-}.$$
Further, conditionally on $(\hat{\mathcal X}(t), n_t)_{0\leq t< b_{\mathcal L}}$,
the processes $\hat{\bf X}_{n,j}$ for $ n\geq 0$ and $ j\geq 1$ are independent, and more precisely, each $\hat{\bf X}_{n,j}$ has the (conditional) law $\P_x$, with $x$ the absolute size of the negative jump of $\hat{\mathcal X}$ with label $(n,j)$. 
\end{theorem} 

An important difference compared to Sec.~\ref{sec:plus} is that now the L\'evy process $\eta^-$ drifts to $-\infty$, so $Y^-$ is absorbed at $\partial$ after an a.s. finite time if $\alpha <0$ (and more precisely, $Y^-(\zeta-)=0$  $Q^-_x$-a.s.), or has an infinite lifetime and converges to $0$ at infinity if $\alpha\geq 0$. In both cases, we can rank the negative jumps of $Y^-$ in the decreasing order, and introduce the law ${\mathcal Q}^-_x$ of the cell system in which the Eve cell evolves according to $Y^-$ and all the other cells according to $X$. The spinal decomposition then shows that the distribution of the growth-fragmentation ${\bf X}$ under ${\mathcal Q}^-_x$ is $\P^-_x$ (recall that when $\alpha \geq 0$, the latter is absolutely continuous with respect to $\P_x$, with density $x^{-\omega_-}M^-(\infty)$).

\section{A distinguished  
family of growth-fragmentations}
\label{sec:gen}

The goal of this section is twofold. First, we give different ways to characterize the law of a self-similar growth-fragmentation. Recall that by \cite{QShi}, the law of a self-similar growth-fragmentation is characterized by a pair $(\kappa,\alpha)$ where $\kappa$ is the cumulant function associated with a driving L\'evy process and $\alpha \in \R$ is the self-similarity parameter. Theorem \ref{thm:kappaphi} shows that the cumulant function $\kappa$ is characterized by one of the following quantities:
\begin{itemize}
\item any shift of the cumulant function $\kappa$, that is any function $q \mapsto \kappa(\omega+q)$ for fixed $\omega \geq 0$. As a consequence, the law of the process $Y^{+}$ (or of $Y^{-}$) introduced in Sec.~\ref{sec:plus} (resp.~Sec.\ref{sec:area}) characterizes the law of a self-similar growth-fragmentation;
\item a particular cell process, which  roughly speaking describes the evolution of the size of the locally largest fragment (that is the cell obtained by following the largest fragment at each splitting).
\end{itemize}
  These are analytical results of independent interest.

Using this, we identify a remarkable one-parameter class of  cumulant functions $\{\kappa_{\theta} : \theta \in (1/2,3/2]\}$ which are closely related to stable L\'evy processes
as well as to random maps (see Sec.~\ref{sec:maps} for the latter connection).  More precisely, for every $ \theta \in ( \frac{1}{2}, \frac{3}{2}]$, we show the existence of a driving L\'evy process (see  \eqref{eq:PsiTheta} for the expression of its Laplace exponent) such that the associated cumulant function  $\kappa_{\theta}$  is\begin{equation}
\label{eq:kappatheta}\kappa_{\theta}(q) = \frac{\cos(\pi(q-\theta)) }{\sin(\pi(q-2\theta))} \cdot \frac{\Gamma (q-\theta)}{\Gamma (q-2\theta)}, \qquad \theta<q<2\theta+1.
\end{equation}
In the notation of Sec.~\ref{sec:many-to-one}, $\kappa_{\theta}$ satisfies Cram\'er's condition with \begin{equation}
\label{eq:cramertheta}\omega_{-}=\theta+1/2, \qquad  \omega_{+}=\theta+3/2
\end{equation} and we will see below that $\Phi^-(q) = \kappa(\omega_{-}+q)$  and $\Phi^+( q)=\kappa(\omega_{+}+q)$ are, up to  scaling constants, the Laplace exponents  of the L\'evy process appearing in the Lamperti representation of a strictly $\theta$-stable L\'evy process with positivity parameter $\rho$ conditioned to die continuously at $0$, resp.~to stay positive, and such that   \begin{eqnarray} \label{eq:positivity}\theta(1-\rho)=1/2. \end{eqnarray}
We refer to \cite{Cha96,CC08} for a definition of these processes.  In other words, the process  $Y^{+}$ and $Y^-$ defined in Sec.~\ref{sec:plus}, \ref{sec:area} and associated with the self-similar growth-fragmentation characterized by the pair $(\kappa_{\theta},-\theta)$  are distributed as the $\theta$-stable L\'evy processes with positivity parameter $\rho$ conditioned to stay positive, resp.~conditioned to die continuously at $0$.

In the particular case $\theta=3/2$, note that there are no positive jumps and $\kappa_{3/2}(q)= \frac{\Gamma(q-3/2)}{\Gamma(q-3)}$ so that, up to a scaling constant, the  growth-fragmentation characterized by the pair $(\kappa_{3/2},-1/2)$ is the one that has been considered in \cite{BCK} (see in particular Eq.~(32) in \cite{BCK}). Using the connection with random maps established in Sec.~\ref{sec:maps} we will be able to identify the law of  the intrinsic area of these growth-fragmentations with cumulant function $\kappa_{\theta}$, which turn out to be size-biased stable distributions (Corollary \ref{cor:area}).

\subsection{Characterizing the cumulant function of a growth-fragmentation
}

If $\kappa$ is a cumulant function of a growth-fragmentation and satisfies \eqref{omega+}, we have seen in Lemma \ref{Lphi+} that $q \mapsto \kappa(w+q)$ is the Laplace exponent of a (possibly killed) L\'evy process provided that $\kappa(w)\leq 0$ (this lemma is actually stated for $w=\omega_{+}$, but the argument is the same).  Conversely, if  $\Phi$ is the Laplace exponent of a (possibly killed) L\'evy process, it is natural to ask if there exists a growth-fragmentation with cumulant function  $\kappa$  and a value $w$ such that $\Phi(\cdot)=\kappa(w+\cdot)$, and in this case if $w$ is uniquely determined. It turns out that such a growth-fragmentation does not always exist, and we now provide a necessary and sufficient condition for this to hold.

\begin{theorem}\label{thm:kappaphi}Let $\Phi$ be a Laplace exponent of a (possibly killed) L\'evy process written in the form 
\begin{equation}
\label{eq:defPhi}\Phi(q) =  \Phi(0)+ \frac{1}{2}s^2q^2 + b' q +\int_{\R}\left( \e^{qy}-1+q(1-\e^y)\right) \Lambda(\d y)\,,\qquad q\geq 0,
\end{equation}
where $\sigma^2\geq 0$, $b\in\R$, $\Phi(0) \leq 0$  and $\Lambda$ is a measure on $\R$ such that 
$\int(1\wedge y^2)\Lambda(\d y)<\infty$ and\footnote{The condition $\int_{y>1}\e^{y}\Lambda(\d y)<\infty$ may be replaced by the weaker condition that there exists $q>0$ such that  $\int_{y>1}\e^{qy}\Lambda(\d y)<\infty$, and by considering an additional cutoff in \eqref{eq:defPhi} but we shall not enter such considerations.} $\int_{y>1}\e^{y}\Lambda(\d y)<\infty$.
The following two assertions are equivalent:
\begin{enumerate}
\item[(i)] 
There exists a unique quintuple $(\sigma^{2}, b , {\tt k} , w, \Lambda_{L})$ with  $\sigma^2\geq 0$, $b\in\R$, ${\tt k}\geq 0$, $w \geq 0$ and $\Lambda_{L}$ a measure on $(-\ln(2),\infty)$ with $\int_{(-\ln(2), \infty)}(1\wedge y^2)\Lambda_{L}(\d y)<\infty$
such that we have  $$\kappa(w+q)=\Phi(q), \qquad \forall  q \geq 0,$$ where we have put
\begin{equation}
\label{eq:defkappa}\kappa(q)=-{\tt k}+ \frac{1}{2}\sigma^2q^2 + bq +\int_{(- \ln(2),\infty)}\left( \e^{qy}-1+ (1-\e^{y})^q \mathbbm{1}_{y<0}+q(1-\e^{y}) \right) \Lambda_{L}(\d y).
\end{equation} 
\item[(ii)] Let $\nu$ be the image of the measure $\Lambda$ by the map $x \mapsto \e^{x}$. There exists a unique $w \geq 0$ such that the measure $x^{-w} \nu(\d x)$, restricted to $(0,1)$, is symmetric with respect to $1/2$. Set
\begin{equation}
\label{eq:f}f(q)= \Phi(q-w)- \int_{(0,\infty)}\left( \e^{qy}-1+q(1-\e^{y})\right) \e^{-w y} \Lambda(\d y).\end{equation}
Then $f(q)$ is well defined and finite for every $q \in [1,w+1]$, the function $f$ admits an analytic continuation to $[1,+\infty)$, and we have \begin{equation}
\label{eq:ineqPhi}15 f(4) \geq 6f(5)+10 f(3).
\end{equation}
\end{enumerate}
In addition, when these assertions are satisfied, $s^{2}=\sigma^{2}$, $\Lambda_{L}(\d y)=\e^{-w y} \Lambda(\d y)$ on $(-\ln(2),\infty)$ and 
$${\tt k}=15 f(4) - 6f(5)-10 f(3).$$ 
In particular, ${\tt k}=0$ if and only if the inequality in $ \eqref{eq:ineqPhi}$ is an equality.
\end{theorem}

Given a growth-fragmentation with cumulant function $\kappa$, Theorem \ref{thm:kappaphi} (and more precisely \eqref{eq:defkappa})  shows that  the self-similar Markov process associated with the  (possibly killed) L\'evy process with Laplace exponent 
\begin{equation}
\label{eq:LL}q \quad  \longmapsto \quad -{\tt k}+ \frac{1}{2}\sigma^2q^2 + bq +\int_{(-\ln(2),\infty)}\left(\e^{qy}-1+q(1-\e^{y}) \right) \Lambda_{L}(\d y)
\end{equation}
may be used as the cell process to construct the growth-fragmentation. Roughly speaking, this self-similar Markov process  describes the evolution of the size of the locally largest fragment (that is the cell obtained by following the largest fragment at each splitting, indeed  $X(t) \geq X(t-)/2$ for every $t>0$ for this self-similar Markov process $X$) in the growth-fragmentation.  In particular, Theorem \ref{thm:kappaphi} gives a means to analytically identify this process.

Theorem \ref{thm:kappaphi} also implies that the laws of the processes $Y^{+}$ and $Y^{-}$ introduced in Sec.~\ref{sec:plus} and \ref{sec:area}  characterize the law of a self-similar growth-fragmentation.

\begin{proof}
We first show that $(i)$ implies $(ii)$. Let $(\sigma^{2}, b , {\tt k} , w, \Lambda_{L})$ be a quintuple such that $(i)$ holds, and let $\nu_{L}$ be the image of the measure $\Lambda_{L}$ by the map $x \mapsto\e^{x}$. Introduce  the symmetrized measure $\overline{\nu}_{L}$ of $\nu_{L}$ on $(0,1)$ defined by
$$\int_{(0,1)} g(x) \overline{\nu}_{L}(\d x) = \int_{(1/2,1)} (g(x)+g(1-x)) \nu_{L}(\d x)$$
for every nonnegative measurable function $g$. By using the definition of $\kappa$, a straightforward computation yields the  equality
$$\kappa(w+q)=\kappa(w)+ \frac{1}{2} \sigma^{2} q^{2} + b_{1} q+\int_{(0,1)}\left( x^{q}-1+q(1-x)\right) x^{\omega} \overline{\nu}_{L}(\d x)+\int_{(1,\infty)}\left( x^{q}-1+q(1-x)\right) x^{\omega} {\nu}_{L}(\d x)$$
for every $q  \geq 0$,
with a certain $b_{1} \in \R$.
Since by hypothesis $\kappa(w+q)=\Phi(q)$ for every $q \geq 0$,  the L\'evy-Khintchin formula implies that $\kappa(w)=\Phi(0)$, $x^{\omega}  \overline{\nu}_{L}(\d x)=\nu(\d x)$ on $(0,1)$ and that $x^{\omega}  {\nu}_{L}(\d x)=\nu(\d x)$ on $(1,\infty)$. In particular, on $(0,1)$, $x^{-\omega} \nu(\d x)$ is symmetric with respect to $1/2$. 

We now check that
$$15 f(4) - 6f(5)-10 f(3)={\tt k}$$
and \eqref{eq:ineqPhi} will follow. To this end,  first notice $f(q)< \infty$ for every $q \in [1,w+1]$ since $\int_{y>1}\e^{y}\Lambda(\d y)<\infty$. In addition, we have
\begin{eqnarray*}
f(q)&=& \kappa(q)-  \int_{(0,\infty)}\left(\e^{qy}-1+q(1-\e^{y})\right)  \Lambda_{L}(\d y)\\
&=& -{\tt k}+ \frac{1}{2} \sigma^{2} q^{2} + b q+\int_{\R_{-}}\left(\e^{qy}+(1-\e^{y})^{q}-1+q(1-\e^{y})\right) {\Lambda_{L}}(\d y).
\end{eqnarray*}
so that $f$ admits indeed an analytic continuation on $[1,\infty)$. Set $P_{q}(x)=x^{q}+(1-x)^{q}-1+q(1-x)$ for every $q, x \in \mathbb{R}$. Since for every $x \in \mathbb{R}$ we have $2P_{3}(x)-3P_{2}(x)=0$ and $5 P_{4}(x)-2P_{5}(x)-5P_{2}(x)=0$, it follows  that
$$2 f(3)-3f(2)={\tt k}+3 \sigma^{2} \qquad \textrm{and} \qquad 5 f(4)-2 f(5)-5f(2)=2{\tt k}+5\sigma^{2}.$$
As a consequence, $15 f(4) - 6f(5)-10 f(3)={\tt k}$, and the proof of the first implication is complete.

Finally, if $(ii)$ is satisfied, the previous calculations show that $\kappa(q) \coloneqq \Phi(q-w)$ can be written in the form \eqref{eq:defkappa} with  $\sigma^{2}=s^{2}$, $\Lambda_{L}(\d y)=e^{-w y} \Lambda(\d y)$ on $(-\ln(2),\infty)$, 
${\tt k}=15 f(4) - 6f(5)-10 f(3)$ and a certain value of $b \in \R$.\end{proof}

\subsection{A one-parameter family of cumulant functions}
\label{sec:oneparameter}
Recall that if $\kappa$ is defined by \eqref{eqkappa} and satisfies \eqref{omega+}, we have seen in Lemma \ref{Lphi+} that $\Phi^{+}$ is the Laplace exponent of a L\'evy process drifting to $+\infty$. As an application of Theorem \ref{thm:kappaphi}, we exhibit a distinguished family of cumulant functions for which the Laplace exponent $\Phi^+$ (and also $\Phi^-$) takes a particularly simple form. This sheds some new light on a calculation performed by Miller \& Sheffield \cite[Sec.~4]{MS15} (see the Remark \ref{rem:MS} at the end of this section). Let us mention that the proof of Proposition \ref{prop:hyper} is technical since it uses hypergeometric L\'evy processes \cite{KP13} (and hence hypergeometric functions) and may be skipped in first reading.

\begin{proposition}\label{prop:hyper}Assume that $\Phi^{+}$ is the Laplace exponent of a hypergeometric L\'evy process without killing and drifting to $+\infty$, that is
\begin{equation}
\label{eq:phiplus}\Phi^{+}(z)= - \frac{\Gamma(\gamma-z)}{\Gamma(-z)} \cdot \frac{\Gamma( \widehat{\beta}+\widehat{\gamma}+z)}{\Gamma( \widehat{\beta}+z)}.
\end{equation}
with $\widehat{\beta} > 0$, $\gamma, \gc \in (0,1)$. Then the assertions of Theorem \ref{thm:kappaphi} are satisfied if and only if $\bc=1$ and $\gc \in (0,1/2]$. Then $\Phi^{+}$ is the Laplace exponent of the L\'evy process appearing in the Lamperti representation of a strictly $\theta$-stable L\'evy process with positivity parameter $\rho$ conditioned to stay positive with $\theta= \gamma+\gc \in (0,3/2]$ and $\rho=\gamma/\theta$.

In addition, the associated growth-fragmentation has no killing (that is ${\tt k}=0$ in \eqref{eq:defkappa}) if and only if $\theta(1-\rho)=1/2$ (in particular, $1/2 < \theta \leq 3/2$), and then its cumulant function is
$$\kappa_{\theta}(q)= \frac{\cos(\pi(q-\theta)) }{\sin(\pi(q-2\theta))} \cdot \frac{\Gamma (q-\theta)}{\Gamma (q-2\theta)}, \qquad \theta<q<2\theta+1.$$
\end{proposition}

\begin{proof} Let $\nu$ be the image  by the map $x \mapsto\e^{x}$ of the L\'evy measure of the L\'evy process with Laplace exponent given by \eqref{eq:phiplus}. By \cite[Proposition 1]{KP13}, setting $\eta=\widehat{\beta}+\gamma+\widehat{\gamma}$, the density of $\nu$ on $(0,1)$ is $$ \nu(z)=- \frac{\Gamma(\eta)}{\Gamma(\eta-\gamma) \Gamma(-\widehat{\gamma})} z^{\bc+\gc-1} {}_{2}F_{1}(1+\gc,\eta;\eta-\gamma;z), \quad \textrm{with} \quad  {}_{2}F_{1}(a,b;c;z)= \sum_{n \geq 0} \frac{\Gamma{(a+n)} \Gamma(b+n) \Gamma(c)}{\Gamma(a) \Gamma(b) \Gamma(n+c)} z^{n}.$$
Let $w \geq 0$ be such that $z^{-w} \nu(z)$ is symmetric with respect to $1/2$ on $(0,1)$. Then
$$ \frac{{}_{2}F_{1}(1+\gc,\eta;\eta-\gamma;z)}{{}_{2}F_{1}(1+\gc,\eta;\eta-\gamma;1-z)} =  \left(  \frac{z}{1-z} \right) ^{w-\bc-\gc+1}.$$
Using the formula (see e.g. \cite[p.~291]{WW96})
\begin{eqnarray*}
{}_{2}F_{1}(a,b;c;z) &=&  \frac{\Gamma(c) \Gamma(c-a-b)}{\Gamma(c-a) \Gamma(c-b)} {}_{2}F_{1}(a,b;a+b+1-c;1-z) \\
&& \qquad \qquad + \frac{\Gamma(c)\Gamma(a+b-c)}{\Gamma(a)\Gamma(b)} (1-z)^{c-a-b} {}_{2}F_{1}(c-a,c-b;1+c-a-b;1-z)
\end{eqnarray*}
and taking $z \rightarrow 0$,
it is a simple matter to check that this implies that $1+\gc=\eta-\gamma$, so that $ \bc=1$. Setting $\theta=\gamma+\gc$, this also yields $w=\theta+\gc+1$ and ${}_{2}F_{1}(1+\gc,\eta;\eta-\gamma;z)=(1-z)^{-\theta-1}$.   As a consequence, by \cite[Proposition 1]{KP13},  the density of $\nu$ is
$$\nu(z)=  \frac{\Gamma(\theta+1)\sin(\pi \widehat{\gamma})}{\pi} \frac{z^{\widehat{\gamma}}}{(1-z)^{\theta+1}} \mathbbm{1}_{0<z<1}+ \frac{\Gamma(\theta+1)\sin(\pi \gamma)}{\pi} \frac{z^{\theta-\gamma}}{(z-1)^{\theta+1}} \mathbbm{1}_{z>1},$$
so that by \cite[Theorem 1]{KP13},  $\Phi^{+}$ is indeed the Laplace exponent of the L\'evy process appearing in the Lamperti representation of a strictly $\theta$-stable L\'evy process with positivity parameter $\rho=\gamma/\theta$ conditioned to stay positive.

It remains to see if \eqref{eq:ineqPhi} holds, and to this end we calculate  $\int_{(1,\infty)}\left( x^{q}-1+q(1-x) \right) x^{-w} \nu(\d x)$ up to a linear term in $q$. On $(1,\infty)$, we have
$$ z^{-w} \nu(\d z)= \frac{\Gamma(\theta+1)\sin(\pi \gamma)}{\pi} \frac{1}{(z(z-1))^{\theta+1}}= \frac{\Gamma(\theta+1)\sin(\pi \gamma)}{\pi} \frac{z^{-2\theta-2}}{(1-z^{-1})^{\theta+1}}.$$
Since a L\'evy process with L\'evy measure $ z^{-w} \nu(\d z)$ on $(1,\infty)$ belongs to the so-called class of $\beta$-family of L\'evy processes introduced by Kuznetsov \cite{Kuz10}, by 
\cite[Proposition 9]{Kuz10} (with, in the notation of the latter reference, $\beta_{1}=1$, $\lambda_{1}=\theta+1$, $\alpha_{1}=2\theta+1$) we get the existence of $b \in \R$ such that
$$\int_{(1,\infty)}\left( x^{q}-1+q(1-x) \right) x^{-w} \nu(\d x)= - \frac{\sin(\gamma \pi)}{\sin (\theta \pi)} \left(  \frac{\Gamma(2\theta+1-q)}{\Gamma(\theta+1-q)}- \frac{\Gamma(2\theta+1)}{\theta+1} \right)+bq, \quad  0 \leq q <2\theta+1.$$
As a consequence,
$$f(q)= \frac{\sin(\pi(\theta+ \gc-q) )}{\pi} \Gamma(2\theta+1-q) \Gamma(q-\theta)+ \frac{\sin(\gamma \pi)}{\sin (\theta \pi)} \left(  \frac{\Gamma(2\theta+1-q)}{\Gamma(\theta+1-q)}- \frac{\Gamma(2\theta+1)}{\theta+1} \right)-bq, \quad  0 \leq q <2\theta+1.$$
Since $\sigma^{2}=0$, the proof of Theorem \ref{thm:kappaphi} shows that ${\tt k}=2 f(3)-3f(2)$. A straightforward computation then gives
$${\tt k} = \cos(\gc \pi) \cdot \frac{2\Gamma(2\theta)}{\Gamma(\theta)}.$$
Therefore ${\tt k} \geq 0$ if and only if $\gc \in (0,1/2]$, which implies that $\theta \in (0,3/2]$. In addition, ${\tt k}=0$ if and only if $ \theta(1-\rho)=\gc=1/2$. 
\end{proof}

In the particular case of $\kappa_{\theta}$ (with $ \theta \in (1/2,3/2]$), the proof of  Proposition \ref{prop:hyper} actually gives the explicit Laplace exponent \eqref{eq:LL} of the L\'evy process involved in the self-similar Markov process describing the evolution of the locally largest fragment in the associated growth-fragmentation.  Let $\nu_{\theta}$ be the measure on $(1/2,\infty)$ defined by 
$$  \nu_{\theta}(\d x)=\frac{\Gamma(\theta+1)}{\pi} \left(  \frac{1}{(x(1-x))^{\theta+1}} \mathbbm{1}_{1/2<x<1}+ \sin(\pi (\theta-1/2)) \cdot  \frac{1}{(x(x-1))^{\theta+1}} \mathbbm{1}_{x>1} \right) \d x$$
and let $\Lambda_{\theta}$ be the image of $\nu_{\theta}$ by the map $x \mapsto \ln x$. Set
\begin{equation}
\label{eq:PsiTheta}\Psi_{\theta}(q)= \left( \frac{ \Gamma (2-\theta)}{2 \Gamma (2-2 \theta) \sin (\pi \theta)}+\frac{\Gamma (\theta+1) B_{\frac{1}{2}}(-\theta,2-\theta)}{\pi } \right) q +\int_{\R}\left(\e^{qy}-1+q(1-\e^{y}) \right) \Lambda_{\theta}(\d y)
\end{equation}
where  $B_{1/2}(a,b)= \int_{0}^{1/2} t^{a-1}(1-t)^{b-1} \d t$ is the incomplete Beta function.
Then
$$\kappa_{\theta}(q)= \Psi_{\theta}(q)+\int_{(-\infty,0)}(1-\e^{y})^q \Lambda_{\theta}(\d y), \qquad \theta<q<2\theta+1.$$

Indeed, the proof of Proposition \ref{prop:hyper}  shows the
 existence of $b \in \R$ such that
$$ \frac{\cos(\pi(q-\theta)) }{\sin(\pi(q-2\theta))} \cdot \frac{\Gamma (q-\theta)}{\Gamma (q-2\theta)}=  bq+\int_{1/2}^{\infty} \left( x^{q}-1+(1-x)^{q} \mathbbm{1}_{x<1}+q(1-x) \right)  \nu_{\theta}(\d x), \qquad \theta<q<2\theta+1.$$
 To find the value of $b$, we simply specify the latter equality for $q=2$:
$$-\frac{ \Gamma (2-\theta)}{2 \Gamma (2-2 \theta) \sin (\pi  \theta)}= 2b+2 \int_{1/2}^{1} \left( 1-x \right)^{2}  \nu_{\theta}(\d x)+\int_{1}^{\infty} \left(1-x \right)^{2}  \nu_{\theta}(\d x).$$
Since
$$ \int_{1/2}^{1}  \frac{1}{x^{\theta+1} (1-x)^{\theta-1}} \d x=-\frac{\Gamma (1-\theta) \Gamma (2-\theta)}{\theta \Gamma (2-2 \theta)}-B_{\frac{1}{2}}(-\theta,2-\theta)$$
and $$ \int_{1}^{\infty}\frac{1}{x^{\theta+1} (x-1)^{\theta-1}} \d x =    \frac{\Gamma (2-\theta) \Gamma (2 \theta-1)}{\Gamma (\theta+1)},$$
we obtain the announced value of $b$.

 In particular, by taking $\theta=3/2$, since $B_{\frac{1}{2}}(-3/2,1/2)=-8/3$, we get that
$$ \frac{\Gamma(q-3/2)}{\Gamma(q-3)} = - \frac{2}{\sqrt{\pi}} q + \frac{3}{4\sqrt{\pi}} \int_{1/2}^{1}(x^{q}-1+q(1-x)+(1-x)^{q}) \cdot \frac{1}{(x(1-x))^{5/2}}\d x \quad  (q >3/2),$$
which gives a  proof of the identity \cite[Eq.~(32)]{BCK}.

\begin{remark}[Interpretation of a calculation by Miller \& Sheffield]\label{rem:MS} Let us draw a connection between this section and a calculation performed by Miller \& Sheffield \cite[Sec.~4]{MS15}. We first  translate (non-rigorously) their setup in our framework of growth-fragmentations\footnote{More precisely, keeping the notation of \cite[Sec.~4]{MS15}, Theorem 4.6 in \cite{MS15} indicates that under $\widetilde{\mu}^{1,L}_{\textrm{DISK}}$, the process describing the boundary lengths of increasing balls from the root is a self-similar growth-fragmentation under the tilted probability measure $\mathbb{P}^{-}_{L}$. The process $(L_{r})$ is $Y^{-}$ (the size of the tagged fragment under $ \widehat{\mathcal{{P}}}^{-}_{L}$), the process $(M^{1}_{r})$ is the size of the Eve cell when one uses the locally largest cell process for the Eve cell under the tilted probability measure $ \widehat{\mathcal{{P}}}^{-}_{L}$ and the process $(M_{r})$ is the size of the Eve cell when one uses the locally largest cell process for the Eve cell under the non-tilted probability measure $ \widehat{\mathcal{{P}}}_{L}$. Theorem 4.6 in \cite{MS15} indicates that the process $(L_{r})$ evolves as a time-reversed $\theta$-stable continuous state branching process.}: Consider a growth-fragmentation with pair $(\kappa,\alpha)$ such that Cram\'er's hypothesis holds. Assume that the process $Y^{-}$ (in the notation of Sec.~\ref{sec:area}) is the time-reversal of a $\theta$-stable branching process. We claim that necessarily we have $\theta=3/2$. Indeed, if $Y^-$ is a time-reversal of the $\theta$-stable branching process, by standard reversal arguments $Y^-$ is a Lamperti time-change of the $\theta$-stable process with no positive jumps conditioned to die continuously at $0$. Therefore by Theorem \ref{thm:kappaphi} we must have $\kappa=\kappa_{\theta}$ and the only $\theta$ for which the process has no positive jumps is $\theta= 3/2$.
\end{remark}

{\begin{remark} The boundary case $\theta=1/2$ and $ \rho=0$ corresponds to $\kappa_{1/2}(q)= -\frac{\Gamma(q-1/2)}{\Gamma(q-1)}$, which is (up to a constant factor) the cumulant function of the self-similar pure fragmentation that occurs when splitting the Brownian Continuum Random Tree at heights; see \cite{BeSSF} and \cite{UB}.
Roughly speaking, in this situation, we have $\omega_-=1$ whereas $\omega_+$ does not exist, as $\kappa_{1/2}$ is a non-increasing function on $[1,\infty)$. Further,  the function $\Phi_{1/2}^-(q)= \kappa_{1/2}(q+1)$ can be identified
as the Laplace exponent of the L\'evy process appearing in the Lamperti representation of the negative of a $1/2$-stable subordinator killed when it becomes negative, and conditioned to hit $0$; see \cite{UB} for details. Note further that in this case, it would make of course no sense to condition this process to stay positive. 
\end{remark}

\section{Applications to large random planar maps} 
\label{sec:maps}

 In this section, we show that  growth-fragmentations with cumulant function \eqref{eq:kappatheta}  appear in the Markovian explorations of  particular models of random planar maps which are, roughly speaking, the dual maps of the so-called stable maps of Le Gall \& Miermont \cite{LGM09}. In particular,  for $ \theta \in (1,3/2]$, we shall show that when taken with the proper self-similarity index, they describe the scaling limit of the perimeters of cycles obtained by slicing these random maps at all heights. For $\theta=3/2$, this was observed  in \cite{BCK} for random Boltzmann triangulations.  An an application, this link will allow us to identify the law of the intrinsic area of  these growth-fragmentations (as defined in Sec.~\ref{sec:area}). 

\subsection{Critical non-generic Boltzmann planar maps}
\label{sec:boltzmann}

We first present the model of random planar maps we are dealing with. As usual, all planar maps in this work are rooted, i.e.~come with a distinguished oriented edge; for technical simplicity we will  only consider \emph{bipartite} planar maps, that is all faces have even degree. If $ \map$ is a (rooted bipartite) planar map we denote by $\mathsf{Faces}( \map)$ the set of its faces, and by
$\rootface \in \mathsf{Faces}( \map)$ the face adjacent to the right of the root edge. This face is called the root face of the map, and the origin vertex of the root edge is called the origin of the map. The integer $ \mathrm{deg}( \rootface)$ is  the perimeter of $\map$; note that the perimeter of a bipartite map must be even due to the parity constraint. We write $| \map|$ for the total number of vertices of $ \map$. For $\ell \geq 0$ and $n \geq 0$, we denote by $ \mathsf{Map}^{(\ell)}_{n}$ the set of all (rooted bipartite) planar maps  of perimeter $2 \ell$ with $n$ vertices. By convention, $ \mathsf{Map}^{(0)}_{1}$ contains a single ``vertex map''. We finally set $ \mathsf{Map}^{(\ell)}= \cup_{n \geq 0} \mathsf{Map}^{(\ell)}_{n}$. Any planar map with at least one edge can be seen as a planar map with perimeter 2 by simply splitting the root edge into a root face of degree $2$. We shall implicitly make this identification many times in this section.

 Given a non-zero sequence $ \mathbf{q}= (q_{k})_{ k \geq 1}$ of non-negative real numbers, we define a measure $ {\tt w}$ on the set of all (finite) bipartite planar maps by the formula
$${\tt w}(\map) \coloneqq \prod_{f\in\faces(\map)\backslash \{ \rootface\}} q_{\deg(f)/2}, \qquad   \map \in \bigcup_{\ell \geq 0} \mathsf{Map}^{(\ell)}.
$$
We then set
 \begin{eqnarray} W^{(\ell)}_{n} = {\tt w}\big(\mathsf{Map}^{(\ell)}_{n}\big), \qquad W^{(\ell)} = \sum_{n \geq 0}  W_{n}^{(\ell)}   \quad \mbox{ and }  \quad W^{(\ell)}_{\bullet} = \sum_{n \geq 0} n\,W^{(\ell)}_{n},   \label{eq:defWl}\end{eqnarray} where the dependence in $ \mathbf{q}$ is implicit. 
 We  assume that $ \mathbf{q}$ is admissible, meaning that  $W^{(\ell)}_{\bullet}< \infty$ for one value of $\ell \geq 1$ (or, equivalently,  that  $W^{(\ell)}_{\bullet} < \infty$ for every $\ell \geq 1$, see e.g.~\cite{Bud15}). 
As in \cite[Sec.~2.2]{LGM09} and in \cite{BBG12}, we henceforth focus  on the case where the admissible weight sequence $ \mathbf{q}$ is critical i.e.~$\sum_{n\geq 0} n^2 W_n^{(\ell)}=\infty$ for some $\ell>0$? (see \cite[Proposition 4.3]{BCMperco}) and non-generic in the sense that
\begin{eqnarray} \label{eq:asymptoticqk} q_{k}  \quad \mathop{\sim}_{k \rightarrow \infty} \quad  c\, \gamma^{k-1} \, k^{-\theta-1}, \qquad \mbox{ for a certain }  \theta\in \left( \frac{1}{2}, \frac{3}{2}\right) \textrm{ and } c, \gamma>0.  \end{eqnarray} 
The reader should keep in mind that the above weight sequence $ \mathbf{q}$ must be very fine-tuned to achieve non-generic criticality, see \cite{LGM09,Bud15,BBG12}. To avoid these complications, one may decide to work with the following concrete admissible, critical and non-generic weight sequence for fixed $\theta\in \left( \frac{1}{2}, \frac{3}{2}\right)$, see \cite[Sec.~5]{BCgrowth}:
\begin{equation}
\label{eq:explicite}q_{k}=c  \gamma^{k-1} \frac{\Gamma \left(  k-\theta -\frac{1}{2}\right) }{ \Gamma \left( \frac{1}{2}+k \right) } \mathbbm{1}_{k \geq 2}, \qquad \gamma = \frac{1}{4\theta + 2}, \qquad c= \frac{-\sqrt{\pi}}{2 \Gamma( \frac{1}{2}-\theta)}.
\end{equation} 
Note that the values $a$ and $\kappa$ in \cite{Bud15,BCgrowth} are denoted here by respectively $\theta+1$ and $\gamma$.

We denote by $\mathbb{P}^{(\ell)}$  the probability measure $ {\tt w}( \cdot \mid \cdot \in \mathsf{Map}^{(\ell)})$ and we say that a random map with distribution $\mathbb{P}^{(\ell)}$ is a $ \mathbf{q}$-Boltzmann planar map with perimeter $2\ell$. Under our assumptions on $ \mathbf{q}$, the scaling limit of these maps under $ \mathbb{P}^{(1)}$ conditioned on having a fixed large number of vertices, is given (at least along subsequences) by the so-called stable maps of Le Gall \& Miermont \cite{LGM09} (which are random compact metric spaces that  look like randomized versions of the Sierpinski carpet or gasket).

We shall also consider pointed $\mathbf{q}$-Boltzmann planar maps. By definition, a pointed map is a pair $(\map,v)$ where $\map$ is a planar map and $v \in \map$ is a vertex (which is called the distinguished vertex). We denote  by ${\mathbb{P}}^{(\ell)}_{\bullet}$ the probability distribution of the set of all pointed planar maps given by
\begin{equation}
\label{eq:defpbullet}{\mathbb{P}}^{(\ell)}_{\bullet} ((\map,v))=  \frac{{\tt w}(\map) }{W^{(\ell)}_{\bullet}}, \qquad \map \in \mathsf{Map}^{(\ell)}, v \in \map.
\end{equation}
We  say that a random variable with distribution ${\mathbb{P}}^{(\ell)}_{\bullet}$ is a pointed $ \mathbf{q}$-Boltzmann planar map with perimeter $2\ell$.  If $(B^{(\ell)}_{\bullet},v_{\bullet})$ is such a random variable, note that for every $\mathfrak{m} \in \mathsf{Map}^{(\ell)}$, we have $\P(B^{(\ell)}_{\bullet}=\mathfrak{m})=  {|\map|  {\tt w}(\map) }/{W^{(\ell)}_{\bullet}}$, so a pointed $\mathbf{q}$-Boltzmann planar map is a size-biased $ \mathbf{q}$-Boltzmann planar map with perimeter $2\ell$. Under our assumptions on $ \mathbf{q}$, the degree of a typical face of a pointed $ \mathbf{q}$-Boltzmann planar map 
  is in the domain of attraction of a stable law of index $\theta+1/2$ (see \cite[Sec.~3.2]{LGM09})

It is further possible to define an infinite version of a $ \mathbf{q}$-Boltzmann planar map with perimeter $2 \ell$  as the local limit of  $ \mathbf{q}$-Boltzmann planar maps with perimeter $2\ell$ conditioned to have  size tending to $ \infty$, see  \cite[Theorem 6.1]{St14}.  The law of  this infinite version is denoted by $\mathbb{P}^{(\ell)}_{\infty}$, which is a probability measure on the set of all infinite (bipartite) planar maps with perimeter $2 \ell$.

\paragraph{Enumeration.}
We now recall some important enumeration results (we refer to \cite{Bud15} and \cite{BBG12} for proofs, see also \cite{BCgrowth} where they have been gathered under the present form). First, recalling that $\mathbf{q}$ is an admissible, critical and non-generic weight sequence such that \eqref{eq:asymptoticqk}  holds, we read from \cite[Eq. 3.15, Eq. 3.16]{BBG12}  that
 \begin{eqnarray} \label{eq:asymptowl} W^{(\ell)}   \quad \mathop{\sim}_{\ell \rightarrow \infty} \quad   \frac{c}{ 2\cos ( (\theta+1) \,\pi)} \gamma^{-\ell-1}\ell^{-\theta-1}.\end{eqnarray}
The asymptotic behavior of $W^{(\ell)}_{\bullet}$  can be deduced from the following surprisingly universal identity (see  \cite[Eq. (17)]{Bud15})
 \begin{eqnarray} \label{eq:egalbullet}
 \gamma^\ell W_{\bullet}^{(\ell)} = h^\newsearrow(\ell) \coloneqq 2^{-2\ell} \binom{2\ell}{\ell}, \qquad  \ell \geq 1.
 \end{eqnarray}

We shall also need to estimate the asymptotic behavior of the partition functions of size-constrained Boltzmann planar maps. The following result appears in \cite[Sec.~3.3]{BBG12} for $\ell=1$, but we give a proof in the general case for completeness.

\begin{proposition}\label{prop:equivtaille}For every fixed $\ell \geq 1$,
$$ W^{(\ell)}_{n}  \quad \mathop{\sim}_{n \rightarrow \infty} \quad c_{0} \cdot h^\uparrow(\ell) \cdot \gamma^{-\ell}  \cdot n^{-\frac{4(\theta+1)}{2\theta+1}}, \quad \textrm{with } c_{0}= \frac{1}{2|\Gamma ( -  \frac{1}{\theta+1/2}) |} \cdot \left( \frac{2 \gamma \cos(\pi(\theta+1)) \Gamma(\theta+3/2) }{c \sqrt{\pi}}  \right)^{ \frac{1}{\theta+1/2}}.$$
\end{proposition}

\begin{proof}Note that, for every $g \in [0,1]$,
$$\sum_{n \geq 0} n W_{n}^{(\ell)} g^{n} = \sum_{\map \in \mathsf{Map}^{(\ell)}, v \in \map } {\tt w}(\map)g^{|\map|},$$
and that this quantity is equal to $W^{(\ell)}_{\bullet}$ for $g=1$. The discussions just before and after Eq.~(24) in \cite{BCgrowth} show that 
\begin{equation}
\label{eq:g}\sum_{n \geq 0} n W_{n}^{(\ell)} g^{n} = g \cdot x(g)^{\ell} \cdot W_{\bullet}^{(\ell)},
\end{equation}
where $x(g) \in (0,1]$ is the unique solution of
$$g=F(x(g)), \qquad \textrm{where} \qquad  F(x)= \frac{x}{4 \gamma}  \left( 1-\sum_{k=1}^{\infty}  \left( \frac{x}{4\gamma} \right) ^{k-1} \binom{2k-1}{k} q_{k} \right) .$$
Equivalently, we have
$$x(g)=g \cdot \Phi(x(g)),  \textrm{ with }\Phi(z)= \frac{z}{F(z)}.$$
By \eqref{eq:asymptoticqk}, using \cite[VI.19 p407]{FS09} we have that $F(z)=1- \frac{c \Gamma(-1/2-\theta)}{2\gamma\sqrt{\pi}}(1-z)^{\theta+1/2}+o((1-z)^{\theta+1/2})$ as $z \rightarrow 1, |z|<1$, within a cone, so that
$$\Phi(z) =1+ \frac{c \Gamma(-1/2-\theta)}{2\gamma\sqrt{\pi}}(1-z)^{\theta+1/2}+o((1-z)^{\theta+1/2}) \textrm{ as } z \rightarrow 1, |z|<1, \textrm{ within a cone}.$$
By  \cite[VI.18 p407]{FS09}, it follows that
$$[g^{n}] x(g)  \quad \mathop{\sim}_{n \rightarrow \infty} \quad  \frac{1}{|\Gamma ( -  \frac{1}{\theta+1/2}) |} \cdot  \left(  \frac{2\gamma\sqrt{\pi}}{c \Gamma(-1/2-\theta)} \right)^{\frac{1}{\theta+1/2}}   \frac{1}{n^{ \frac{1}{\theta+1/2}+1}}.$$
 The result readily follows from \eqref{eq:g}.
  \end{proof}

\paragraph{Harmonic functions. }The functions  $h^{\newsearrow}$ and $h^{\uparrow}$, which, remarkably, do not depend on the weight sequence $ \mathbf{q}$ will play an important role below, in particular through their relation with a random walk whose step distribution we now define. Let $\nu$ be the probability measure on $ \mathbb{Z}$ defined by 
 \begin{eqnarray} \label{eq:defnu} \nu(k) = \left\{ \begin{array}{ll} q_{k+1} \gamma^{-k} & \mbox{for } k \geq 0 \\
2W^{(-1-k)} \gamma^{-k} & \mbox{for } k \leq -1. \end{array} \right.  \end{eqnarray}
Under our assumptions, $\nu$ is indeed a probability distribution which is centered and in the domain of attraction of the $\theta$-stable law with positivity parameter $\rho$, and which further satisfies \eqref{eq:positivity}. Let $(S_{n})_{n \geq 0}$ be the random walk on $\Z$  with independent increments distributed according to $\nu$. It has been remarked in \cite{Bud15} that $h^{\uparrow}$, with the convention $h^\uparrow(\ell)=0$ when $\ell\leq 0$, is up to a multiplicative constant the unique harmonic function on $\{1,2,3,\ldots\}$ for this random walk (we say that $h^\uparrow$ is $\nu$-harmonic at these points) that vanishes on $\{\ldots,-2,-1,0\}$. This fact has been used in \cite{Bud15} to give an alternative definition of critical weight sequences (see Corollary 1 in \cite{Bud15}). We also note that $h^\newsearrow$ is the 
 discrete derivative of the function $h^\uparrow$, namely
$$h^\newsearrow(\ell) = h^{\uparrow}(\ell+1)-h^\uparrow(\ell), \qquad \ell \geq 0.$$ It readily follows that $h^\newsearrow$ is harmonic on $\{1,2,3,\ldots\}$ and vanishes on $\{\ldots, -2,-1\}$.  By classical results \cite{BD94} we deduce that the $h$-transform of the walk $(S_{n})$ with the harmonic functions $ h^\uparrow$ (resp.\,$h^\newsearrow$) can be interpreted as the walk $(S_{n})$ conditioned to stay positive (resp.\,to be absorbed at $0$ without touching $ \mathbb{Z}^-$). This fact should be reminiscent of the various transformations performed in Sec.~\ref{sec:spine1}.

\subsection{Edge-peeling explorations}

As we said above, the growth-fragmentation processes of Sec.~\ref{sec:oneparameter} with cumulant function $\kappa_{\theta}$ will appear as a scaling limit of the perimeters of the holes encountered in the Markovian explorations of $ \mathbf{q}$-Boltzmann random planar maps. By exploration we mean the so-called ``lazy peeling process'' introduced in \cite{Bud15}.  We first present the branching peeling exploration (using the presentation of \cite{BCgrowth}) in a deterministic setting and move to random maps in the next section.

\subsubsection{Submaps in the primal and dual maps}
Let $\map$ be a (bipartite rooted) planar map. A submap $ \mathfrak{e} \subset \mathfrak{m}$ of the map $ \mathfrak{m}$ is a finite planar map with a certain number of distinguished faces with simple boundary (no pinch points) called the holes of $ \mathfrak{e}$ and which can be ``filled-in'' with appropriate planar maps with a (general) boundary so as to recover $ \mathfrak{m}$. More precisely if $ h_{1}, \ldots ,h_{k}  \in  \mathsf{Faces}(  \mathfrak{e})$ are the holes of $ \mathfrak{e}$ then there exist planar maps $ \mathfrak{u}_{1}, ... ,  \mathfrak{u}_{k}$ whose perimeters match those of $h_{1}, ... , h_{k}$ and such that the gluings of $ \mathfrak{u}_{i}$ inside $h_{i}$ give $ \mathfrak{m}$. 

To perform this gluing operation, we implicitly assume that an oriented edge has been distinguished (e.g.~using a deterministic rule) on the boundary of each hole $ h_{i}$ of $ \mathfrak{e}$, on which we glue the root edge of $ \mathfrak{u}_{i}$. Notice that after this gluing operation, it might happen that several edges on the boundary  of a given hole of $ \mathfrak{e}$ get identified because the boundary of $ \mathfrak{u}_{i}$ may not be simple, see Fig.\,\ref{fig:gluing} below. We will alternatively speak of ``gluing'' as ``filling-in the hole''.

\begin{figure}[!h]
 \begin{center}
 \includegraphics[width=12cm]{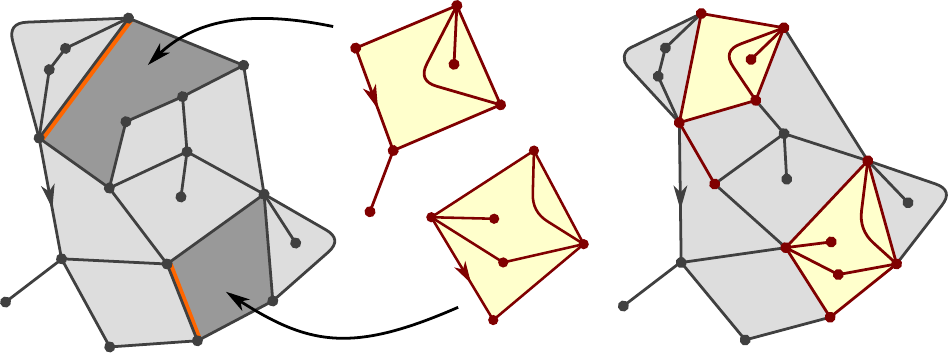}
 \caption{ \label{fig:gluing} An illustration of the filling-in of holes: The left figure shows a planar map with two holes (the darker faces) equipped with distinguished edges on their boundary. The right figure shows the result of filling-in the holes with the pair of maps shown in the middle.}
 \end{center}
 \end{figure}
It is easy to see that this operation is rigid (see \cite[Definition 4.7]{AS03}) in the sense that if $ \mathfrak{e} \subset \map$, then the maps $ (\mathfrak{u}_{i})_{1 \leq i \leq k}$ are uniquely defined (in other words, if one glues different maps inside a given planar map with holes, one gets different maps after the gluing procedure). This definition even makes sense when $ \mathfrak{e}$ is a finite map and $ \map$ is an infinite map.

\subsubsection{Branching edge-peeling explorations}
\label{sec:br}

We now define the branching edge-peeling exploration, which is a means to explore a planar map edge after edge. If $ \mathfrak{e}$ is a planar map with holes, a cycle\footnote{Contrary to \cite{BCK}, the cycles cannot be seen as self-avoiding loops on the original map $\map$ since they are closed paths which may visit twice the same edge; they are called frontiers in \cite{Bud15}. } of $\mathfrak{e}$ is by definition a connected subset of edges adjacent to a hole of $ \mathfrak{e}$. We denote by $ \mathcal{C}( \mathfrak{e})$ the union of the cycles of $ \mathfrak{e}$.  Formally, a branching peeling exploration depends on a function $ \mathcal{A}$, called the \emph{peeling algorithm}, which associates with any planar map with holes $  \mathfrak{e}$  an edge of $ \mathcal{C}( \mathfrak{e}) \cup \{\dagger\}$, where $\dagger$ is a cemetery point which we interpret as ending the exploration. In particular, if $  \mathfrak{e}$ has no holes, we must have $ \mathcal{A}(  \mathfrak{e}) = \dagger$. We say that this peeling algorithm is deterministic, meaning that no randomness is involved in the definition of $ \mathcal{A}$. 

 Intuitively speaking, given the peeling algorithm $ \mathcal{A}$, the branching edge-peeling process of a planar map $\map$ is a way to iteratively explore $\map$ starting from its boundary and  discovering at each step a new edge by \emph{peeling an edge}  determined by the algorithm $ \mathcal{A}$. If $\mathfrak{e} \subset \map$ is a planar map with holes and $e$ is an edge belonging to a cycle  $\mathscr{C}$ of $\mathfrak{e}$, the planar map with holes $\mathfrak{e}_{e}$ obtained by peeling $e$ is defined as follows. Let $\mathsf{F}_{e}$ be the face of $\map$ that is adjacent to the same side of $e$ as the hole in $\mathfrak{e}$. Then there are two possibilities, see Fig.\,\ref{fig:peel}:

\begin{itemize}
\item \emph{Event $ \mathsf{C}_{k}$}: the face $\mathsf{F}_{e}$ is not a face of $\mathfrak{e}$ and has degree $2k$. Then  $\mathfrak{e}_{e}$ is obtained by gluing $\mathsf{F}_{e}$ on $e$.

\item  \textit{Event $\mathsf{G}_{k_{1},k_{2}}$}: the face $\mathsf{F}_{e}$ is actually a face of $\mathfrak{e}$. In this case, the edge $e$ is identified in $\map$ with another edge $e'$ of the same cycle $ \mathcal{C}$ where $2k_{1}$ (resp.~$2k_{2}$) is the number of edges of $ \mathcal{C}$ strictly between $e$ and $e'$ when turning in clockwise order around the cycle, and $\mathfrak{e}_{e}$ is the map after this identification in $\mathfrak{e}$.
\end{itemize}

When $k_{1}>0$ and $k_{2}>0$, note that the event $\mathsf{G}_{k_{1},k_{2}}$ results in the splitting of a hole into two holes, and the event $\mathsf{G}_{0,0}$ results in the disappearance of a hole.

\begin{figure}[!h]
 \begin{center}
   \includegraphics[width=.9\linewidth]{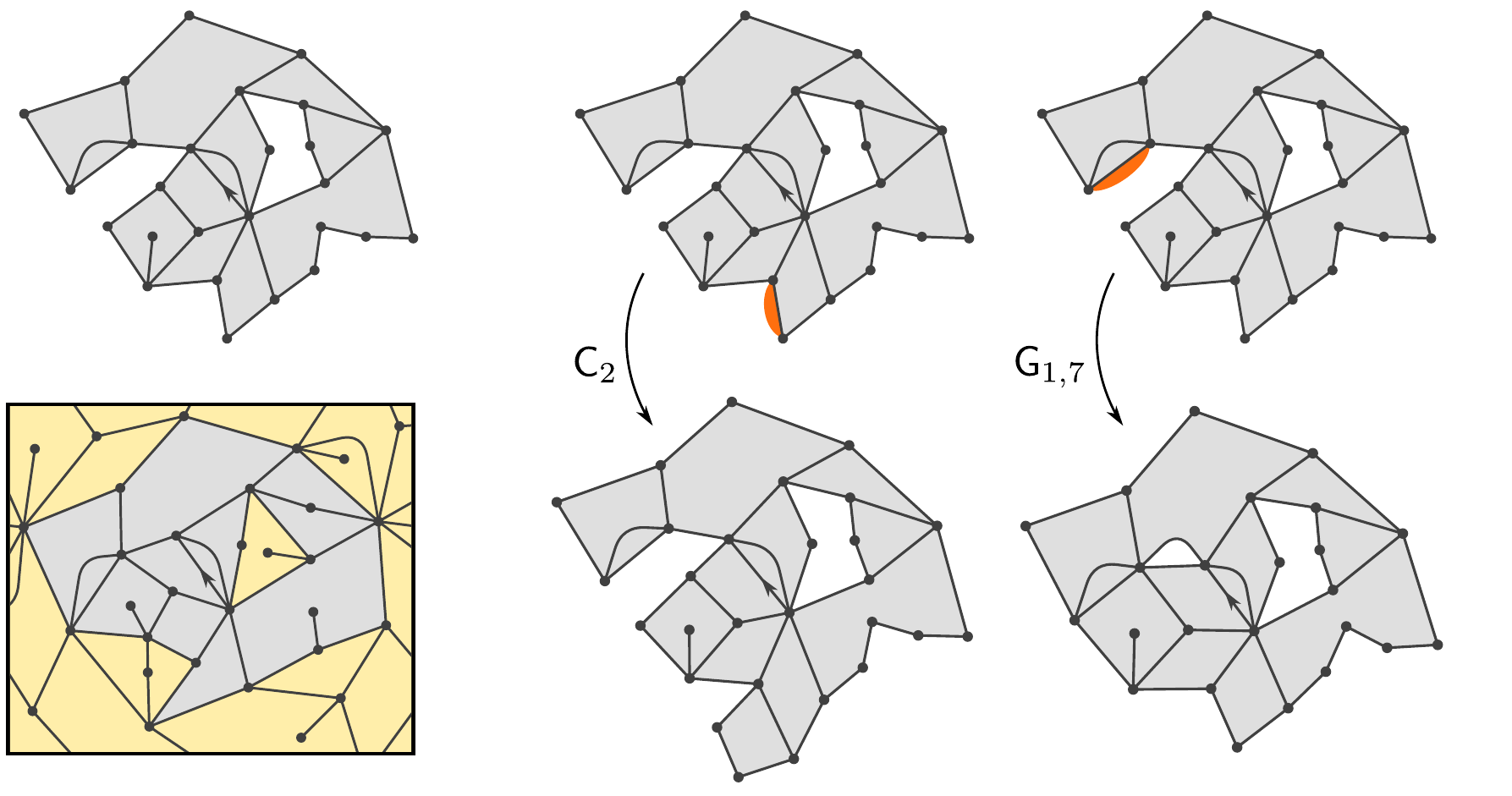}
 \caption{\label{fig:peel}Illustration of the different edge-peeling events. The left column illustrates the submap $\emap$ (top) of a map $\map$ (bottom). The center and right columns represent two different peeling events, with the edge to be peeled indicated in orange.}
 \end{center}
 \end{figure}

Formally, if $\map$ is a (finite or infinite) planar map, the branching edge-peeling exploration of $ \map$ with algorithm $ \mathcal{A}$ is by definition the sequence of planar maps with holes
$$ \mathfrak{e}_{0}(\map) \subset  \mathfrak{e}_{1}(\map) \subset \cdots \subset \mathfrak{e}_{n}(\map) \subset \cdots \subset \map,$$ obtained as follows:
\begin{itemize}
\item the map $  \mathfrak{e}_{0}(\map)$ is made of a simple face corresponding to the root face $  \rootface$ of the map and a unique hole of the same perimeter;
\item for every $ i \geq 0$, if  $  \mathcal{A}(\mathfrak{e}_{i}(\map)) \ne \dagger$, then the planar map with holes $ \mathfrak{e}_{i+1}(\map)$ is obtained from $ \mathfrak{e}_{i}(\map)$ by peeling the edge $ \mathcal{A}( \mathfrak{e}_{i}(\map))$. If $  \mathcal{A}(\mathfrak{e}_{i}(\map))= \dagger$, then $ \mathfrak{e}_{i+1}(\map) = \mathfrak{e}_{i}(\map)$ and the exploration process stops.
\end{itemize}

In particular, observe that if  $\mathfrak{e}_{i}(\map) \neq  \mathfrak{e}_{i-1}(\map)$ with $i \geq 1$, then $\mathfrak{e}_{i}(\map) $ has exactly $i$ internal edges (which are by definition edges of $\mathfrak{e}_{i}(\map)$ which do not belong to cycles).  If $i \geq 0$, the map with holes  $ \mathfrak{e}_{i}( \map)$ is obviously a (deterministic) function of $ \map$. But  note that  $( \mathfrak{e}_{j}( \map); 0 \leq j \leq i)$
is also a (deterministic) function of $ \mathfrak{e}_{i}( \map)$
. Finally, to simplify notation, we will often write $  \mathfrak{e}_{i}$ instead of $  \mathfrak{e}_{i}(\map)$.

\begin{remark} At this point, the reader may compare the above presentation with that of \cite[Sec.~2.3]{BCK}. In the peeling process considered in \cite[Sec.~2.3]{BCK}, the sequence $ \mathfrak{e}_{0} \subset \cdots \subset \mathfrak{e}_{n} \subset \cdots \subset \map$ is again a sequence of maps with simple holes (with the slight difference that in this case the holes can share vertices but not edges) but (unless the peeling has stopped), $ \mathfrak{e}_{i+1}$ is obtained from $ \mathfrak{e}_{i}$ by the addition of a new \emph{face}. Furthermore, in this peeling process, $ \map$ is obtained from $ \mathfrak{e}_{i}$ by filling-in the holes of $ \mathfrak{e}_{i}$ with maps having \emph{simple} boundary. In other words, the peeling process of \cite[Sec.~2.3]{BCK} is ``face''-peeling, while in the present work we have an ``edge''-peeling.
\end{remark}

In the sequel,  we usually simply say peeling exploration instead of branching edge-peeling exploration. Notice also that in our notation the sequence of explored maps $ (\mathfrak{e}_{i})$ depends obviously on the underlying map, but also on the peeling algorithm $ \mathcal{A}$. In the following it should be clear from the context which statements are valid for all peeling explorations and which for specific ones.

\subsection{Peeling of random Boltzmann maps}
\label{sec:peeling}

When dealing with finite or infinite Boltzmann planar maps, we  work on the canonical space $ \Omega$ of all the (rooted bipartite, possibly infinite) random maps with holes, possibly pointed. This space is equipped with the Borel $\sigma$-field for the local topology. The notation \begin{center}
$ \mathbb{P}^{(\ell)}, \mathbb{E}^{(\ell)},\quad$ resp.~$ \mathbb{P}^{(\ell)}_{\infty}, \mathbb{E}^{(\ell)}_{\infty},\quad$ resp.~${\mathbb{P}}^{(\ell)}_{\bullet},{\mathbb{E}}^{(\ell)}_{\bullet}$\end{center} is used for the probability and expectation on $\Omega$ relative to the law of a  $\mathbf{q}$-Boltzmann map with perimeter $2 \ell$, resp.~the infinite $ \mathbf{q}$-Boltzmann map with perimeter $ 2 \ell$, resp.\,a pointed $\mathbf{q}$-Boltzmann map with perimeter $2 \ell$. A generic element of the canonical space will be either denoted by $\map$ or by $\map_{\bullet}=(\map,v_{\bullet})$ if it is pointed.

In this section, we fix a peeling algorithm $ \mathcal{A}$ and  first  compute the law of the branching edge-peeling exploration  $ \mathfrak{e}_{0} \subset \mathfrak{e}_{1} \subset \cdots \subset \map$ under $ \mathbb{P}^{(\ell)}$,  $ {\mathbb{P}}^{(\ell)}_{\bullet}$ and $ \mathbb{P}^{(\ell)}_{\infty}$. We denote by $ \mathcal{F}_{n}$ the $\sigma$-algebra on $\Omega$ generated by the functions $\map \mapsto \mathfrak{e}_{0}(\map) ,\map \mapsto \mathfrak{e}_{1}(\map), \ldots, \map \mapsto  \mathfrak{e}_{n}(\map)$.  If $ \mathrm{e}$ is a planar map with holes, we set
$$ \widetilde{\tt w}( \mathrm{e}) = \prod_{ \begin{subarray}{c}
f \in \faces( \mathrm{e}) \backslash \{ \rootface\} \\
f \mathrm{ \ is \ not \ a \ hole}
\end{subarray}} q_{ \mathrm{deg}(f)/2}$$
and let $|\mathrm{e}|$ be the number of internal vertices of $\mathrm{e}$ (an internal vertex of $\mathrm{e}$ is a vertex that does not belong to a cycle of  $\mathrm{e}$).
Finally, to simplify notation,  for every $\ell \geq 0$, we set
$$ f^\uparrow( \ell) = \frac{ h^\uparrow(\ell)}{W^{(\ell)}\gamma^{\ell}} \qquad \mbox{ and } \qquad f^{{\newsearrow}}( \ell) = \frac{ h^{\newsearrow }(\ell)}{W^{(\ell)}\gamma^{\ell }} \underset{ \eqref{eq:egalbullet}}{=}   \frac{W^{(\ell)}_{\bullet}}{W^{(\ell)}}.$$
 Note that, by \eqref{eq:asymptowl} we have for some $c,c'>0$
\begin{equation}
\label{eq:asymph} f^\uparrow(\ell)  \quad \mathop{\sim}_{\ell \rightarrow \infty} \quad  c\cdot 
\ell^{\theta+3/2}, \qquad  f^\newsearrow(\ell)  \quad \mathop{\sim}_{\ell \rightarrow \infty} \quad c'
\cdot \ell^{\theta+1/2}.
\end{equation}

\begin{proposition} \label{prop:peelinggenerallaw} Fix $\ell \geq 1$ and $n \geq 0$. Let $ \mathrm{e}$ be a planar map with holes which can be obtained after $n$ peeling steps starting from a simple face of perimeter $2\ell$. Denote by   $\ell_{1}, \ell_{2}, \ldots, \ell_{k}$ the half-perimeters of the holes of $ \mathrm{e}$. Then
\begin{equation}
\label{eq:peellaw1}\mathbb{P}^{(\ell)}(   \mathfrak{e}_{n} = \mathrm{e}) = \frac{\widetilde{\tt w}(\mathrm{e})}{W^{(\ell)}} \prod_{i = 1}^k W^{( \ell_{i})}, \qquad {\mathbb{P}}^{(\ell)}_{\bullet}(   \mathfrak{e}_{n} = \mathrm{e}) = \frac{1}{f^\newsearrow(\ell)} \left(|\mathrm{e}|+ \sum_{j=1}^{k} f^\newsearrow( \ell_{j}) \right) \cdot \mathbb{P}^{(\ell)}(   \mathfrak{e}_{n} = \mathrm{e}) 
\end{equation}
and
\begin{equation}
\label{eq:peellaw2}\mathbb{P}^{(\ell)}_{\infty}( \mathfrak{e}_{n}=  \mathrm{e}) =  \frac{1}{f^\uparrow(\ell)}    \left(  \sum_{j= 1}^k f^\uparrow( \ell_{j}) \right) \cdot \mathbb{P}^{(\ell)}(   \mathfrak{e}_{n} = \mathrm{e}).
\end{equation}
Furthermore:
\begin{enumerate}
\item[(i)]
Under $ \mathbb{P}^{(\ell)}$ and conditionally on $\{  \mathfrak{e}_{n} =  \mathrm{e}\}$, the random maps filling-in the holes of $ \mathfrak{e}_{n}$ inside $ \map$ are independent $ \mathbf{q}$-Boltzmann maps. 
\item[(ii)] Under $ {\mathbb{P}}^{(\ell)}_{\bullet}$, conditionally given that $\{  \mathfrak{e}_{n} =  \mathrm{e}\}$ and that  $v_{\bullet}$ is not an internal vertex of $\mathrm{e}$,  the planar maps filling-in the holes of $ \mathfrak{e}_{n}$ inside $ \map$ are independent, all being $ \mathbf{q}$-Boltzmann maps, except for the $J$-th hole which is filled-in with a pointed $ \mathbf{q}$-Boltzmann map distributed as  $ {\mathbb{P}}^{(\ell_{J})}_{\bullet}$,  where the index $J$ is chosen at random, independently and proportionally to $j \mapsto f^{\newsearrow}(\ell_{j})$.
\item[(iii)] Under $ \mathbb{P}^{(\ell)}_{\infty}$ and conditionally on $\{  \mathfrak{e}_{n} =  \mathrm{e}\}$,  the planar maps filling-in the holes of $ \mathfrak{e}_{n}$ inside $ \map$ are independent, all being $ \mathbf{q}$-Boltzmann maps, except for the $J$-th hole which is filled-in with an infinite $ \mathbf{q}$-Boltzmann map of perimeter $2 \ell_{J}$ where the index $J$ is chosen at random, independently and proportionally to $j \mapsto f^{\uparrow}(\ell_{j})$.
 \end{enumerate}
\end{proposition}

\begin{proof} Since the peeling algorithm is deterministic and by rigidity, the event $ \{\mathfrak{e}_{n} = \mathrm{e}\} $ happens if and only if $\map$ is obtained from $\mathrm{e}$ by filling-in the holes of $\mathrm{e}$ with certain maps. The assertions concerning $ \mathbb{P}^{(\ell)}$  simply follow from the definition of the Boltzmann measure. The identities concerning $\mathbb{P}^{(\ell)}_{\infty}$ are obtained in a similar way to those of Proposition 3 in \cite{BCK}, but we give a full proof for sake of completeness. We have
\begin{eqnarray*}
 {\mathbb{P}}^{(\ell)}( \mathfrak{e}_{n}(\mathfrak{m})=\textrm{e} \mid  |\mathfrak{m}|=N) &=& \frac{\widetilde{\tt  w}(\textrm{e})}{W^{(\ell)}_{N}}   \sum_{ \underset{\map_{1} \in \mathsf{Map}^{(\ell_{1})}_{n_{1}}, \ldots,\map_{k} \in \mathsf{Map}^{(\ell_{k})}_{n_{k}}}{n_{1}+ \cdots+n_{k}=N-|\mathrm{e}|} } {\tt w}(\map_{1}) \cdots {\tt w}(\map_{k}) \\
 &=& \frac{\widetilde{\tt w}(\textrm{e})}{W^{(\ell)}_{N}}   \sum_{{n_{1}+ \cdots+n_{k}=N-|\mathrm{e}|}} W^{(\ell_{1})}_{n_{1}} \cdots W^{(\ell_{k})}_{n_{k}}.
 \end{eqnarray*}
Using Proposition \ref{prop:equivtaille}, it is then an easy matter to
verify that, for any $\varepsilon>0$, we can choose $K$ sufficiently large so that as $N \rightarrow \infty$, the asymptotic contribution of terms
corresponding to choices of $n_{1}, \ldots, n_{k}$ where $n_{i} \geq K$ for two distinct values of $ i \in \{1, \ldots, m\}$ is bounded above by $\varepsilon$ (see \cite[Lemma 2.5]{AS03}), so that by Proposition \ref{prop:equivtaille}
\begin{eqnarray*}
 {\mathbb{P}}^{(\ell)}( \mathfrak{e}_{n}(\mathfrak{m})=\textrm{e} \mid  |\mathfrak{m}|=N)    \displaystyle \mathop{\longrightarrow}_{N \rightarrow \infty}   \frac{\widetilde{\tt w}(\textrm{e})}{\gamma^{-\ell} h^\uparrow(\ell)} \cdot  \sum_{i=1}^{k} \gamma^{-\ell_{i}}h^\uparrow(\ell_{i}) \prod_{\underset{j \neq i}{j =1}}^{k} W^{(\ell_{j})} =
 \frac{1}{f^\uparrow(\ell)}    \left(  \sum_{j= 1}^k f^\uparrow( \ell_{j}) \right) \cdot \mathbb{P}^{(\ell)}(   \mathfrak{e}_{n} = \mathrm{e}).
 \end{eqnarray*}
By \cite[Theorem 6.1]{St14}, $ {\mathbb{P}}^{(\ell)}( \mathfrak{e}_{n}(\mathfrak{m})=\textrm{e} \mid  |\mathfrak{m}|=N) \rightarrow \mathbb{P}^{(\ell)}_{\infty}( \mathfrak{e}_{n}=  \mathrm{e}) $. This completes the proof of \eqref{eq:peellaw2}. The assertion $(iii)$ is established by using similar arguments and is left to the reader.

 Let us now prove the assertions concerning ${\mathbb{P}}^{(\ell)}_{\bullet}$. By definition of  ${\mathbb{P}}^{(\ell)}_{\bullet}$ (see \eqref{eq:defpbullet}),
\begin{eqnarray*}
 {\mathbb{P}}^{(\ell)}_{\bullet}(   \mathfrak{e}_{n} = \mathrm{e}, v_{\bullet} \textrm{ is an internal vertex of } \mathrm{e}) & = & \frac{ |e| \widetilde{\tt w}(\textrm{e})}{W^{(\ell)}_{\bullet}} \cdot \sum_{\map_{1} \in \mathsf{Map}^{(\ell_{1})}, \ldots,\map_{k} \in \mathsf{Map}^{(\ell_{k})}} {\tt w}(\map_{1}) \cdots {\tt w}(\map_{k})\\
 &=&  \frac{ |e| \widetilde{\tt w}(\textrm{e})}{W^{(\ell)}_{\bullet}} \cdot W^{(\ell_{1})} \cdots W^{(\ell_{k})} =\frac{ \widetilde{\tt w}(\mathrm{e})}{ W^{(\ell)}}   \left(  \prod_{i = 1}^k W^{( \ell_{i})} \right) 
 \cdot \frac{1}{f^\newsearrow(\ell)} \cdot |\mathrm{e}|,
\end{eqnarray*}
where we have used the fact that $ f^{{\newsearrow}}( \ell) ={W^{(\ell)}_{\bullet}}/{W^{(\ell)}}$ for the last equality. Similarly, for $1 \leq j \leq k$,
\begin{eqnarray*}
 {\mathbb{P}}^{(\ell)}_{\bullet}(   \mathfrak{e}_{n} = \mathrm{e}, v_{\bullet} \in \map_{j}) & = & \frac{\widetilde{\tt w}(\textrm{e})}{W^{(\ell)}_{\bullet}} \cdot \sum_{\map_{1} \in \mathsf{Map}^{(\ell_{1})}, \ldots,\map_{k} \in \mathsf{Map}^{(\ell_{k})}} |\map_{j}| \cdot {\tt w}(\map_{1}) \cdots {\tt w}(\map_{k})\\
 &=&  \frac{\widetilde{\tt w}(\textrm{e})}{W^{(\ell)}_{\bullet}} \cdot W^{(\ell_{1})} \cdots W^{(\ell_{k})}  \cdot \frac{W^{(\ell_{j})}_{\bullet}}{W^{(\ell_{j})}}=\frac{ \widetilde{\tt w}(\mathrm{e})}{ W^{(\ell)}}   \left(  \prod_{i = 1}^k W^{( \ell_{i})} \right) 
 \cdot \frac{1}{f^\newsearrow(\ell)} \cdot f^\newsearrow(\ell_{j}).
\end{eqnarray*}
The expression of ${\mathbb{P}}^{(\ell)}_{\bullet}(   \mathfrak{e}_{n} = \mathrm{e}) $ then follows by summing the previous equalities. Assertion $(ii)$ is established in a similar way, and we leave details to the reader.\end{proof}

\medskip

We now exhibit two martingales that appear in every peeling exploration of a $ \mathbf{q}$-Boltzmann planar map. As before, we denote by $ \mathfrak{e}_{0} \subset \cdots \subset \mathfrak{e}_{n} \subset \cdots $ the branching peeling exploration of an underlying planar map $\map$ (recall that the peeling algorithm $ \mathcal{A}$ is fixed). To simplify notation, for every $n \geq 0$, we let $\boldsymbol{\ell}(n)=(\ell_{1}(n), \ell_{2}(n), \ldots)$ be the lengths of the holes of the explored map $ \mathfrak{e}_{n}$. If $g$ is a function and $\boldsymbol{\ell}=(\ell_{1}, \ell_{2}, \cdots)$ is a  sequence of integers we simply write $g( \boldsymbol{ \ell})$ for the sum $ g( \ell_{1}) + g( \ell_{2}) + \cdots$.

\begin{proposition} Under $ \mathbb{P}^{(\ell)}$, the processes 
$$ \Big(f^\uparrow \big(  \boldsymbol{ \ell}(n)\big)\Big)_{n \geq 0} \quad \mbox{ and } \quad \Big( |  \mathfrak{e}_{n}| + f^\newsearrow \big(  \boldsymbol{ \ell}(n)\big)\Big)_{n \geq 0}$$
are positive $( \mathcal{F}_{n})$ martingales, which are respectively called the cycle and the area martingales. The cycle martingale is not uniformly integrable and converges almost surely to $0$, whereas the area martingale is closed and at time $n$ is equal to the conditional expectation of the total number of vertices of a  $ \mathbf{q}$-Boltzmann planar map given the map with holes obtained after $n$ peeling steps.
\label{prop:martingales}
\end{proposition}

\begin{remark} For the connection with growth-fragmentation processes that we have in mind, the above two martingales are the discrete analogs of respectively the martingale $M^+$ and the martingale $M^-$ studied in the preceding sections.
\end{remark}
\begin{proof} The proofs are the same as that of Propositions 6 and 7 in \cite{BCK}. More precisely, from Proposition \ref{prop:peelinggenerallaw} it follows that $f^\uparrow( \boldsymbol{\ell}(n))/ f^\uparrow(\ell)$ is the Radon-Nikodym derivative of the branching peeling exploration until step $n$ under $ \mathbb{P}^{(\ell)}_{\infty}$ with respect to the same exploration under $ \mathbb{P}^{(\ell)}$. This implies the first assertion. For the second one, we similarly rely on Proposition \ref{prop:peelinggenerallaw} and observe that
\begin{equation}
\label{eq:esvolume}\mathbb{E}^{(\ell)}(|\map|) = \frac{W^{(\ell)}_{\bullet}}{W^{(\ell)}} \underset{ \eqref{eq:egalbullet}}{=}  f^\newsearrow(\ell).\qedhere
\end{equation}
\end{proof}

\subsection{Scaling limits for the perimeter of a distinguished cycle}

\label{sec:distcycle}
In this section, we establish a scaling limit result for the perimeter of a distinguished cycle under $ \mathbb{P}^{(\ell)}$,  $\mathbb{P}^{(\ell)}_{\bullet}$ and  $\mathbb{P}^{(\ell)}_{\infty}$. We fix a peeling algorithm $ \mathcal{A}$ such that $ \mathcal{A}(\map) \neq \dagger$ if $ \map$ has at least one hole. Recall that $(S_{n})_{n \geq 0}$ denotes the random walk on $\Z$ with jump distribution $\nu$.

\paragraph{Infinite Boltzmann planar maps.} When performing a branching peeling exploration under $\mathbb{P}^{(\ell)}_{\infty}$, there is  one cycle that  plays a particular role, namely the one separating the origin from infinity (the infinite version of a  $ \mathbf{q}$-Boltzmann planar map has almost surely one end \cite[Lemma 6.3]{St14}).  Specifically, during a peeling exploration of an infinite planar map with one end, we define a family of distinguished cycles $  (\mathscr{C}_{\infty}(i))_{i \geq 0}$ as follows. The initial distinguished cycle $\mathscr{C}_{\infty}(0)$ is the only cycle of $ \mathfrak{e}_{0}$ and $ \sigma_{0}=0$.
Then, inductively, for $i \geq 0$, if $\mathscr{C}_{\infty}(i)=\dagger$ (the cemetery point), set $\mathscr{C}_{\infty}(i+1)=\dagger$, and otherwise define $\sigma_{i+1}= \inf \{ j >\sigma_{i} : \mathcal{A}(  \mathfrak{e}_{j}) \in \mathscr{C}_{\infty}(i) \}$ (with the usual convention $\inf \emptyset =\infty$). If $\sigma_{i+1}=\infty$, we define $\mathscr{C}_{\infty}(i+1)=\mathscr{C}_{\infty}(i)$. Otherwise, when peeling the edge $\mathcal{A}(  \mathfrak{e}_{\sigma_{i+1}})$, we define $\mathscr{C}_{\infty}(i+1)$ depending on what peeling event happens:
\begin{itemize}
\item If the event $\mathsf{C}_{k}$ occurs, we define  $\mathscr{C}_{\infty}(i+1)$ to be the new cycle thus created,
\item If the event $\mathsf{G}_{0,0}$ occurs (disappearance of the hole), we define $\mathscr{C}_{\infty}(i+1)=\dagger$,
\item If the event $ \mathsf{G}_{k_{1},k_{2}}$ occurs with $(k_{1},k_{2}) \neq (0,0)$, two new cycles are created (one possibly empty) when peeling the edge $\mathcal{A}(  \mathfrak{e}_{\sigma_{i+1}})$. We define $ \mathscr{C}_{\infty}(i+1)$ to be the cycle  that separates the origin of the map from infinity.
\end{itemize}
Observe that under $\mathbb{P}^{(\ell)}_{\infty}$, the event $ \mathsf{G}_{0,0}$ has probability $0$. Finally, we agree by convention that the perimeter of $\dagger$ is $0$, and for every $i \geq 0$ we let $P_{\infty}(i)$ be the half-perimeter of $\mathscr{C}_{\infty}(i)$.

It follows from Proposition \ref{prop:peelinggenerallaw} that, under  $\mathbb{P}^{(\ell)}_{\infty}$, the process $(P_{\infty}(i): i \geq 0)$ evolves as a Markov chain on the positive integers starting from $\ell$ with explicit transitions:
$${\mathbb{P}}^{(\ell)}_{\infty}( P_{\infty}(i+1)=k+m \mid  P_{\infty}(i)=m) = \nu(k) \cdot \frac{h^\uparrow(m+k)}{h^\uparrow(m)} \qquad (k \geq -m).$$
 This Markov chain  can be seen as the random walk $(S_{n})_{n \geq 0}$ starting from $\ell$ conditioned to remain positive, that is, rigorously, as the Doob $h^{\uparrow}$-
 transform of the random walk $(S_{n})_{n \geq 0}$ killed when entering $\Z_-$ (see \cite{Bud15},\cite[Sec.~1.2.2]{BCgrowth}).

\paragraph{Pointed Boltzmann planar maps.}

When performing a branching peeling exploration under ${\mathbb{P}}^{(\ell)}_{\bullet}$, there is  one cycle that  plays a particular role, namely the one separating the origin from $v_{\bullet}$.  Specifically,
we define a family of distinguished cycles $  (\mathscr{C}_{\bullet}(i))_{i \geq 0}$ exactly as in the case of infinite Boltzmann planar maps, with the only difference that if the event $ \mathsf{G}_{k_{1},k_{2}}$ occurs we define $ \mathscr{C}_{\bullet}(i+1)$ to be the cycle which is filled-in by the planar map containing the distinguished vertex $v_{\bullet}$. If the latter cycle is empty, i.e. when $v_\bullet$ is encountered, we set $\mathscr{C}_{\bullet}(i+1)=\dagger$. Finally, for every $i \geq 0$ we let $P_{\bullet}(i)$ be the half-perimeter of $\mathscr{C}_{\bullet}(i)$.

It follows from Proposition \ref{prop:peelinggenerallaw} that, under  ${\mathbb{P}}^{(\ell)}_{\bullet}$, the process $(P_{\bullet}(i): i \geq 0)$ evolves as a Markov chain on the non-negative integers starting from $\ell$ with the following explicit transitions:
$${\mathbb{P}}^{(\ell)}_{\bullet}( P_{\bullet}(i+1)=k+m \mid  P_{\bullet}(i)=m) = \nu(k) \cdot \frac{h^\newsearrow(m+k)}{h^\newsearrow(m)} \qquad (k \geq -m).$$
This Markov chain  can therefore be seen as the random walk $(S_{n})_{n \geq 0}$ starting from $\ell$, absorbed when entering $\Z_-$ and 
conditioned to be absorbed at $0$,  or, equivalently, as a Doob $h^\newsearrow$-transform of the random walk $(S_{n})_{n \geq 0}$ killed when entering $\Z_-$.

\paragraph{Boltzmann planar maps.}  During a peeling exploration under $\mathbb{P}^{(\ell)}$, we can still distinguish a natural family of cycles, namely the locally largest ones. Specifically,  we define a family of distinguished cycles $  (\mathscr{C}_{\ast}(i))_{i \geq 0}$ exactly as in the case of infinite Boltzmann planar maps, with the only difference that if the event $ \mathsf{G}_{k_{1},k_{2}}$ occurs with $(k_{1},k_{2}) \neq (0,0)$,  one creates two new cycles (one possibly empty), and we define  $ \mathscr{C}_{\ast}(i+1)$ to be the cycle with largest perimeter  (if $\ell_{1}=\ell_{2}$, we choose between the two in a deterministic way). Finally, for every $i \geq 0$ we let $P_\ast(i)$ be the half-perimeter of $\mathscr{C}_{\ast}(i)$.

Again, Proposition \ref{prop:peelinggenerallaw} shows that under  ${\mathbb{P}}^{(\ell)}$,  $(P_{\ast}(i): i \geq 0)$ is a Markov chain on the positive integers starting from $\ell$ with the following explicit transitions:
$${\mathbb{P}}^{(\ell)}( P_{*}(i+1)=k+m \mid  P_{*}(i)=m) = \nu(k) \cdot \frac{\gamma^{m+k}W^{(m+k)}}{\gamma^m W^{(m)}} \qquad (k > \frac{-m-1}{2} ),$$ and the transition is just half of the last display if $m-1$ is even and $-k = \frac{m+1}{2}$. However, we will not need the exact value of these transitions in the following.

\paragraph{Scaling limits for the perimeter of  the distinguished cycle.} We now introduce the different scaling limits of the three families of distinguished cycles we have just defined. They are all related to the growth-fragmentation $ (\mathbf{X}^{(-\theta)}_{\theta}(t), t \geq 0)$ with cumulant function $\kappa_{\theta}$ given by \eqref{eq:kappatheta} and self-similarity parameter $-\theta$.

Using the results of Sec.~\ref{sec:oneparameter}; we may construct $ \mathbf{X}_{\theta}^{(-\theta)}$ by choosing the evolution of the Eve cell to be the self-similar Markov process $X_{\theta}^{(-\theta)}$ with characteristics $(\Psi_{\theta},-\theta)$, where we recall that $\Psi_{\theta}$ is the Laplace exponent \eqref{eq:PsiTheta}. In particular, $X_{\theta}^{(-\theta)}$ does not make negative jumps larger than half of its current value and, roughly speaking, describes the evolution of the size of the locally largest particle.
Recalling the results of the last section, the processes $Y_{\theta}^{+}$ and $Y_{\theta}^{-}$ defined in Sec.~\ref{sec:plus} and Sec.~\ref{sec:area}  corresponding to the evolution of the tagged particles in the biased versions of $ \mathbf{X}_{\theta}^{(-\theta)}$ are distributed as follows: Let $\Upsilon_{\theta}$ be the $\theta$-stable L\'evy process with positivity parameter $\rho=\Pr{\Upsilon_{\theta}(1) \geq 0}$ satisfying $\theta(1-\rho)=1/2$ and normalized so that its L\'evy measure is $$ \frac{\Gamma(1+\theta)}{\pi} \cdot \cos((1+\theta)\pi) \cdot \frac{\d x}{x^{1+\theta}} \mathbbm{1}_{x>0}+  \frac{\Gamma(1+\theta)}{\pi} \cdot \frac{\d x}{|x|^{1+\theta}} \mathbbm{1}_{x<0}.$$
Then introduce $\Upsilon^\uparrow_{\theta}$ and $\Upsilon_{\theta}^\newsearrow$ the versions of  the L\'evy process $\Upsilon_{\theta}$  conditioned to stay positive and respectively to die continuously at $0$ when it enters $ \R_{-}$ (see \cite{Cha96,CC08} for the definition of these processes). Then by the result of the last section we have $Y_{\theta}^+ = \Upsilon_{\theta}^\uparrow$ and $Y_{\theta}^-= \Upsilon_{\theta}^\newsearrow$ in law.

\begin{proposition} \label{prop:scalingll} Assume that $ \mathbf{q}$ is admissible, critical, non-generic and satisfies $q_{k} \sim c \cdot \gamma^{k-1} \cdot k^{-1-\theta}$ as $k \rightarrow \infty$, with $\theta \in (1/2,3/2)$. Then, setting  $\mathsf{c}_{ \mathbf{q}}= \frac{\pi c}{ \Gamma(1+\theta) \cos( (1+\theta) \pi)}$ the following three convergences hold  in distribution for the Skorokhod $J_{1}$ topology on $\mathbb{D}(\R_{+},\R)$:
\begin{eqnarray*}
\textrm{under }  \mathbb{P}^{(\ell)}_{\infty}, \qquad  \left( \frac{1}{\ell} \cdot P_{\infty}(\lfloor \ell^{\theta} \cdot t \rfloor ) : t \geq 0\right)  &\displaystyle \mathop{\longrightarrow}^{(d)}_{\ell \rightarrow \infty} & \big( {Y}^{+}_{\theta}(  \mathsf{c}_{ \mathbf{q}} \cdot t) : t \geq 0\big);\\
\textrm{under }  \mathbb{P}^{(\ell)},   \qquad  \left( \frac{1}{\ell} \cdot P_{*}(\lfloor \ell^{\theta}\cdot t \rfloor ) : t \geq 0\right)  &\displaystyle \mathop{\longrightarrow}^{(d)}_{\ell \rightarrow \infty} & \big( {X}^{(-\theta)}_{\theta}(  \mathsf{c}_{ \mathbf{q}} \cdot t) : t \geq 0\big);\\
\textrm{under }  {\mathbb{P}}^{(\ell)}_{\bullet}, \qquad  \left( \frac{1}{\ell} \cdot P_{\bullet}(\lfloor \ell^{\theta}\cdot t \rfloor ) : t \geq 0\right)  &\displaystyle \mathop{\longrightarrow}^{(d)}_{\ell \rightarrow \infty} & \big( {Y}^{-}_{\theta}(  \mathsf{c}_{ \mathbf{q}} \cdot t) : t \geq 0\big).
\end{eqnarray*}
\end{proposition}

\begin{proof}We have seen that  the process $(P_{\bullet}(i): i \geq 0)$ (resp.~$(P_{\infty}(i): i \geq 0)$) evolves as the random walk $(S_{n})_{n \geq 0}$ starting from $\ell$ and  conditioned to be absorbed at $0$ before touching $ \mathbb{Z}_{-}$ (resp.~conditionned to stay positive) or, equivalently as the Doob $h^\newsearrow$-transform (resp.~ $h^{\uparrow}$) of the random walk $(S_{n})_{n \geq 0}$ absorbed when entering $\Z_-$. The first and third convergences follow from the general invariance principle proved \cite[Theorem 1.1 and 1.3]{CC08} once we know that 
$$ \frac{1}{\ell} S_{\ell^{\theta}} \xrightarrow[\ell\to\infty]{(d)} \Upsilon_{\theta}(1).$$  
The last display is proved in the case $\theta \ne 1$ in \cite[Proposition 3.2]{BCgrowth} but with another constant due to the different normalization of the L\'evy measure. The more delicate Cauchy case $\theta=1$ is \cite[Proposition 2]{BCMCauchy}.

For the second assertion, instead of adapting the proof of \cite[Proposition 9]{BCK} and applying results of \cite{BK14}, we shall transfer a convergence in distribution under $\mathbb{P}^{(\ell)}_{\infty}$ to a convergence under $\mathbb{P}^{(\ell)}$ by using absolute continuity relations in both the discrete setting (relying on Sec.~\ref{sec:peeling}) and the continuous setting (relying on Prop.~\ref{P2}). Specifically, Proposition \ref{prop:peelinggenerallaw} shows that for every $n\geq 0$ and every sequence $x_0=\ell, x_1, \ldots, x_n$ in $\{1, 2, \ldots\}$ with $x_{i+1}\geq \frac{1}{2}x_i$ for all $i=0, \ldots, n-1$, there is the identity
$$\P^{(\ell)}\left( P_\ast(0)=x_0, \ldots, P_\ast(n)=x_n\right)=\frac{f^{\uparrow}({\ell)}}{f^{\uparrow}(x_n)}
\P^{(\ell)}_{\infty}\left( P_{\infty}(0)=x_0, \ldots,P_{\infty}(n)=x_n\right).$$
As a consequence, if we fix $T>0$ and let $F : \mathbb{D}([0,T],\R) \rightarrow \R_{+}$ be a bounded continuous function such that $F(X)=0$ if $X \in  \mathbb{D}([0,T],\R) $ is such that there exists  $0<t \leq T$ with $X(t)<X(t-)/2$, we get that
$$
\E^{(\ell)} \left( F \left( \frac{1}{\ell} P_\ast(\lfloor \ell^{\theta}t \rfloor ) :  0 \leq t \leq T\right)   \right)   \quad \mathop{\longrightarrow}_{\ell \rightarrow \infty} \quad \E\left( {Y}^{+}_{\theta}(  \mathsf{c}_{ \mathbf{q}} \cdot T)^{-(\theta+3/2)} F \left( Y^{+}_{\theta}(  \mathsf{c}_{ \mathbf{q}} \cdot t) : 0 \leq t \leq T \right) \right),$$
where we have also used the estimate  \eqref{eq:asymph} concerning the asymptotic behavior of $f^{\uparrow}({\ell)}$. 

Assume that the growth-fragmentation  $ (\mathbf{X}^{(-\theta)}_{\theta}(t), t \geq 0)$ is constructed by using the self-similar Markov process  $X^{(-\theta)}_{\theta}$ for the evolution of the Eve cell. Recall from Sec.~\ref{sec:spine1} the notation $\widehat{\mathcal P}^+_1$  describing the joint distribution of a cell system and a leaf,  and that $\hat{\mathcal X}(t)$ is the size of the tagged cell at time $t$.  In particular, we have $\omega_{+}=\theta+3/2$ (recall \eqref{eq:cramertheta}). In addition, by Theorem \ref{T2},  $\hat{\mathcal X}$ under $\widehat{\mathcal P}^+_1$ has the same distribution as $Y^{+}_{\theta}$.  Therefore, setting $ \mathbf{X}^{(-\theta)}_{\theta}(t)=\{X_{1}(t),X_{2}(t), \ldots\}$,
\begin{eqnarray*}
 \E\left( {Y}^{+}_{\theta}(  T)^{-(\theta+3/2)} F \left( Y^{+}_{\theta}(   t) : 0 \leq t \leq T \right) \right)  &= & \widehat{\mathcal E}^+_1 \left(  \hat{\mathcal X}(T)^{-\omega_{+}} F \left( \hat{\mathcal X}(t): 0 \leq t \leq T\right) \right)\\
 &\underset{(\textrm{Prop.}~\ref{P2})}{=}&   \mathbb{E}_{1} \left( \sum_{i=1}^{\infty}  F \left(X_{i}(t): 0 \leq t \leq T\right) \right)  \\
 &=& {E} \left( F(X^{(-\theta)}_{\theta}(t) : 0 \leq t \leq T)  \right).
\end{eqnarray*}
For the last equality, we have used the fact that  if $X_{i}(T)>0$ and $X_{i}(t) \geq X_{i}(t-)/2$ for every $0<t <T$, then $X_{i}$ is the Eve cell (it is the locally largest particle) and $X_{i}(t)=X^{(-\theta)}_{\theta}(t)$ for every $0 \leq t \leq T$.  We conclude that
$$\E^{(\ell)} \left( F \left( \frac{1}{\ell} P_\ast(\lfloor \ell^{\theta}t \rfloor ) :  0 \leq t \leq T\right)   \right)   \quad \mathop{\longrightarrow}_{\ell \rightarrow \infty} \quad  \mathbb{E}\left( F(X^{(-\theta)}_{\theta}(\mathsf{c}_{ \mathbf{q}} \cdot t) : 0 \leq t \leq T)  \right).$$
This shows the second convergence and completes the proof.
\end{proof}

As an application of Proposition \ref{prop:scalingll}, we compute the law of the intrinsic area of the growth-fragmentation $ \mathbf{X}^{(-\theta)}_{\theta}$ (recall from Remark \ref{rem:lawM} that this law only depends on the cumulant function $\kappa_{\theta}$). For $\beta \in (0,1)$, let $ \mathfrak{S}_\bullet(\beta)$ be a positive $\beta$-stable random variable with Laplace transform $\mathbb{E} (e^{-\lambda \mathfrak{S}_\bullet(\beta)}) = \exp\left( - (\Gamma(1+1/\beta) \lambda)^{\beta}\right)$. Then $ \mathbb{E} (1/\mathfrak{S}_\bullet(\beta)) = 1$ and we can define a random variable $\mathfrak{S}(\beta)$ by biasing $\mathfrak{S}_{\bullet}(\beta)$ by $x\to 1/x$, that is for any $f \geq 0$
$$ \mathbb{E}(f(\mathfrak{S}(\beta))) =  \mathbb{E}\left(f( \mathfrak{S}_{\bullet}(\beta)) \frac{1}{ \mathfrak{S}_{\bullet}(\beta)}\right).$$

\begin{corollary}\label{cor:area}Fix $\theta \in (1/2,3/2]$. Denote by $ \mathcal{M}^-_{\theta}(\infty)$  the intrinsic area of a growth-fragmentation with cumulant function $\kappa_{\theta}$. Then
$$ \mathcal{M}^-_{\theta}(\infty)   \quad \mathop{=}^{(d)} \quad  \mathfrak{S} \left(  \frac{1}{\theta+ \frac{1}{2}} \right) .$$
\end{corollary}

\begin{proof}Fix $\theta \in (1/2,3/2)$ and let $ \mathbf{q}$ be an admissible, critical, non-generic sequence  with polynomial tail of exponent $-\theta-1$ (e.g.~the one given by given by \eqref{eq:explicite}).  To simplify notation, we write $\mathfrak{S}$ for the random variable $\mathfrak{S}\big(( \theta+1/2)^{-1}\big)$ and we let $B^{(\ell)}$ be a random $\mathbf{q}$-Boltzmann planar map with perimeter $2 \ell$ distributed according to $\P^{(\ell)}$. The key is to use the fact that the area of $B^{(\ell)}$, appropriately rescaled, converges in distribution to $\mathfrak{S}$ as $\ell \rightarrow \infty$.  More precisely, it is proved in \cite[Proposition 4]{BCgrowth} that we have the convergence in distribution 
\begin{equation}\label{eq:volumeconv}
\frac{1}{\ell^{\theta+1/2}} \cdot | B^{(\ell)}|  \quad \mathop{\longrightarrow}^{(d)}_{\ell \rightarrow \infty} \quad   \mathsf{b}_{ \mathbf{q}}\cdot \mathfrak{S}  \qquad \textrm{with} \quad     \mathsf{b}_{ \mathbf{q}}= \frac{2 \gamma \cos(\pi (1+\theta))}{c \sqrt{\pi}}.
\end{equation}
Now imagine that we start a branching peeling exploration of $B^{(\ell)}$ which only peels along the locally largest hole until it stops (that is, the holes which are not the locally largest one are frozen and never explored afterwards, see the proof of Lemma 13 in \cite{BCK} for a precise definition). If we denote by $V^{(\ell)}$ the number of internal vertices revealed during this exploration and if $(\delta_{i}^{(\ell)})_{i \geq 1}$ are the half-lengths of the holes which are frozen, ranked in decreasing order, then by Proposition \ref{prop:peelinggenerallaw} and Proposition \ref{prop:martingales} we have
$$ |B^{(\ell)}| = V^{(\ell)} + \sum_{i \geq 1} |B^{(\delta_{i}^{(\ell)})}_{i}|,$$ where for every $i \geq 1$, the variable  $|B^{(\delta_{i}^{(\ell)})}_{i}|$ has the law of the area of a $ \mathbf{q}$-Boltzmann map of perimeter $2 \delta_{i}^{(\ell)}$ and are independent conditionally on the exploration so far. Denote by $(\Delta_{1}, \Delta_{2}, \ldots)$  the absolute values of the negative jumps of $t \mapsto  {X}_{\theta}^{(-\theta)}(  \mathsf{c}_{ \mathbf{q}} t )$ ranked in decreasing order. Using the fact that by Proposition \ref{prop:scalingll}, $ \ell^{-1} (\delta_{1}^{(\ell)},\delta_{2}^{(\ell)},\ldots)$ converges in distribution to $(\Delta_{1}, \Delta_{2}, \ldots)$, using \eqref{eq:volumeconv} it follows that we have the following  inequality for the stochastic order
\begin{equation}
\label{eq:stoineg} \mathfrak{S} \quad \underset{ \mathrm{sto}}{\geq} \quad  \sum_{ i \geq 0} \Delta_{i}^{\theta+1/2} \mathfrak{S}_{i},
\end{equation}
where $(\mathfrak{S}_{i} : i \geq 1)$ are i.i.d.~copies of $\mathfrak{S}$. By \eqref{eq:pointfixevolume} we also have $\sum_{ i \geq 0} \E(\Delta_{i}^{\theta+1/2})=1$  (recall that in our setting $\omega_{-} = \theta+1/2$), it follows that both sides  of \eqref{eq:stoineg} have the same expectation, so that  \eqref{eq:stoineg} is actually an equality in distribution. Hence $\mathfrak{S}$ is a fixed point of the recursive distributive equation \eqref{eq:pointfixevolume}, which, as we have already seen, has  a unique solution with given mean. Since $\mathfrak{S}$ has mean $1$, the conclusion follows. The case $ \theta = \frac{3}{2}$ is established in a similar way,  by using \cite[Proposition 9]{CLGpeeling} and \cite[Proposition 9]{BCK}, instead of respectively \cite[Proposition 4]{BCgrowth} and Proposition \ref{prop:scalingll}.
\end{proof} 

\subsection{Slicing at heights large planar maps with high degrees}

In this section, we  give a more geometric flavor to our connection between growth-fragmentation processes and planar maps. We show that the scaling limit of the perimeters of the cycles obtained by slicing $B^{(\ell)}$ at fixed heights is a time-changed version of the growth-fragmentation appearing in the previous section. Unlike Proposition \ref{prop:scalingll} which is valid for every $\theta \in (1/2,3/2)$, the results in this section only hold for the so-called dilute phase, where $\theta \in (1,3/2)$.  Before stating the main result, we introduce some notation.

If $\map$ is a (bipartite) planar map, recall that $\map^\dagger$ stands for the dual map of $\map$. If $f$ is a face of $\map$, its height  is by definition the dual graph distance $ \mathrm{d}^\dagger_{ \mathrm{gr}}(f,\rootface)$ in $ \mathfrak{m}^{\dagger}$ between $f$ and the root face $\rootface$. For $r \geq 0$, we let
$$ \mathrm{Ball}_{r}^\dagger( \map)$$
Then $\mathrm{Ball}_{r}^\dagger( \map)$ is a submap of $\map$ with possibly several holes. We denote by 
 $$ {{\bf L}}(r) \coloneqq \left(L_1(r),   L_2(r), \ldots\right),$$ the half-lengths of the cycles of $  \mathrm{Ball}^\dagger_{r}( \map)$ ranked in decreasing order. An example is shown in Fig.\,\ref{fig:slicing}.
 
 \begin{figure}[!h]
  \begin{center}
  \includegraphics[width=0.95\linewidth]{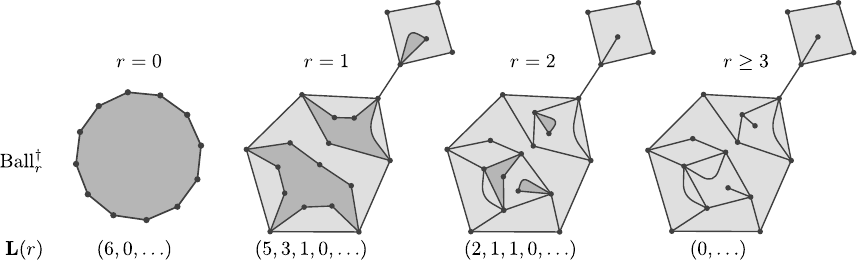}
  \caption{Example of the slicing at heights for a finite planar map (shown on the far right) with root face taken to be the outer face. The dark shaded faces correspond to the holes.\label{fig:slicing}}
  \end{center}
  \end{figure}
 
Let $ \mathbf{X}_{\theta}^{(1-\theta)}$ be the growth-fragmentation process with characteristics $(\kappa_{\theta}, 1-\theta)$: it can be constructed from the process $X_{\theta}^{(1-\theta)}$ with characteristics $(\Psi_{\theta}, 1-\theta)$ for the evolution of the Eve cell. Note from, e.g., Proposition 13.5 in \cite{Kyp}, that the self-similar Markov processes ${X}^{(1-\theta)}_{\theta}$ and ${X}^{(-\theta)}_{\theta}$ (which was introduced just before Proposition \ref{prop:scalingll}) are related by the following time-change relation:
$$X^{(1-\theta)}_{\theta}(t)=X^{(-\theta)}_{\theta}\left(\int_0^t\frac{\d s }{{X}^{(1-\theta)}_{\theta}(s)}\right), \qquad t\geq 0.$$ Recall also from Sec.~\ref{sec:spine1} and Sec.~\ref{sec:area} the notation $ \mathbb{P}^{+}_{1}$ and $\mathbb{P}^{-}_{1}$ for the laws of the biased versions of the growth-fragmentation, and  that $\mathsf{b}_{ \mathbf{q}}= \frac{2 \gamma \cos(\pi (1+\theta))}{c \sqrt{\pi}}$, $\mathsf{c}_{ \mathbf{q}} = \frac{\pi c}{ \Gamma(1+\theta) \cos((1+\theta)\pi)}$. Finally, set
$$\mathsf{a}_{\mathbf{q}}= \frac{1}{2} \left( 1+ \sum_{k=0}^{\infty}(2k+1) \nu(k) \right)$$
and
$$\ell^{\downarrow}_{\theta+3/2} \coloneqq  \left\{{\bf x}=(x_i)_{i\in\N}: x_1\geq x_2\geq \cdots \geq 0 \hbox{ and } \sum_{i=1}^{\infty} x_i^{\theta+3/2}<\infty\right\}.$$

 \begin{theorem}[Slicing at heights]\label{thm:GF} Let $ \mathbf{q}=( q_{k})_{k \geq 1}$ be an admissible critical and non-generic weight sequence satisfying $q_{k} \sim c \gamma^{k-1} k^{-1-\theta}$ as $k \rightarrow \infty$, for $\theta \in (1,3/2)$. Then the following three  convergences hold in distribution
 \begin{eqnarray}
\textrm{under }  \mathbb{P}^{(\ell)}_{\infty}, \qquad   \left(  \frac{1}{\ell} \cdot \mathbf{L}( \ell^{\theta-1} \cdot  t) : t \geq 0  \right)    & \displaystyle\mathop{\longrightarrow}^{(d)}_{\ell \rightarrow \infty} & \left( {\bf X}^{(1-\theta)}_{\theta} \left(  \frac{\mathsf{c}_{\mathbf{q}}}{\mathsf{a}_{\mathbf{q}}} \cdot t \right) : t \geq 0 \right)    \textrm{ under } \mathbb{P}^{+}_{1}\notag\\
\textrm{under }  \mathbb{P}^{(\ell)},   \qquad \left(    \frac{1}{\ell} \cdot    \mathbf{L}( \ell^{\theta-1}\cdot t  )  : t \geq 0     \right)  & \displaystyle\mathop{\longrightarrow}^{(d)}_{\ell \rightarrow \infty} & \left( {\bf X}^{(1-\theta)}_{\theta} \left(  \frac{\mathsf{c}_{\mathbf{q}}}{\mathsf{a}_{\mathbf{q}}} \cdot t \right) : t \geq 0\right)  \textrm{ under } \mathbb{P}_{1} \label{eq:cv2}\\
\textrm{under }  \mathbb{P}^{(\ell)}_{\bullet}, \qquad  \left(   \frac{1}{\ell} \cdot  \mathbf{L}( \ell^{\theta-1} \cdot t  )  : t \geq 0  \right)   & \displaystyle\mathop{\longrightarrow}^{(d)}_{\ell \rightarrow \infty} & \left( {\bf X}^{(1-\theta)}_{\theta} \left(  \frac{\mathsf{c}_{\mathbf{q}}}{\mathsf{a}_{\mathbf{q}}} \cdot t \right)  : t \geq 0 \right)  \textrm{ under } \mathbb{P}^{-}_{1}\notag
 \end{eqnarray}
 in the space of c\`adl\`ag process taking values in $\ell^{\downarrow}_{\theta+3/2}$ equipped with the Skorokhod $J_{1}$ topology. 
 \end{theorem}

Let us first explain why we need to restrict to the case $\theta \in (1,3/2)$ for this theorem. Let $B^{(\ell),\dagger}$ be the dual of a random map distributed according to $ \mathbb{P}^{(\ell)}$. The second convergence above as well as the results of \cite{BCgrowth} suggest that the typical distances in $ B^{(\ell), \dagger}$ are of order $ \ell^{\theta-1}$ when $\theta \in (1, 3/2)$. However, in the dense case $\theta \in (1/2,1)$ the results of \cite{BCgrowth} suggest that this typical distance should be of order $\log(\ell)$. This logarithmic scaling prevents us from expecting a non-trivial self-similar object in the limit and this is why we restrict our attention to the dilute case $\theta \in (1,3/2)$.

\begin{remark} By \eqref{eq:volumeconv}, under $\mathbb{P}^{(\ell)}$ the convergence $ \frac{1}{\ell^{\theta+1/2}} \cdot |\map| \rightarrow    \mathsf{b}_{ \mathbf{q}} \cdot \mathcal{M}_{\theta}^{-}(\infty)$ holds in distribution, where  $ \mathcal{M}_{\theta}^{-}(\infty)$ is the intrinsic area of $ \mathbf{X}_{\theta}^{(1-\theta)}$ under $ \mathbb{P}_{1}$. This convergence actually holds jointly with \eqref{eq:cv2} but do not prove it here to keep the length of the paper reasonable.
\end{remark}

The proof of Theorem \ref{thm:GF} follows the same lines as that of \cite{BCK} and we will only sketch its proof. We first introduce the peeling algorithm that we use to study the geometric structure of the dual of $ \mathbf{q}$-Boltzmann planar maps, and which is adapted from \cite{BCgrowth}.

Let $ \map$ be a  rooted bipartite planar map. Recall that a peeling exploration starts with $ \mathfrak{e}_{0}$, the map consisting of the root face of $\map$ (seen as a simple cycle). Inductively suppose that at step $i \geq 0$, the following hypothesis is satisfied:
\begin{center}
\begin{minipage}{14cm} $(H)$: There exists an integer $h \geq 0$ such all the faces adjacent to the holes of $ \mathfrak{e}_{i}$ are at height $h$ or $h+1$ in $\map$.  Suppose furthermore that the faces adjacent to a same hole in $ \mathfrak{e}_{i}$ and which are at height $h$ form a connected part on the boundary of that hole.
\end{minipage}
\end{center}
We will take the height of an edge of the boundary of a hole of $ \mathfrak{e}_{i}$ to mean the height of its incident face.  If $(H)$ is satisfied by $ \mathfrak{e}_{i}$ the next edge to peel $ \mathcal{L}^\dagger( \mathfrak{e}_{i} )$ is chosen as follows:
\begin{itemize}
\item If all edges on the boundaries of the holes of $ \mathfrak{e}_{i}$ are at height $h$ then $ \mathcal{L}^\dagger( \mathfrak{e}_{i})$ is a deterministic edge on the boundary of one of its hole,
\item Otherwise $ \mathcal{L}^\dagger( \mathfrak{e}_{i})$ is a deterministic edge at height $h$ such that the next edge in clockwise order around its hole is at height $h+1$.
\end{itemize}

\begin{figure}[!h]
 \begin{center}
 \includegraphics[width=0.7\linewidth]{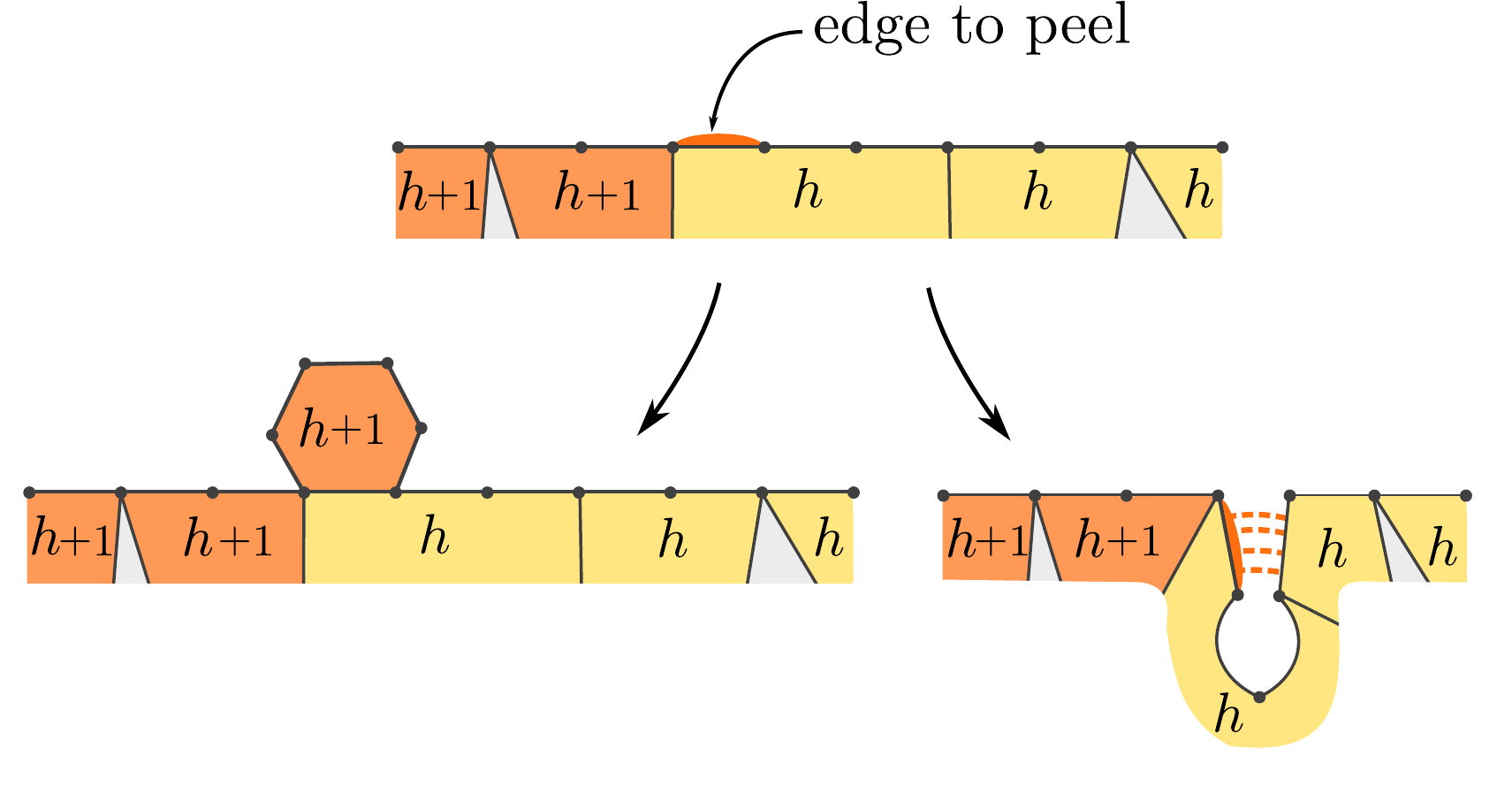}
 \caption{Illustration of the peeling using algorithm $ \mathcal{L}^\dagger$.}
 \end{center}
 \end{figure}

It is easy to check by induction that if one uses the above algorithm starting at step $i={0}$ to peel the edges of $ \map$, then for every $i \geq 0$ the explored map $ \mathfrak{e}_{i}$  satisfies the hypothesis $(H)$ and one can indeed define the peeling exploration of $ \map$ using the algorithm $ \mathcal{L}^\dagger$.  Let us give a geometric interpretation of this peeling exploration. We denote by $ \mathsf{H}( \mathfrak{e}_{i})$ the minimal height in $ \map$ of a face adjacent to one of its holes and let $\theta_{r} = \inf\{ i \geq 0 : \mathsf{H}( \mathfrak{e}_{i}) = r\}$ for $r \geq 0$. We easily prove by induction on $r\geq 0$ that:
$$ \mathfrak{e}_{\theta_{r}} = \mathrm{Ball}^\dagger_{r}( \map).$$

\begin{proof}[Sketch of the proof of  Theorem \ref{thm:GF}.] Let us focus first on the second convergence which is the extension of the convergence proved in \cite{BCK}. We sample a random planar map distributed according to $ \mathbb{P}^{(\ell)}$ and explore its dual metric structure using the above peeling exploration. As in \cite[Sec.~2.6]{BCK}, during this exploration we first track the evolution of the locally largest cycle $ \mathscr{C}_{*}$ starting with the boundary of the map and we freeze all the other cycles when created. By Proposition \ref{prop:peelinggenerallaw} we already know that the scaling limit for the half-perimeter $P_{*}$ of this cycle as a function of the number of peeling steps is given by the process $X_{\theta}^{(-\theta)}$. Our first goal is to prove that the scaling limit of $P_{*}$ in the \emph{height parameter} is given by the process $X_{\theta}^{(1-\theta)}$. This is proved as in \cite[Proposition 12 and Lemma 13]{BCK} using the work \cite{BCgrowth} instead of \cite{CLGpeeling} as the key input. Specifically, if $\tilde{P}_{*}(h)$ is the half-perimeter of the locally largest cycle (that starts with the boundary of the map) when height $h\geq 0$ is reached then under $ \mathbb{P}^{(\ell)}$ we have 
 \begin{eqnarray} \label{eq:cycleheight} \left(\frac{1}{\ell} \tilde{P}_{*}( \lfloor \ell^{\theta-1} \cdot t \rfloor )\right)_{t \geq 0} \xrightarrow[\ell\to\infty]{(d)} \left(X_{\theta}^{(1-\theta)}(  \mathsf{c}_{ \mathbf{q}} t /  \mathsf{a}_{ \mathbf{q}})\right)_{t \geq 0},  \end{eqnarray} in distribution for the Skorokhod $J_{1}$ topology. Once this is granted we perform an exploration with a cutoff $ \varepsilon>0$ exactly as in \cite[Sec.~3.3]{BCK} which consists roughly speaking in exploring the map with the above peeling algorithm but freezing those cycles whose half-perimeters drop below $ \varepsilon \ell$. We then use the Markov property of the underlying random map and  \eqref{eq:cycleheight} to deduce that after performing this cutoff we have convergence of the non-frozen cycles at heights towards the cutoff version of the growth-fragmentation process $ \mathbf{X}_{\theta}^{(1-\theta)}$, see \cite[Corollary 18]{BCK}. The point is that to perform these explorations with cutoff we only need to follow a tight number of locally largest cycles. Finally, we need to control, in the $\ell_{\theta+3/2}$-sense, the error made by the cutoff procedure both on the planar map side (see \cite[Proposition 15]{BCK}) and on the growth-fragmentation side (see \cite[Lemma 22]{BCK}). Here again, the proofs are easily adapted from \cite{BCK} to our case using the discrete and continuous martingales developed in this text rather than those used in \cite{BCK}. 
 
 For the first and third convergence we proceed similarly except that instead of first following the locally largest cycle that starts with the boundary of the map, we first follow the distinguished cycles $ \mathscr{C}_{\infty}$ in the first case and $ \mathscr{C}_{\bullet}$ in the third case. We can then establish the analogs of \eqref{eq:cycleheight} involving respectively the processes $Y^+$ and $Y^-$ associated with $ \mathbf{X}_{\theta}^{(1-\theta)}$ using similar arguments as above together with Proposition \ref{prop:scalingll}. The rest of the proof goes through: we fix $t_{0}>0$ and perform a cut-off exploration to prove that we have convergence of the non-frozen cycles at heights smaller than $ \ell^{\theta-1} t_{0}$ towards the cutoff version of the growth-fragmentation process $ (\mathbf{X}_{\theta}^{(1-\theta)}(s) : 0 \leq s \leq t_{0})$ under $ \mathbb{P}_{1}^+$ and $ \mathbb{P}_{1}^{-}$ respectively. The $\ell^{\theta+3/2}$-control of the error made is then transferred from the previous case using absolute continuity relations between the laws of $( \mathbf{L}(h) : 0 \leq h \leq \ell^{\theta-1} \cdot t_{0})$ under $ \mathbb{P}^{(\ell)},  \mathbb{P}^{(\ell)}_{\bullet}$ and $ \mathbb{P}^{(\ell)}_{\infty}$ and similarly between the laws of $( \mathbf{X}_{\theta}^{(1- \theta)}(s) : 0 \leq s \leq t_{0})$ under $ \mathbb{P}_{1}, \mathbb{P}_{1}^{+}$ and $ \mathbb{P}_{1}^-$. 
 \end{proof}
 
 \bibliographystyle{siam}
 {\small

}
\end{document}